\documentclass[10pt,a4paper]{amsart} 

\usepackage[utf8]{inputenc}
\usepackage{verbatim}
\usepackage{lmodern}
\usepackage{textcomp}
\usepackage{amssymb,amscd,amsthm,amsrefs,amsxtra,mathtools}
\usepackage{amsmath}
\usepackage{dsfont,latexsym}
\usepackage{mathrsfs}
\usepackage{fancyhdr,enumerate}
\usepackage{bm}

\usepackage{soul}

\usepackage{color}
\definecolor{rltred}{rgb}{0.75,0,0}
\definecolor{rltgreen}{rgb}{0,0.5,0}
\definecolor{rltblue}{rgb}{0,0,0.75}
\usepackage[
colorlinks=true,
urlcolor=rltblue,       
filecolor=rltgreen,     
citecolor=rltgreen,     
linkcolor=rltred,       
]{hyperref} 

\newtheorem{theorem}{Theorem}[section]
\newtheorem{prop}[theorem]{Proposition}
\newtheorem{corol}[theorem]{Corollary}
\newtheorem{lemma}[theorem]{Lemma}
\newtheorem{defn}[theorem]{Definition}
\newtheorem{rem}[theorem]{Remark}
\newtheorem{example}[theorem]{Example}

\numberwithin{equation}{section}
\allowdisplaybreaks[4]

\def\loc{\operatorname{\rm loc}}

\def\XXint#1#2#3{{\setbox0=\hbox{$#1{#2#3}{\int}$}
		\vcenter{\hbox{$#2#3$}}\kern-.5\wd0}}

\newcommand{\Rn}{\mathbb{R}^n}
\newcommand{\R}{\mathbb{R}}
\newcommand{\N}{\mathbb{N}}

\newcommand{\snr}[1]{\lvert #1\rvert}

\newcommand{\Ec}{\mathcal{E}_K}
\newcommand{\Hc}{\mathcal{H}_K}
\newcommand{\Pc}{\mathcal{P}}
\newcommand{\Hbb}{\mathbb{H}^K}

\renewcommand{\Gamma}{\varGamma}
\renewcommand{\leqslant}{\leq}
\renewcommand{\geqslant}{\geq}
\renewcommand{\epsilon}{\varepsilon}

\begin{document}
	
\title[Fractional Allen-Cahn equations]{Heteroclinic connections \\
for fractional Allen-Cahn equations\\ with degenerate potentials}

\author[F.~De~Pas]{Francesco De Pas} \address{Francesco De Pas\\  Department of Mathematics and Statistics\\
	University of Western Australia,\\
	35 Stirling Highway, WA 6009 Crawley, Australia}
	\email{\href{mailto:francesco.depas@uwa.edu.au}{francesco.depas@uwa.edu.au}}
	
\author[S.~Dipierro]{Serena Dipierro}  \address{Serena Dipierro\\  Department of Mathematics and Statistics\\
	University of Western Australia,\\
	35 Stirling Highway, WA 6009 Crawley, Australia}
	\email{\href{mailto:serena.dipierro@uwa.edu.au}{serena.dipierro@uwa.edu.au}}
	
	\author[M.~Piccinini]{Mirco Piccinini} 
\address{Mirco Piccinini\\
	Dipartimento di Scienze Matematiche, Fisiche e Informatiche, Universit\`a di Parma\\
	Parco Area delle Scienze 53/a, Campus, 43124 Parma, Italy, \ and \ Dipartimento di Matematica, Universit\`a di Pisa, ~L.go~B.~Pontecorvo~5, 56127, Pisa, Italy}
\email{\href{mirco.piccinini@dm.unipi.it}{mirco.piccinini@dm.unipi.it}}

\author[E.~Valdinoci]{Enrico Valdinoci}  \address{Enrico Valdinoci\\Department of Mathematics and Statistics\\
	University of Western Australia,\\
	35 Stirling Highway, WA 6009 Crawley, Australia}
\email{\href{mailto:enrico.valdinoci@uwa.edu.au}{enrico.valdinoci@uwa.edu.au}}

	

\thanks{{\it Aknowledgements.} 
SD and EV are members of the Australian Mathematical Society.
FDP and EV are supported by the Australian Laureate Fellowship FL190100081 ``Minimal
surfaces, free boundaries and partial differential equations''.
SD is supported by the Australian Future Fellowship
FT230100333 ``New perspectives on nonlocal equations''.
MP is supported by INdAM project  ``Problemi non locali: teoria cinetica e non uniforme ellitticit\`a'', CUP\_E53C22001930001 and by the Project ``Local vs Nonlocal: mixed type operators and nonuniform ellipticity", CUP\_D91B21005370003.
Part of this work has been carried out during a visit of MP to the University
of Western Australia, which we thank for the warm hospitality.}

\begin{abstract}
	We investigate existence, uniqueness and asymptotic behavior of minimizers of a family of non-local energy functionals of the type
	\[
	\frac{1}{4}\iint_{\R^{2n}\setminus (\Rn \setminus \Omega)^2}\snr{u(x)-u(y)}^2 {K}(x-y) \,dx dy + \int_\Omega W(u(x)) \,dx.
	\]
Here,~$W$ is a possibly~\textit{degenerate} double well potential with a polynomial control on its second derivative near the wells. Also,~${K}$ belongs to a wide class of measurable kernels and is modeled on that of the fractional Laplacian.
\end{abstract}

\maketitle

\setcounter{equation}{0}\setcounter{theorem}{0}

\setcounter{tocdepth}{1} 
\begin{center}
	\begin{minipage}{11cm}
\footnotesize
		\tableofcontents
	\end{minipage}
\end{center}

\section{Introduction}
\subsection{Problem setting}
In this paper we deal with  the minimization problem of an energy functional related to phase transition phenomena
with long range particle interactions. Specifically, we are interested in functionals of the form
\begin{equation}\label{main_fun}
  \Ec(u;\Omega) := \Hc(u,\Omega)+\Pc(u,\Omega),
\end{equation}
where the non-local interaction term~$\Hc$ and the potential term~$\Pc$ are given, respectively, by
\begin{equation*}
\Hc(u,\Omega):=  \frac{1}{4}\iint_{\R^{2n}\setminus (\Rn \setminus \Omega)^2}\snr{u(x)-u(y)}^2 {K}(x-y) \, dx dy\,,
\end{equation*}
and
\begin{equation}\label{pot_term}
\Pc(u,\Omega) := \int_\Omega W(u(x)) \,dx.
\end{equation}

Here,~$K$ is a positive kernel modeled on that of the fractional Laplacian, while~$W$ is a double well potential, with wells at~$\pm1$. Also, differently from the classical literature on this topic,
the derivatives of~$W$, up to any integer order, are allowed to vanish at~$\pm1$.

More precisely,~${K}: \Rn \to [0,+\infty]$ is a measurable function satisfying
\begin{align}
K(x)= K(-x) \qquad \text{for a.~\!e.}~x\in \R^n\,, \label{krn_symm}\tag{K1}
\end{align}
and
\begin{align}
\frac{\lambda\mathds{1}_{B_{r_0}}(x)}{|x|^{n+2s}}\leq K(x) \leq \frac{\Lambda}{|x|^{n+2s}} \qquad \text{for a.~\!e.}~x \in \R^n, \label{main_ellipt}\tag{K2} \end{align}
for some~$s\in(0,1)$, $0 <\lambda \leq \Lambda$ and~$ r_0>0$.

Given a set~$A\subseteq \R^n$, here above and throughout the paper, we use the notation
$$ \mathds{1}_{A}(x):=\begin{cases}
1 &{\mbox{ if }} x\in A,\\
0 &{\mbox{ if }} x\in \Rn\setminus A.
\end{cases}
$$

In order to gain better decay results for the minimizers of~$\Ec$, we will sometimes ask that~$K$ satisfies
\begin{equation}\label{newbound}
\frac{\lambda}{|x|^{n+2s}}\leq K(x) \leq \frac{\Lambda}{|x|^{n+2s}} \quad \text{for a.~\!e.}~x \in \R^n~\text{and some}~s\in(0,1)~\text{and}~0 <\lambda \leq \Lambda.\tag{K3}
\end{equation}
We stress that condition~\eqref{newbound} is stronger than~\eqref{main_ellipt} as, for instance, prevents the kernel from having compact support.

Moreover, in some cases, we will also assume that~$K$ satisfies the following 
``slow oscillation'' assumption:
\begin{equation}
\limsup_{j \to +\infty}\left(\sup_{x \in \R^n \setminus \{0\}} \frac{K(\sigma_j x)}{K(x)} -1\right) \frac{1}{(1-\sigma_j)^{1-\epsilon}}\in [-\infty, 0], \quad\text{for any}~\sigma_j \nearrow 1~\text{and}~\epsilon \in (0,1) \ \label{nuovissima}\tag{K4}.
\end{equation}
We refer the reader to Appendix~\ref{kern_ex} for examples of kernels satisfying these hypotheses.

We notice here that~$\eqref{krn_symm}$ allows us to express the kinetic term~$\Hc$ as
\begin{eqnarray*}
\Hc(u,\Omega) &=&\frac{1}{4} \int_{\Omega} \int_{\Omega} | u(x)-u(y) |^2 K(x-y) \, dx \, dy \\
&& + \frac{1}{2} \int_{\Omega}  \int_{\R^n\setminus\Omega} |u(x)-u(y) |^2 K(x-y) \, dx \, dy.
\end{eqnarray*}

Regarding the potential~$W: \R \to [0,+\infty)$, we assume that\footnote{As customary, when writing~$C^\vartheta(\Omega)$, we suppose that, if~$\vartheta>1$, the notation~$u\in C^\vartheta(\Omega)$ means~$u\in C^{k,\theta}(\Omega)$ with~$k\in\mathbb{N}$, $\theta\in(0,1]$ and~$\vartheta=k+\theta$.} 
   \begin{align}
   &W\in C_{\rm loc}^{2,\vartheta}(\R) \ \mbox{for some~$\vartheta>0$ and} \  W(\pm1)=W\rq{}(\pm1)=0 \quad\mbox{and} \label{pot_reg} \tag{W1}
   \\
   &W(t)>0 \quad \mbox{for all } t \in (-1,1).\label{pot_zero}\tag{W2}
   \end{align}
Also, we assume that there exist~$C_2 \geq C_1 >0$, $C_4\ge C_3>0$,
$\xi \in (0,1)$, $\alpha \geq \beta \geq 2$ and~$\gamma \geq \delta \geq 2$ such that
   \begin{equation}\label{pot_deg}\tag{W3}
      \begin{cases}
       C_1 (1+t)^{\alpha-2} \leq W\rq{}\rq{}(t) \leq C_2 (1+t)^{\beta-2}& \text{for } t \in (-1,-1+\xi] \\ \mbox{and} \\
        C_3 (1-t)^{\gamma-2} \leq W\rq{}\rq{}(t) \leq C_4 (1-t)^{\delta-2} & \text{for } t  \in [1-\xi,1). 
   \end{cases}   
\end{equation}
   
In addition, some of the results will ask for the potential to be symmetric, i.e.
\begin{equation}\label{pot_symm}\tag{W4}
W(t)= W(-t) \qquad\mbox{for any } t \in [-1,1].
\end{equation}

We stress that condition~\eqref{pot_deg} is very general and, for instance, allows~$W$ to be~\textit{degenerate}\footnote{Throughout this work, given a double well potential~$V$ with wells at~$a$ and~$b$, we call it~\textit{non-degenerate} if~$V''(a)>0$ and~$V''(b)>0$. If this condition is not satisfied, we call it~\textit{degenerate}} and also to present an oscillatory behavior near the wells.

\begin{rem}
{\rm
The energy functional~$\Ec$ in~\eqref{main_fun} can be compared to the energy functional~$\mathcal{F}$ considered in~\cite{PSV13}, expressed as
\begin{equation}\label{psv_en}
\mathcal{F}(u,\Omega) := \mathscr{H}(u, \Omega) + \Pc(u, \Omega),
\end{equation}
where the potential term~$ \Pc$ coincides with~\eqref{pot_term}, while the kinetic one is
\begin{equation*}
\mathcal{H}(u, \Omega):= \frac{1}{2} \int_{\Omega} \int_{\Omega} \frac{| u(x)-u(y) |^2}{|x-y|^{n+2s}} \, dx \, dy + \int_{\Omega}  \int_{\R^n\setminus\Omega}\frac{| u(x)-u(y) |^2}{|x-y|^{n+2s}} \, dx \, dy.
\end{equation*}
We mention that, if we set~$K$ to be the kernel of the fractional Laplacian, namely~$K(x)=|x|^{-n-2s}$, then~$\mathcal{H}$ and~$\Hc$ coincide, up to a constant. In particular,~\eqref{main_ellipt} yields
\begin{equation}\label{en_relation}
\Ec(u,\Omega)\leq \max \left\{ \frac{\Lambda}{2},1\right\} \mathcal{F}(u,\Omega).
\end{equation}
}
\end{rem}
   
In the present paper, we address existence, uniqueness and asymptotic behavior of minimizers of~$\Ec$. As a matter of fact, in~\cite[Remark 2.4]{CP16}, the construction of such minimizers is left as an open problem and, to the best of the authors\rq{} knowledge, constitutes a new result in literature. 

The main novelty of this study is to be found in condition~\eqref{pot_deg} which provides a polynomial control on the growth of~$W\rq{}\rq{}$ near~$\pm1$. Thus, the potential might present an oscillatory nature
and cause different decay rates of minimizers near the wells.
In particular, \eqref{pot_deg} encodes both the cases of a symmetric and of a non-symmetric potential, which will be considered separately (see Theorems~\ref{main_thm} and~\ref{main_thm_symm} in Section~\ref{S2}).

Moreover, in addition to the existing literature, we show that the decay estimates provided for the minimizers and their derivatives are optimal (see Remark~\ref{optimality} in Section~\ref{S2}).

In the local case, degenerate potentials have been
considered in~\cite{farina}, where density estimates are obtained
actually in the more general setting of~$p$-Laplace equations and quasiminima of the energy.

We remark that functionals as in~\eqref{main_fun} constitute a non-scaled Ginzburg-Landau-type energy, in which the kinetic term~$\Hc$ is given by some non-local integrals, in place of the classical Dirichlet integral. Models of this form have attracted a great deal of attention, due to their capability to capture long range interactions between particles. They naturally arise, for instance, when dealing with phase transition phenomena involving non-local tension effects (see e.g~\cite{CozziValdNONLINEARITY, SV12,SV14}), or in the study of the Peierls-Nabarro model for crystal dislocation (see e.~\!g.~\cite{BV16, DPV15, DFV14, GM12}).

In particular, the problem of investigating the qualitative properties of
minimizers of such functionals is not new in literature. Indeed, since the breakthrough works of De Giorgi, Modica and Mortola, it is known that minimizers of the Ginzburg-Landau energy functional are deeply connected with minimal surfaces, leading to the famous conjecture by De Giorgi on the symmetry of monotone entire solutions of the Allen-Cahn equation (see~\cite{Fari2009, FV13, BV16, TRUDI} for a more comprehensive treatment). 
Also, we recall that in~\cite{SV12}
the authors study the relation between solutions to the fractional Allen-Cahn equation and (non) local minimal surfaces, obtaining the fractional counterpart of the~$\Gamma$-convergence result by Modica and Mortola.
Moreover, we refer the reader to~\cite{AB94, PV20} for results on the~$\Gamma$-convergence of non-local Allen-Cahn energy functionals in one dimension, respectively for the cases~$s =1/2$ and~$s \in (0,1)$.

Specifically, the results that we present here are in the direction
of~\cite{PSV13, CP16}.
In~\cite{PSV13}, indeed, the authors analize the minimizers of the energy~$\mathcal{F}$ in~\eqref{psv_en}, considering the related Euler-Lagrange equation   
\begin{equation}\label{all_ca_frl}
L_s u = W\rq{}(u),
\end{equation}
where~$W$ is a~\textit{non-degenerate} potential and~$L_s u$ stands for the fractional Laplacian (for which we mantain the notation used in~\cite{DPV15})
$$
L_s u(x) := PV_x \int_{\R^n}\frac{u(y)-u(x)}{ |x-y|^{n+2s}} \, dy.
$$
We recall that equation~\eqref{all_ca_frl} is often credited as a non-local analogue of the so-called \textit{(elliptic) Allen–Cahn equation} --- the classical, local one being just~\eqref{all_ca_frl} with~$s=1$, formally.

The paper~\cite{CP16} extends the results of~\cite{PSV13} to the case of non-local operators with general kernels satisfying~\eqref{krn_symm} and~\eqref{main_ellipt}. Hence, \eqref{all_ca_frl} becomes
\begin{equation}\label{eurl_lag_eq}
L_K u = W\rq{}(u),
\end{equation}
where~$W$ is a~\textit{non-degenerate} and even potential, while
\begin{equation}\label{main_op}
L_K u(x) := PV_x \int_{\R^n}(u(y)-u(x)) K(x-y) \, dy.
\end{equation}

See also~\cite{CS15}, where extension techniques are used to deal with
existence, uniqueness and qualitative properties of solutions of~\eqref{all_ca_frl}.

The present paper deals with the same operator as in~\eqref{main_op}, nevertheless in our framework the potential will be allowed to be~\textit{degenerate} and non-symmetric.

\begin{rem}
{\rm
Being~\eqref{krn_symm} in force, ${L}_K$ can be represented as a non-singular integral. Indeed, it holds that
\begin{equation*}
{L}_K u(x) = \frac{1}{2} \int_{\R^n} \delta u(x,z) K(z) \, dz,
\end{equation*}
where~$\delta u (x,z)$ is the second order increment
\begin{equation*}
\delta u (x,z) := u(x+z)-u(x-z)-2u(x).
\end{equation*}
Note that if~${K}(x)=|x |^{-n-2s}$, then~${L}_K$ boils down to the classical fractional Laplacian (for which we mantain the notation used in~\cite{DPV15})
\begin{equation*}
L_s u (x)= \frac{1}{2} \int_{\R^n} \frac{\delta u(x,z)}{|z |^{n+2s}} \, dz.
\end{equation*}
}
\end{rem}

\subsection{Minimizers of~$\Ec$ with general potentials}\label{S2}
 
In order to state the main results of this paper, some further notation is needed.
  \begin{defn}\label{defini}
  	Let~$\Omega$ be a bounded domain of~$\R^n$. A measurable function~$u : \R^n \to \R$ is a  {\rm local minimizer} of~$\Ec$ in~$\Omega$ if~$\Ec(u;\Omega)<+\infty$ and
  	$$
  	\Ec(u;\Omega) \leqslant\Ec(u+\phi;\Omega),
  	$$
  	for any~$\phi\in C^\infty_0(\Omega)$.
  	
  	Moreover, we say that a measurable function~$u : \Rn \to \R$ is a  {\rm class~A minimizer} of~$\Ec$ if it is a local minimizer in any bounded domain~$\Omega\subset \Rn$.
  \end{defn}
 
We stress that 
the introduction of the notion of
class~A minimizers is due to the fact that the energy~$\Ec$
may in principle diverge when evaluated in unbounded domains. Therefore, to keep track of its growth, one introduces the following
``renormalized'' energy
   \begin{equation}\label{G_star}
       \mathcal{G}(u):= \limsup_{\rho \to+ \infty }\frac{\Ec(u;[-\rho,\rho])}{\Psi_s(\rho)},
   \end{equation}
   where the function~$\Psi_s$ is given by
\begin{equation}\label{fnc_psi}
    \Psi_s(\rho):=
    \begin{cases}
        \rho^{1-2s} \, & \mbox{ if }  s \in (0,1/2),\\
        \log\rho \, & \mbox{ if }  s=1/2,\\
         1 \, & \mbox{ if } s \in (1/2,1).
    \end{cases}
   \end{equation}

Furthermore, as the next remark points out, this concept of local minimization is consistent with respect to set inclusion.  
\begin{rem}[Remark 2.2 in~\cite{CP16}]\label{remmarco}
{\rm
Let~$\Omega\rq{} \subset \Omega$ be two given domains of~$\R^n$. Then a local minimizer~$u$ for~$\Ec$ in~$\Omega$ is also a local minimizer\footnote{This fact
follows from the following inclusion
\begin{equation}\label{maius}
 \R^{2n} \setminus \left( \R^n \setminus \Omega\rq{}\right)^2 \subset \R^{2n} \setminus \left( \R^n \setminus \Omega \right)^2
\end{equation}
and from the fact that, if~$u$ and~$v$ coincide in~$\R^n\setminus \Omega\rq{}$, then
\[ |u(x)-u(y)|^2 = |v(x)-v(y)|^2 \quad\mbox{for any} \ (x,y) \in \left( \R^n \setminus \Omega\rq{} \right)^2. \]
We refer the interested reader to~\cite[Remark 1.2]{CozziValdNONLINEARITY} for a more detailed explanation.

In particular, \eqref{maius} implies that the energy~$\Ec(u,\cdot)$ is non-decreasing with respect to set inclusion.} of~$\Ec$ in~$\Omega\rq{}$.
}
\end{rem}  
  
We are now in the position to state our main result, dealing with existence, uniqueness and decay properties of class~A minimizers for~$\Ec$ in the class of admissible functions
  \begin{equation}\label{defmathcalixs00}
      \mathcal{X} := \Big\{f \in L^1_{\loc}(\R) \;{\mbox{ s. ~\!t. }}\;
      	\lim_{x\to  \pm \infty}f(x) =\pm 1 \Big\}.
  \end{equation}

The following theorem guarantees, in dimension~$1$, the existence and uniqueness (up to translations) of a class~A minimizer for~$\Ec$ and lists some of its qualitative properties. Furthermore, it proves that such minimizers are the only non-decreasing solutions of~\eqref{eurl_lag_eq}. 

 \begin{theorem}\label{main_thm}
Let~$n=1$. Let~\eqref{krn_symm}, \eqref{main_ellipt}, \eqref{nuovissima}, \eqref{pot_reg}, \eqref{pot_zero} and~\eqref{pot_deg} hold true.
     	         
	         Then, in the family~$\mathcal{X}$ of admissible functions, there exists a unique (up to translations) nontrivial class~A minimizer~$u^{(0)}$ of~$\Ec$.
	         
	         Moreover, $u^{(0)}$         
	         is strictly increasing. 
	         
	         Also, $u^{(0)} \in C^{1+2s +{\theta}}(\R)$ for some~$\theta \in (0,1)$ and there exist~$\widetilde{C}> 0$ and~$R>0$ such that
          \begin{equation}\label{asymp_decay}
          	\begin{split}
          1+u^{(0)}(x) \leq \widetilde{C} |x|^{-\frac{2s}{\alpha-1}} &\quad\mbox{if } x\leq -R,\\   
         1 - u^{(0)}(x) \leq \widetilde{C}|x|^{-\frac{2s}{\gamma-1}} &\quad\mbox{if } x\geq R.
           \end{split}
          \end{equation}

Furthermore, up to translations, $u^{(0)}$ is the only increasing solution to
              \begin{equation}\label{ru3qio2ruebn68594}
                L_K u = W'(u) \quad \text{in } \R\,,
            \end{equation}  
            in the family of admissible functions~$\mathcal{X}$ and
\begin{equation*}
\mathcal{G}(u^{(0)}) <+\infty. 
\end{equation*}

In addition, if~$K$ satisfies~\eqref{newbound}, then there exists~$\widehat{C}\in (0, \widetilde{C}]$ such that
         \begin{equation}\label{asymp_decay_lowbound}
          	\begin{split}
          1+u^{(0)}(x) \geq \widehat{C} |x|^{-\frac{2s(\alpha-\beta+1)}{\alpha-1}} &\quad\mbox{if } x\leq -R,\\   
         1 - u^{(0)}(x) \geq \widehat{C}|x|^{-\frac{2s(\gamma-\delta+1)}{\gamma-1}} &\quad\mbox{if } x\geq R\,,
           \end{split}
          \end{equation}
and
\begin{equation}\label{398r7gree}
\begin{split}
 (u^{(0)}(x))' \geq \widehat{C} |x|^{-\left(1+\frac{2s(\alpha-\beta+1)}{\alpha-1}\right)} & \quad\mbox{if } x\leq -R,\\
(u^{(0)}(x))' \geq  \widehat{C}|x|^{-\left( 1+\frac{2s(\gamma-\delta+1)}{\gamma-1}\right)}  &\quad\mbox{if } x\geq R.
\end{split}
\end{equation}

Moreover, if~$K$ satisfies~\eqref{newbound} and
\begin{equation}\label{ricdifar}
\max \left\{ (\alpha-2)(\alpha-\beta),  (\gamma-2)(\gamma-\delta) \right\} <1,
\end{equation}
then
 \begin{equation}\label{eq:asymp-derivata}
\begin{split}
 (u^{(0)}(x))' \leq \widetilde{C}|x|^{-\left(1+\frac{2s(1-(\alpha-2)(\alpha-\beta))}{\alpha-1}\right)}  &\quad\mbox{if } x\leq -R ,\\
 (u^{(0)}(x))' \leq \widetilde{C}|x|^{-\left(1+\frac{2s(1-(\gamma-2)(\gamma-\delta))}{\gamma-1}\right)} &\quad\mbox{if } x\geq R.
\end{split}
 \end{equation}   
                \end{theorem}
 
The next theorem can be seen as the dual of the previous one when~\eqref{pot_symm} is in force. Hypotheses~\eqref{pot_symm} allows us to drop the assumption~\eqref{nuovissima} on~$K$ and, moreover, to retrieve odd minimizers of~$\Ec$. We recall  that, in this case, $\alpha=\gamma$ and~$\beta= \delta$ in~\eqref{pot_deg}.
              
  \begin{theorem}\label{main_thm_symm}
Let~$n=1$. Let~\eqref{krn_symm}, \eqref{main_ellipt}, \eqref{pot_reg}, \eqref{pot_zero}, \eqref{pot_deg} and~\eqref{pot_symm} hold.	         
	         Then, in the family~$\mathcal{X}$ of admissible functions, there exists a unique (up to translations) nontrivial class~A minimizer~$u^{(0)}$ of~$\Ec$.
	         
	         Moreover, $u^{(0)}$ is strictly increasing and odd.
	         
	         Also, $u^{(0)} \in C^{1+2s +{\theta}}(\R)$ for some~$\theta \in (0,1)$ and there exist~$\widetilde{C}> 0$ and~$R>0$ such that
          \begin{equation}\label{asymp_decay_symm}
          | u^{(0)}(x)- {\rm sign} (x) | \leq \widetilde{C} |x|^{-\frac{2s}{\alpha-1}} \quad\mbox{if } |x | \geq R.
          \end{equation}
             
Furthermore, up to translations, $u^{(0)}$ is the only non-decreasing solution to
              \begin{equation*}
                L_K u = W'(u) \quad \text{in}~\R,
            \end{equation*}  
            in the family of admissible functions~$\mathcal{X}$ and
\begin{equation*}
\mathcal{G}(u^{(0)}) <+\infty. 
\end{equation*}

In addition, if~$K$ satisfies~\eqref{newbound}, then there exists~$\widehat{C}\in (0,\widetilde{C}]$ such that
          \begin{equation}\label{asymp_decay_symm_lowbound}
          | u^{(0)}(x)- {\rm sign} (x) | \geq \widehat{C} |x|^{-\frac{2s(\alpha-\beta+1)}{\alpha-1}} \quad\mbox{if } | x | \geq R.
          \end{equation}
      and
          \begin{equation}\label{97e382d}
     (u^{(0)}(x))'   \geq    \widehat{C}|x|^{-\left(1+\frac{2s(\alpha-\beta+1)}{\alpha-1}\right)} \quad\mbox{if } | x | \geq R .
          \end{equation}

Moreover, if~$K$ satisfies~\eqref{newbound} and
\begin{equation}\label{gufiodghfdjeryueyu4i}
(\alpha-2)(\alpha-\beta)<1,\end{equation} then
\begin{equation}\label{98fgvrjfbgvvfoi88}
(u^{(0)}(x))' \leq \widetilde{C}|x|^{-\left(1+\frac{2s(1-(\alpha-2)(\alpha-\beta))}{\alpha-1}\right)} \quad\mbox{if } | x | \geq R.
\end{equation}
                \end{theorem}
                
\begin{rem}\label{optimality}
{\rm Theorems~\ref{main_thm} and~\ref{main_thm_symm} entail that if~\eqref{newbound} and~\eqref{ricdifar} 
(which boils down to~\eqref{gufiodghfdjeryueyu4i} in Theorem~\ref{main_thm_symm})
are satisfied and~$\alpha=\beta$ and~$\gamma=\delta$, then the decay estimates for~$u^{(0)}$ and~$(u^{(0)})\rq{}$ are optimal. Notice that, thanks to the full generality of our framework, this consideration applies to the minimizers considered in~\cite[Theorem~2]{PSV13} and~\cite[Theorem~1]{CP16}.             
}
\end{rem}

We mention that some partial contributions for degenerate potential and general kernels have been provided in the earlier works~\cite{BB98,BBB98} and in~\cite{DT24}. In these papers, the authors address the existence and monotonicity of minimizers properties for energy functionals of the form~$\mathcal{E}_K$. However, in their full generality, our results for degenerate potentials are new even in the case of the fractional Laplacian.
On a similar note, our results for general kernels are new even in the case of non-degenerate potentials,
since we do not necessarily
assume that the potential is even. 
Furthermore, our results are new even when the potentials
are degenerate but their second derivative presents
a specific behavior at the potential minima (that is, when~$\alpha=\beta$ and~$\gamma=\delta$ in~\eqref{pot_deg}).
In this spirit, we believe that our general setting
is also helpful to provide a unified approach to problems in which the structures of the interaction and potential energies
possess different features which can potentially influence each other. 
 
We also point out that extending the results in~\cite{PSV13,CP16} to a~\textit{degenerate} and non-symmetric potential asks for a careful approach. More specifically, these changes in the hypotheses lead to significant modifications in the structure of the problem itself.
For instance, we point out that
hypotheses~\eqref{pot_symm} plays a major role in order to
prove the existence of particular class~A minimizers of~$\Ec$. Indeed, an even potential gives rise to an odd minimizer in bounded intervals (see Proposition~\ref{symm} below) and this allows, in Section~\ref{main_thm_proof}, to successfully perform a passage to the limit. Without this condition, we have to add to the more standard hypotheses~\eqref{krn_symm} and~\eqref{main_ellipt}
an \lq\lq{}upper semicontinuity\rq\rq{} property for~$K$, i.e.~\eqref{nuovissima} which is not present in~\cite{CP16}.
 
Also, whenever~$W\rq{}\rq{}(\pm1)>0$, decay estimates as in~\eqref{asymp_decay} can be inferred relying on barriers like the ones in~\cite[Lemma~4.1]{CP16} (that are inspired by the barriers in~\cite[Lemma~3.1]{SV14} and adapt them to the case of
more general kernels), that are well established in literature. Nevertheless, the general framework that
we work in asks for more refined barriers, inspired by the ones built in the recent paper~\cite{giovanni}.

Moreover, we stress that the estimates for the derivative of~$u^{(0)}$ cannot rely on a ``linearization argument'', as done in~\cite{CP16, PSV13},
since in our general setting~$W''$ may degenerate at the wells. Our strategy requires instead a suitable
``convexity property'', as provided by Lemma~\ref{thc_deg}, and estimates on the decay of~$u^{(0)}$ in order to employ
a suitable barrier argument.

Furthermore, we note that the decay estimates for the solution and for its derivatives,
as presented in~\eqref{asymp_decay}, \eqref{asymp_decay_lowbound}, \eqref{398r7gree},
\eqref{eq:asymp-derivata}, \eqref{asymp_decay_symm}, \eqref{asymp_decay_symm_lowbound}
\eqref{97e382d} and~\eqref{98fgvrjfbgvvfoi88},
showcase the interesting feature that the behavior of~$u^{(0)}$ and its derivatives at~$+\infty$ only depends on the structure of the potential at~$+1$ (and, likewise, the behavior of~$u^{(0)}$ and its derivatives at~$-\infty$ only depends on the structure of the potential at~$-1$):
this is not completely obvious, due to the non-local nature of the problem, and allows us to consider the case in which the potential
is degenerate at one well but non-degenerate at the other
(namely, $W''(+1)>0=W''(-1)$, or~$W''(-1)>0=W''(+1)$),
which is also a new feature with respect to the existing literature. 

\begin{rem}{\rm
As a natural direction for further investigation, one may consider the possibility of improving the decay estimates for minimizers in the presence of oscillatory potentials. 

For instance, the decay estimates in~\eqref{asymp_decay} depends solely on the parameters~$\alpha$ or~$\gamma$, which suggests that a sharp lower bound in~\eqref{asymp_decay_lowbound} should similarly depend only on the parameters~$\beta$ or~$\delta$.

Moreover, in~\cite[Proposition 6.4]{DPV15}, the authors refine the asymptotics of their minimizer to a higher order. Extending such results to our framework
would enable a deeper investigation on the evolution
of the atom dislocation function in the Peierls-Nabarro model
for crystals. We plan to address these lines of research in a forthcoming paper.}
\end{rem}

\subsection{Organization of the paper}
The rest of the paper is organized as follows. In Section~\ref{prel}, we address some regularity properties related to the fractional Allen-Cahn equation, providing the definitions of weak and pointwise solutions to~\eqref{eurl_lag_eq} and a suitable variational framework.

Section~\ref{usfbarrier} constructs a barrier used later on to prove the decay estimates in~\eqref{asymp_decay}. 

Section~\ref{barr_comp} contains three separate subsections: the first one investigating the~\textit{convexity property} of~$W$, the second one stating a strong comparison principle for~${L}_K$, while the third one providing some regularity results related to the fractional Allen-Cahn equation. 

Then, in Section~\ref{90fxd9916} we collect some asymptotic estimates for the operator~${L}_K$ in dimension~$n=1$.
Some preliminary results on 1-D minimizers are collected in Sections~\ref{min_on_int} and~\ref{usflemma}.

Finally, we prove Theorems~\ref{main_thm} and~\ref{main_thm_symm} in Section~\ref{main_thm_proof}.

     \section{Preliminaries}\label{prel}
In this section we 
provide the analytical setting that we work in and recall some
regularity results for solutions of 
     \begin{equation}\label{nonloc_eq}
         -{L}_K u =f \quad \mbox{in } \Omega\,,
     \end{equation}
    and of their associated Dirichlet problem
    \begin{equation}\label{dirichlet_pbm}
        \begin{cases}
            -{L}_K u = f & \mbox{in }  \Omega,\\
             u = g & \mbox{in } \mathscr{C}\Omega,
        \end{cases}
    \end{equation}
    where~$\Omega$ is a domain of~$\R^n$ and~$f$ and~$g$ are measurable functions. 
    
    We mention here that this section is meant to be self contained, does not contain any new results and only addresses the regularity theory that we need in the rest of the paper. A more exhaustive treatise of these topics can be found, e.~\!g. in~\cite{CS09,CS11,CP16,Ros16,RS14,Sil06}.
    
    \subsection{Analytical setting}\label{secanalytfra00}
     
In this subsection, $K$ is assumed to satisy~\eqref{krn_symm} and~\eqref{main_ellipt}. Given a domain~$\Omega \subset \Rn$, we consider the linear space
    $$
    \Hbb(\Omega) := \big\{u : \Rn \to \R \, \mbox{ measurable} \;{\mbox{ s. ~\!t. }}\; u|_{\Omega} \in L^2(\Omega) \, \mbox{ and }  \, [u]_{\Hc(\Omega)} < +\infty\big\},
    $$
    where 
     \begin{equation}\label{resis}
    [u]^2_{\Hbb(\Omega)}:=2 \Hc(u,\Omega) = \frac{1}{2}\iint_{\R^{2n}\setminus(\mathscr{C}\Omega)^2}\snr{u(x)-u(y)}^2 {K}(x-y) \, dx \, dy.
\end{equation}

Furthermore, we consider the space
    $$
    \Hbb_0 (\Omega) := \big\{u \in \Hbb(\Omega): u =0 \  \mbox{a.~\!e. in} \  \Rn \setminus \Omega\big\}.
    $$
If~$\Omega$ has continuous boundary (in the sense of~\cite[Definition~4]{fiscella}),
    we have that
    \begin{equation}\label{hfuroieytuhgvkus549860}
    \Hbb_0(\Omega) = \overline{C^\infty_0(\Omega)}^{\|\cdot\|_{\Hbb(\Omega)}},
    \end{equation}
    where
$$ \|u\|_{\Hbb(\Omega)}:=\|u\|_{L^2(\Omega)} + [u]_{\Hbb(\Omega)},$$
see~\cite[Remark~3.1]{CP16} and~\cite[Theorem~6]{fiscella}

We also set
    \begin{equation*}
    \langle u,v \rangle_{\Hbb(\Omega)}:= \frac{1}{2}\iint_{\R^{2n} \setminus (\R^n \setminus \Omega)^2} (u(x)-u(y))(v(x)-v(y)) {K}(x-y) \, dx \, dy.
    \end{equation*} 
    
\subsection{Regularity results}

    With this bit of notation, we can give the definition of weak solutions to~\eqref{nonloc_eq} and to the Dirichlet problem~\eqref{dirichlet_pbm}.

    \begin{defn}
    Given a bounded Lipschitz domain~$\Omega\subset\Rn$ and~$f \in L^2(\Omega)$, we say that~$u \in \Hbb(\Omega)$ is a weak solution to~\eqref{nonloc_eq} in~$\Omega$ if
    \begin{equation}\label{weak_formulation}
        \langle u, \phi \rangle_{\Hbb(\Omega)} = \langle f, \phi \rangle_{L^2(\Omega)} \quad\mbox{for any } \phi \in \Hbb_0(\Omega).
    \end{equation}
 
    For any~$g \in \Hbb(\Omega)$ we say that~$u \in \Hbb(\Omega)$ is a weak solution to~\eqref{dirichlet_pbm} if~\eqref{weak_formulation}
    is satisfied and~$u-g \in \Hbb_0(\Omega)$.

    If~$\Omega$ is unbounded, we say that~$u$ is a weak solution to~\eqref{nonloc_eq} if, for any Lipschitz subdomain~$\Omega' \Subset \Omega$, we have that~$u \in \Hbb(\Omega')$ and~\eqref{weak_formulation} is satisfied in~$\Omega'$.
    \end{defn}

Notice that, in light of~\eqref{hfuroieytuhgvkus549860}, one can
relax the assumption~$\phi \in \Hbb_0(\Omega)$ in~\eqref{weak_formulation}, by requiring instead that~$\phi \in C^\infty_0(\Omega)$.

    In order to strenghten the notion of solution under consideration, further notation is needed. In particular, we consider the ``tail space''
    $$
    L^1_{2s}(\Rn) := \big\{u : \Rn \to \R \ \ \mbox{measurable}\;{\mbox{ s. ~\!t. }}\; \|u\|_{L^1_{2s}(\Rn)}<+\infty\big\},
    $$
    where
    $$
    \|u\|_{L^1_{2s}(\Rn)}:= \int_{\Rn} \frac{\snr{u(x)}}{1+\snr{x}^{n+2s}}\,dx.
    $$ 
    Clearly~$L^\infty (\Rn) \subset L^1_{2s}(\Rn)$, and therefore the space~$L^1_{2s}(\Rn)$ is not empty. 
    
    \begin{defn}
     A function~$u \in L^1_{2s}(\Rn) \cap C^{2s+{\theta}}_{\rm loc}(\Omega)$, with~${\theta} >0$, is a pointwise solution to~\eqref{nonloc_eq} in~$\Omega$ if~$f$ is continuous and the equation is satisfied at any point~$x \in \Omega$. 

     Moreover, given~$g \in C^{2s+{\theta}}(\Rn\setminus\Omega)$, we say that~$u$ is a pointwise solution to~\eqref{dirichlet_pbm} if~\eqref{nonloc_eq} is satisfied in the pointwise sense and~$u\equiv g$ in~$\Rn \setminus \Omega$.
    \end{defn}
     
We mention here some regularity results for semilinear equations of the type~\eqref{nonloc_eq}.

     \begin{prop}[Proposition~3.12 in~\cite{CP16}]\label{reg_dirichlet_pbm}
         Let~$s \in (0,1)$ and let~$K$ satisfy~\eqref{krn_symm} and~\eqref{main_ellipt}. Let~$\Omega \subset \Rn$ be a bounded~$C^{1,1}$-domain, $W \in C^{1,{\vartheta}}_{\rm loc}(\R)$, for some~${\vartheta} \in (0,1)$, and~$g \in C^{2s+{\sigma}}(\Rn \setminus \Omega)$, for some~${\sigma} \in (0,2-2s)$.
         
         Let~$u \in \Hbb(\Omega) \cap L^\infty (\Rn)$ be a weak solution to
           \begin{equation*}
        \begin{cases}
            L_K u = W'(u) & \mbox{in } \, \Omega,\\
             u = g & \mbox{in } \, \mathscr{C}\Omega.
        \end{cases}
    \end{equation*}
    
    Then, $u \in C^{{\theta}}(\Rn)\cap C^{2s+{\theta}}_{\rm loc}(\Omega)$, for some~${\theta} \in (0,s)$, depending only on~$n$,~$s$,~$\lambda$,~$\Lambda$,~$\vartheta$ and~${\sigma}$.      
     \end{prop}
 
     \begin{prop}[Proposition~3.13 in~\cite{CP16}]\label{reg_entire_sol}
              Let~$s \in (0,1)$ and let~$K$ satisfy~\eqref{krn_symm} and~\eqref{main_ellipt}. Let~$W \in C^{2,{\vartheta}}_{\rm loc}(\R)$, for some~$\vartheta>0$, and~$u \in L^\infty(\Rn)$ be a weak solution to
$$
{L}_K u = W\rq{}(u) \quad\mbox{in} \ \R^n.
$$    

Then, $u \in C^{1+2s+{\theta}}(\Rn)$ for some~${\theta} >0$.
     \end{prop}
     
 \section{Construction of a useful barrier}\label{usfbarrier}
 
 In this section we construct an auxiliary function that will be used as a barrier to retrieve the decay estimates in~\eqref{asymp_decay} and~\eqref{asymp_decay_symm}.

Barriers like the one that we propose here have been already
introduced in the literature. But,
while in both~\cite[Lemma~8]{PSV13} and~\cite[Lemma~4.1]{CP16} the authors consider barriers that match the non-degenerancy condition of the potential~$W$ (see also~\cite{SV14}),
in our framework, $W$ can possibly decay faster at the wells~$\pm 1$ and so the need of a more refined approach. A similar barrier, constructed in the~$p$-Laplacian framework, can be found in~\cite{giovanni}.
 
      \begin{prop}\label{prop_bar}
     Let~${K}$ satisfy~\eqref{krn_symm} and~\eqref{main_ellipt}.
     
     Then, for any~$\zeta>0$ and~$m \geq2$, there exist two constants~$R_0>0$, depending on~$n$, $s$, $\Lambda$, $\zeta$ and~$m$, and~$C\geq 1$, depending on~$n$,~$s$,~$\Lambda$ and~$\zeta$, such that the following holds.
     
     For any~$R\geq R_0$ there exists a rotationally symmetric function
     \begin{equation}\label{barrier}
w \in C(\R^n,(-1,1])\,,
\end{equation}
such that
\begin{equation}\label{bar_outside_interval}
w \equiv 1 \quad\mbox{in} \  \R^n \setminus B_R\,,
\end{equation}
and
\begin{equation}\label{bar_sol}
|{L}_K w | \leq \zeta (1+w)^{m-1}\quad\mbox{in} \ B_R.
\end{equation}
Also, for any~$x \in B_R$,
\begin{equation}\label{bar_decay}
\frac{1}{C(1+R-|x |)^{\frac{2s}{m-1}}} \leq 1+w(x) \leq \frac{C}{(1+R-|x |)^{\frac{2s}{m-1}}}.
\end{equation}
\end{prop}

\begin{proof}
We set
\begin{equation}\label{err1}
q:=\frac{2}{m-1}\,,
\end{equation}
and take~$ r_1 \geq 2^{\frac{5}{qs}}$ and~$r \geq r_1$
(notice that these choices imply that~$r\ge 4\sqrt{2}>4$).

We also set~$\ell(t):= (r-t)^{-qs}$ for any~$t \in (0,r)$ and define
\begin{equation*}
\gamma_r := \left(\ell(r-1)-\ell\left(\frac{r}2\right)-
\ell\rq{}\left(\frac{r}2\right)\left(\frac{r}2-1\right)\right)^{-1}.
\end{equation*}
We notice that
\begin{eqnarray*}
\gamma_r^{-1}& = & 1 - 2^{qs}r^{-qs} - qs 2^{qs+1} r^{-(qs+1)}\left(\frac{r}{2}-1\right)\\
& \geq & 1-2^{qs}(1+qs)r_1^{-qs}\\
&\geq & 1- 12r_1^{-qs}\\ 
& > & \frac{1}{2}.
\end{eqnarray*}
Thus, $\gamma_r$ is well defined and
\begin{equation}\label{gammaerre}
1 <\gamma_r < 2.
\end{equation}

Now, we consider the function~$h:[0,+\infty) \to [0,1]$ defined by
\begin{equation}\label{tr4egfhdjgher5u4ty3nfdsu4y653}
h(t):= \begin{cases}
0 &\mbox{if} \ t \in [0,r/2),\\
\gamma_r (\ell(t)-\ell(r/2)-\ell\rq{}(r/2)(t-r/2))  &\mbox{if} \  t \in [r/2,r-1), \\
1  &\mbox{if} \ t \in[r-1,+\infty).
\end{cases}
\end{equation}
Obviously
$$
h(r/2)=0, \qquad h\rq{}(r/2)=0 \qquad\mbox{and}\qquad h(r-1)=1,
$$
so that~$h$ is Lipschitz continuous.

Moreover, thanks to~\eqref{gammaerre}, for any~$t \in (r/2,r-1)$,
\begin{equation}\label{pesderiv}
\begin{split}
|h\rq{}(t)| & =  \gamma_r |\ell\rq{}(t)-\ell\rq{}(r/2) | \leq qs \gamma_r (r-t)^{-(qs+1)}\leq 4 (r-t)^{-(qs+1)} \\\mbox{and}\qquad
| h\rq{}\rq{}(t) | &= \gamma_r |\ell\rq{}\rq{}(t) |= qs (qs+1)\gamma_r (r-t)^{-(qs+2)} \leq 12 (r-t)^{-(qs+2)}.
\end{split}
\end{equation}

On the other hand,
$$ \lim_{t\to (r-1)_-} h'(t)=\gamma_r qs\left(1-2^{qs+1}r^{-(qs+1)}\right)\neq 0 = \lim_{t\to (r-1)_+} h'(t).$$
In light of this, we see that~$h$ is not~$ C^{1,1}$
and we now want to modify it in the interval~$(r-2,r-1)$
in order to obtain a~$C^{1,1}$ function.

For this, we take~$\eta \in C^{\infty}([0,+\infty),[0,1])$ such that~$\eta =1$
in~$(0,r-7/4)$, $\eta =0$ in~$(r-5/4,+\infty)$, $\eta\rq{} \in (-4,0)$ and~$| \eta\rq{}\rq{} | \leq 32$. Also, we set
\[ g(t):= \eta(t) h(t)+1-\eta(t) \qquad\mbox{for any} \ t \in[ 0,+\infty). \]
We notice that~$g(t) \in [h(t),1]$ and~$g$ coincides with~$h$ outside the interval~$(r-2,r-1)$. 

Moreover, we have that~$g \in C^{1,1}([0,+\infty))$ and, by~\eqref{pesderiv}, we obtain that, for any~$t \in (r-2,r-1)$,
\begin{equation*}
 |g\rq{}(t) | \leq |\eta\rq{}(t)| (1-h(t)) + |h\rq{}(t)| \mathds{1}_{(r/2,r-5/4)}(t) \leq 8
\end{equation*}
and
\begin{equation*}
| g\rq{}\rq{}(t)| \leq | \eta \rq{}\rq{}(t)| (1-h(t)) +2 |\eta\rq{}(t) h\rq{}(t) |+ | h\rq{}\rq{}(t) | \mathds{1}_{(r/2,r-5/4)}(t) \leq 76.
\end{equation*}
In particular, this and~\eqref{pesderiv} imply that there exists a constant~$c_1>0$ (e.~\!g.~$c_1:=2^7$) such that,
for any~$t \in [0,r]$,
\begin{equation}\label{fandstrst}
| g\rq{}(t) | \leq c_1 \min \big\{ (r-t)^{-(qs+1)},1 \big\} \qquad\mbox{and}\qquad | g\rq{}\rq{}(t) | \leq c_1 \min\big\{(r-t)^{-(qs+2)},1 \big\}.
\end{equation}

In addition, we claim that, for any~$t \in [0,r]$,
\begin{equation}\label{tdj96}
\min\{(r-t)^{-qs},1 \} \leq g(t) + 12 r^{-qs} \leq 18  \min \{(r-t)^{-qs},1\} \quad\mbox{for any} \ t \in [0,r].
\end{equation}
In order to show this claim, we start proving the left-hand inequality. If~$t \in [r-1,r]$, then~$g(t)=1$ and the inequality is straightforward. If instead~$t \in [0,r/2)$, then
\begin{equation*}
g(t)+12r^{-qs}= 12 r^{-qs} \geq 4 r^{-qs} \geq 2^{qs} r^{-qs} \geq (r-t)^{-qs}. 
\end{equation*}
Also, if~$t \in [r/2,r-1)$, we use~\eqref{gammaerre} to see that
\begin{eqnarray*}
g(t) & \geq & h(t) \geq (r-t)^{-qs} - 2^{qs}r^{-qs}\left(1+qs2r^{-1}(t-r/2) \right)\\
&\geq & (r-t)^{-qs} - 2^{qs}r^{-qs}\left(1+qs2r^{-1}(r/2) \right)\\
& = & (r-t)^{-qs}  - 2^{qs}r^{-qs}(1+qs)\\
& \geq &  (r-t)^{-qs}  - 12 r^{-qs}.
\end{eqnarray*}

For what concerns the right-hand inequality in~\eqref{tdj96}, if~$t \in[r-1,r]$ then we observe that
$$
g(t)+12r^{-qs} \leq 1+12 r_1^{-qs} \leq 2.
$$
If instead~$t \in[0,r-1)$, then
\begin{eqnarray*}
g(t) + 12 r^{-qs}&\leq & h(t) \mathds{1}_{(0,r-5/4)}(t) + \mathds{1}_{(r-7/4,r-1)}(t) +12 (r-t)^{-qs}\\
&\leq & 2 (r-t)^{-qs} +\left(\frac{7}{4}\right)^{qs} (r-t)^{-qs} +12 (r-t)^{-qs}\\
&\leq & 18 (r-t)^{-qs},
\end{eqnarray*}
and this completes the proof of~\eqref{tdj96}.

Let now~$v(x):= g(|x |)$ for any~$x \in \R^n$. By the properties of~$g$, we see that~$v \in C^{1,1}(\R^n)$ is radially symmetric, radially non-decreasing and satisfies~$v=0$ in~$B_{r/2}$ and~$v=1$ in~$\R^n \setminus B_{r-1}$. Also, from~\eqref{tdj96}, we infer that, for any~$x \in B_r$,
\begin{equation}\label{rtrc444}
\min\{(r-|x |)^{-qs},1 \} \leq v(x) + 12 r^{-qs} \leq 18  \min \{(r-|x |)^{-qs},1\} .
\end{equation}
Thus, recalling the definition of~$q$ in~\eqref{err1}, it follows that, for any~$x \in B_r$,
\begin{equation}\label{rtgy96}
\min\{(r-|x |)^{-2s},1 \} \leq ( v(x) + 12 r^{-qs} )^{m-1}.
\end{equation}

We now claim that there exists a constant~$c_2>0$, depending on~$n$, such that, for any~$x \in B_r$,
\begin{equation}\label{rtbtg3}
\Vert D^2 v \Vert_{L^{\infty}(B_{(\max\{(r-| x |)/2,1)\}}(x)} \leq c_2 \max \left\{ \frac{r -|x |}{2},1\right\}^{-(qs+2)}.
\end{equation}
To prove this, we observe that, by~\eqref{fandstrst}, for any~$y \in B_r \setminus \overline{B_{r/2}}$,
\begin{eqnarray}\label{444f28f}
| D^2 v(y) | &\leq & n^2 \left( |g\rq{}\rq{}(|y|)|+ 2 \frac{|g\rq{}(|y |)|}{|y |}\right) \notag\\
&\leq & 2n^2 c_1 \left( \min\left\{ (r-| y |)^{-(qs+2)},1 \right\} +| y |^{-1}\min\left\{ (r-|y |)^{-(qs+1)},1 \right\} \right) \notag\\
&\leq & 4n^2 c_1 \min\left\{ (r-|y |)^{-(qs+2)},1 \right\} ,
\end{eqnarray}
where the last inequality exploits the fact that~$|y | > r - |y |$ and~$r \geq r_1 \geq 2$.
Since~$v=0$ in~$B_{r/2}$, we can say that~\eqref{444f28f}
holds true for any~$x\in B_r$.

Also, if~$x \in B_r$ with~$r-|x |>2$ and~$y \in B_{(r-| x |)/2}(x)$, then~$y \in B_r$ and
$$
|y | \leq |y -x | + |x | \leq \frac{r-|x |}{2}+|x | = \frac{r+|x |}{2},
$$
and therefore
$$
r-|y | \geq r - \frac{r + |x |}{2} = \frac{r-|x |}{2}.
$$
As a consequence, by~\eqref{444f28f} we obtain that, for any~$y  \in B_{(r-|x |)/2}(x)$,
$$
|D^2 v(y)| \leq 4n^2 c_1 (r-|y |)^{-(qs+2)} \leq 4n^2 c_1  \left(\frac{r -|x |}{2}\right)^{-(qs+2)} = 4 n^2 c_1  \max \left\{ \frac{r -|x |}{2},1\right\}^{-(qs+2)}.
$$

If instead~$r-|x |<2$, for any~$y \in B_1(x) \cap B_r$, it holds that
$$
|D^2 v(y)| \leq 4n^2 c_1=4n^2 c_1 \max\left\{\frac{r-|x |}{2},1 \right\}^{-(qs+2)},
$$
this completing the proof of~\eqref{rtbtg3}.

Now, for any~$\sigma>0$, we define the scaled kernel
$$
K_{\sigma}(z):= \sigma^{n+2s} K(\sigma z), \quad\mbox{for a.~\!e.} \ z \in \R^n.
$$
Notice that~$K_{\sigma}$ satisfies conditions~\eqref{krn_symm} and~\eqref{main_ellipt} with the same~$\lambda,\Lambda$ of~$K$ and with~$r_0/\sigma$ in place of~$r_0$. Then, we apply~\cite[Lemma~6.9]{cozzibarrier} with~$\rho := \max\{(r-|x |)/2,1\}$ and, by~\eqref{rtgy96} and~\eqref{rtbtg3}, we obtain that, for any~$x \in B_r$,
\begin{eqnarray}\label{peccxce}
| {L}_{K_{\sigma}}v (x) | & \leq & c_3 \bigg[ \Vert v \Vert_{L^{\infty}(\R)} \max \left\{\frac{r-|x |}{2},1 \right\}^{-2s}\notag\\
&& + \Vert D^2v\Vert_{L^{\infty}(B_{\max\{(r-|x |)/2,1\}}(x))} \max\left\{\frac{r-| x|}{2},1 \right\}^{2(1-s)}\bigg] \notag \\
&\leq & c_4  \min\left\{(r-|x |)^{-2s},1 \right\} \notag\\
&\leq & c_4 (v(x)+12r^{-qs})^{m-1},
\end{eqnarray}
for some~$c_3$, $c_4>0$ depending on~$n$, $s$ and~$\Lambda$.

We are now able to construct the function~$w$. We set
\begin{equation}\label{mnbvcxasdfghj098765werty}
R_0:= \left( \frac{c_4}{\zeta}\right)^{\frac{1}{2s}} r_1
\end{equation}
and we take~$R\geq R_0$. Also, let
\begin{equation}\label{mnbvcxasdfghj098765werty2}
r:= \frac{r_1}{R_0}R \qquad\mbox{and}\qquad \beta:= 24 r^{-qs}
\end{equation}
and notice that~$r\geq r_1$ and~$\beta \in (0,1)$. Then, we define
\begin{equation}\label{mnbvcxasdfghj098765werty23}
w(x):= (2-\beta) v\left(\frac{rx}{R}\right)+\beta -1.
\end{equation}
We observe that the function~$w$ inherits all the qualitative properties of~$v$. In particular, $w$ is of class~$C^{1,1}(\R^n)$, is radially symmetric and radially non-decreasing. In addition, $w= \beta-1$ in~$B_{R/2}$ and~$w=1$ in~$\R^n \setminus B_{[(r-1)R]/r}$. 

Now, we show that~$w$ satisfies~\eqref{bar_sol}. To do this, we exploit~\eqref{peccxce}
and we recall the definitions in~\eqref{mnbvcxasdfghj098765werty},
\eqref{mnbvcxasdfghj098765werty2} and~\eqref{mnbvcxasdfghj098765werty23} to see that, for any~$x \in B_R$,
\begin{eqnarray*}
|{L}_{K} w (x) | & = & (2-\beta) \left| {L}_K \left(v \left(\frac{rx}{R}\right) \right)\right|\\
& \leq & (2-\beta)^{m-1} \left(\frac{R}{r}\right)^{-2s} \left| {L}_{K_{R/r}} v \left(\frac{rx}{R}\right) \right| \\
&\leq & c_4  \left(\frac{R}{r}\right)^{-2s} \left((2-\beta) v \left(\frac{rx}{R}\right)+24 r^{-qs}\right)^{m-1}\\
&=& \zeta \left(1+w(x)\right)^{m-1},
\end{eqnarray*}
which establishes~\eqref{bar_sol}.

We now focus on the estimates in~\eqref{bar_decay}.
For this, we set
$$
c_5:=36 \max\left\{ \frac{c_4}{\zeta}, 1 \right\} \left(1+\left(\frac{\zeta}{c_4}\right)^{\frac{1}{2s}}\right)^2
$$
and we claim that, for any~$ x \in B_{[(r-1)R]/r}$,
\begin{equation}\label{5334}
 \min\left\{ \frac{c_4}{\zeta}, 1 \right\} (R+1-|x |)^{-qs} \leq 1+w(x) \leq c_5 (R+1-|x |)^{-qs}.
\end{equation}
To show this, we preliminarly notice that~$
x \in B_{[(r-1)R]/r}$ if and only if~$r|x|/{R} <r-1$, which is equivalent to
$$ \left(\frac{c_4}{\zeta}\right)^{\frac{q}{2}}(R-|x |)^{-qs}<1.
$$
{F}rom this and~\eqref{rtrc444}, we have that
\begin{eqnarray*}
1+w(x) &\leq & 2 \left( v\left( \frac{rx}{R}\right) + 12 r^{-qs} \right) \\ 
&\leq & 36 \left(\frac{c_4}{\zeta}\right)^{\frac{q}{2}} \left(R - |x | \right)^{-qs}\\
&\leq &36 \max\left\{ \frac{c_4}{\zeta}, 1 \right\} \left( \frac{R+1-|x |}{R-|x |}\right)^{qs} (R+1-|x |)^{-qs}\\
&\leq & c_5 (R+1-|x |)^{-qs}.
\end{eqnarray*}
Similarly, we observe that
\begin{eqnarray*}
1+w(x) &\geq & \left( v\left( \frac{rx}{R}\right) + 12 r^{-qs} \right) \\ 
&\geq & \left(\frac{c_4}{\zeta}\right)^{\frac{q}{2}}  \left(R -|x | \right)^{-qs} \\
&\geq &  \min\left\{ \frac{c_4}{\zeta}, 1 \right\}  \left(R +1- |x | \right)^{-qs},
\end{eqnarray*}
thus yielding~\eqref{5334}.

Moreover, if~$x \in B_R \setminus B_{[(r-1)R]/r}$, we have that
\begin{equation}\label{p46xxd9}
2 (R+1-|x |)^{-qs} \leq 1+w(x) \leq 2 \left( 1+\left(\frac{c_4}{\zeta}\right)^{\frac{1}{2s}} \right)^{2}  (R+1-|x |)^{-qs}.
\end{equation}
Indeed,
the left-hand inequality follows from the fact that~$w=1$ in~$\R^n\setminus B_{[(r-1)R]/r}$.
The right-hand inequality can be shown as follows
\begin{eqnarray*}
1+w(x)& = &2 (R+1-|x |)^{qs}(R+1-|x |)^{-qs}\\
&\leq & 2 \left(1+\frac{R}{r}\right)^{qs}  (R+1-|x |)^{-qs}\\
& = & 2 \left( 1+\left(\frac{c_4}{\zeta}\right)^{\frac{1}{2s}} \right)^{qs}  (R+1-|x |)^{-qs}.
\end{eqnarray*}

Finally, putting together~\eqref{5334} and~\eqref{p46xxd9}, we obtain that, for any~$x \in B_R$, there exist two constants~$c_6,c_7$ depending on~$n$, $s$, $\Lambda$ and~$\zeta$ such that
\begin{equation*}
c_6 (R+1-|x |)^{-qs} \leq 1+w(x) \leq c_7 (R+1-|x |)^{-qs},
\end{equation*}
and the estimates in~\eqref{bar_decay} can be obtained setting, for instance, $C:=\max\{1,c_6^{-1},c_7\}$
and recalling~\eqref{err1}.
\end{proof}

\begin{rem}
{\rm
The barrier reported in~\cite[Lemma 4.1]{CP16} can be retrieved setting~$m=2$ in Proposition~\ref{prop_bar}.
}
\end{rem}
   
     \section{Some auxiliary results}\label{barr_comp}
     In this section we prove some useful statements, that will be employed to prove the main results of the paper.
     
\subsection{On the potential~$W$}\label{otpot}

We provide here a regularity result on the potential~$W$. We stress that this result plays a fundamental role in the study of the decay estimates in~\eqref{asymp_decay}. 

   \begin{lemma}\label{thc_deg}
Let~$W:\R \to\R$ be a function satisfying~\eqref{pot_reg} and~\eqref{pot_deg}.

Then, for any~$r$, $t \in [-1,-1+\xi]$ with~$r\leq t$, it holds
\begin{equation}\label{5968hht}
\begin{split}
&\displaystyle\frac{C_1}{\alpha(\alpha-1)}\left((1+t)^{\alpha}-(1+r)^{\alpha}\right) \leq W(t)-W(r) \leq \frac{C_2}{\beta(\beta-1)}\left((1+t)^{\beta}-(1+r)^{\beta}\right) \\ &\mbox{and}\\
&\displaystyle\frac{C_1}{\alpha-1} \left( (1+t)^{\alpha-1} - (1+r)^{\alpha-1}\right) \leq W\rq{}(t) - W\rq{}(r) \leq \frac{C_2}{\beta-1} \left((1+t)^{\beta-1}-(1+r)^{\beta-1}\right).
\end{split}
\end{equation}
Moreover, for any~$r$, $t \in [1-\xi,1]$ with~$r\leq t$, it holds
\begin{equation}\label{eoi3j59}
\begin{split}
&  \displaystyle\frac{C_3}{\gamma(\gamma-1)} \left((1-t)^{\gamma}-(1-r)^{\gamma}\right)\leq W(t) - W(r) \leq \displaystyle\frac{C_4}{\delta(\delta-1)} \left( (1-t)^{\delta} - (1-r)^{\delta}\right) \\& \mbox{and} \\
&\displaystyle\frac{C_3}{\gamma-1} \left( (1-r)^{\gamma-1} - (1-t)^{\gamma-1}\right) \leq W\rq{}(t) - W\rq{}(r) \leq \frac{C_4}{\delta-1} \left((1-r)^{\delta-1}-(1-t)^{\delta-1}\right).
\end{split}
\end{equation}
\end{lemma}  
   
\begin{proof} 
We will only show~\eqref{5968hht}, since~\eqref{eoi3j59} can be
proved in a similar way. Moreover, we will focus on the upper bounds
in~\eqref{5968hht}, being the lower bounds based on analogous calculations.

Let~$r$, $t \in [-1,-1+\xi]$ and~$r\leq t$. Then, thanks to~\eqref{pot_reg}
and~\eqref{pot_deg}, we can apply the Fundamental Theorem of Calculus and obtain that
\begin{eqnarray*}
W\rq{}(t) - W\rq{}(r) & = & \int_{r}^t W\rq{}\rq{}(\sigma) \, d\sigma \leq C_2 \int_{r}^t (1+\sigma)^{\beta-2} \, d\sigma \\
& = & \frac{C_2}{\beta-1} \left((1+t)^{\beta-1} - (1+r)^{\beta-1}\right) \end{eqnarray*}
which establishes the upper bound in the second line of~\eqref{5968hht}.

We now show the upper bound in the first line of~\eqref{5968hht}.
For this, we use the Fundamental Theorem of Calculus 
and the fact that~$W\rq{}(-1)=0$ to write
\begin{eqnarray*}&& W(t)-W(r)=\int_r^t W\rq{}(\sigma) \, d\sigma
=\int_r^t \big(W\rq{}(\sigma)- W\rq{}(-1)\big)\, d\sigma.
\end{eqnarray*}
Since~$\sigma\ge r\ge-1$, we can exploit the upper bound
in the second line of~\eqref{5968hht} and we obtain that
$$W(t)-W(r)\le \frac{C_2}{\beta-1}
\int_r^t  (1+\sigma)^{\beta-1} \, d\sigma
=  \frac{C_2}{\beta(\beta-1)} \left( (1+t)^{\beta} - (1+r)^{\beta}\right),
$$
as desired.
\end{proof}

We point out that
in the case~$\alpha=\beta=\gamma=\delta=2$, Lemma~\ref{thc_deg} implies that
\begin{equation*}
   W\rq{}(t) \geq W\rq{}(r)+c(t-r) \quad\mbox{for any } r\leq t , \; r,t \in [-1,-1+c]\cup[1-c,1]
\end{equation*}
 and for some~$c >0$. This inequality is the so called~\textit{convexity property} for a \textit{non-degenerate} potential W and is used in~\cite{PSV13, CP16} in order to compute the decay estimates for the minimizers.   

    \subsection{Barriers and comparison principle}\label{bars}
 In this subsection we recall a strong comparison principle related to the operator~${L}_K$, that will be extensively used in the proof of
 the estimates in~\eqref{asymp_decay} and~\eqref{eq:asymp-derivata}.

   \begin{prop}[Proposition~4.4 in~\cite{CP16}]\label{comp}
     Let~$s\in (0,1)$ and~${K}$ satisfy~\eqref{krn_symm} and~\eqref{main_ellipt}. Let~$f_1$, $f_2: \R^n \times \R \to \R$ be two continuous functions. 
     
     Let~$\Omega$ be a domain of~$\R^n$ and~$v$, $w \in L^{\infty}(\R^n) \cap C^{2s+\theta}(\Omega)$, for some~$\theta>0$, be such that
   \begin{equation*}
   	\begin{cases}
   	L_K v \leqslant f_1(\cdot,v) & \mbox{in} \ \Omega,\\
   	L_K w \geqslant f_2(\cdot,w)  & \mbox{in} \ \Omega,\\
   	v \geqslant w & \mbox{in} \ \R.
   	\end{cases}
   \end{equation*}
   
   Suppose also that, for any~$x \in \Omega$,
   \begin{equation*}
   f_1(x,w(x)) \leqslant f_2(x,w(x)).
   \end{equation*}
  If there exists a point~$x_{0} \in \Omega$ such that~$v(x_{0}) =w(x_{0})$, then~$v \equiv w$ in the whole~$\Omega$.
   \end{prop}
  
    \subsection{Basic properties of the fractional Allen-Cahn equation}
We mention here two handful lemmata. The first one poses sufficient conditions in order for the fractional Allen-Cahn equation to behave well under limit procedures. The second one relates the solutions of the fractional Allen-Cahn equation to class~A minimizers of~$\Ec$.

 \begin{lemma}[Lemma~4.8 in~\cite{CP16}]\label{eq_stab}
      Let~$s\in (0,1)$ and~${K}$ satisfy~\eqref{krn_symm} and~\eqref{main_ellipt}. Let~$W \in C^{1,\vartheta}_{\rm loc}(\R)$ for some~$\vartheta \in (0,1)$. 
      
      Let~$\Omega \subset \Rn$ be a Lipschitz domain and let~$(v_j)_{j \in \N} \subset \Hbb(\Omega)\cap L^\infty(\Rn)$ be a sequence of weak solutions of
       $$
       L_K v_j = W'(v_j) \qquad \mbox{in} \ \Omega,
       $$
such that there exists~$c>0$ such that, for any~$j\in\mathbb{N}$,
       $$
       [v_j]_{\Hbb(\Omega)} + \|v_j\|_{L^\infty(\Rn)} \leqslant c.
       $$
       
       Suppose also that~$v_j$ converges to a function~$v$ uniformly on compact subset of~$\Rn$, as~$j\to+\infty$. 
       
       Then, $v \in \Hbb(\Omega)\cap L^\infty(\Rn)$ is a weak solution of
       $$
       L_K v = W'(v) \qquad \mbox{in} \ \Omega.
       $$
       \end{lemma}
    
    \begin{lemma}[Theorem~3 in~\cite{CP16}]\label{lemma:degiorgi} Let~$n \geq 1$, $s\in (0,1)$ and~${K}$ satisfy~\eqref{krn_symm} and~\eqref{main_ellipt}. Let~$W$ be such that
   \[W(\pm 1) = W\rq{}(\pm 1) =0. \] 
   
    Let~$u: \Rn \to (-1,1)$ belong to~$C^{1+2s+{\theta}}(\Rn)$, for some~${\theta} >0$. Suppose that~$u$ is a solution of
    	\begin{equation*}
    		L_K u = W'(u)\qquad \mbox{in} \ \Rn,
    	\end{equation*}
    	satisfying 
    	$$
    	\partial_{x_n} u(x) \geqslant 0, \qquad {\mbox{for all }} x \in \Rn,
    	$$
    	and
    	$$
    	\lim_{x_n\to \pm \infty} u(x',x_n) = \pm 1, \qquad {\mbox{for all }}  x' \in \R^{n-1}.
    	$$
    
    	Then, $u$ is a class~A minimizer of~$\Ec$.
    \end{lemma}

We also refer the reader to~\cite[Lemma 7]{PSV13} and~\cite[Theorem 1]{PSV13} for the proofs of, respectively, Lemmata~\ref{eq_stab} and~\ref{lemma:degiorgi}, when the kernel~${K}(z) :=\snr{z}^{-n-2s}$.

\section{Asymptotic estimates for~${L}_K$ in dimension~$n=1$}\label{90fxd9916}
 
In this section we collect asymptotic results regarding the operator~${L}_K$, defined in~\eqref{main_op}, in dimension~$n=1$. They will be used in the proofs of Theorems~\ref{main_thm} and~\ref{main_thm_symm}.

\begin{prop}\label{tonto}
Let~$\kappa\in (0, +\infty)$ and~$\sigma$, $\tau\in(1,+\infty)$. Let~$K$ satisfy~\eqref{krn_symm} and~\eqref{main_ellipt}. 

Let~$\phi \in C^{\infty}(\R)$ be such that
\begin{equation*}
\phi(x)= \begin{cases}
|x|^{-\sigma} &\mbox{if } x<-\kappa, \\
|x|^{-\tau} &\mbox{if } x>\kappa
\end{cases}
\end{equation*}
and
\begin{equation}\label{addipotesi76}
{\mbox{$\phi(x)\ge\gamma$ for all~$x\in[-\kappa,\kappa]$, for some~$\gamma\in(0,+\infty)$.}}
\end{equation}

Then,
\begin{equation}\label{yuuy_secs}
\lim_{|x | \to  +\infty} |x|^{1+2s} {L}_K\phi(x) \leq \Lambda \left( \frac{\kappa^{1-\sigma}}{\sigma-1}+ \int_{-\kappa}^{\kappa} \phi(y) \, dy+\frac{\kappa^{1-\tau}}{\tau-1}\right),
\end{equation}
where~$\Lambda$ is the quantity appearing in~\eqref{main_ellipt}.
\end{prop}

\begin{proof}
We will only prove~\eqref{yuuy_secs} as~$x \to -\infty$, being
the limit as~$x\to +\infty$ analogous.

For any~$x<-2\kappa$, we set
\begin{equation}\label{52BIS}
\begin{split}
&A_K(x):= PV_x \int_{-\infty}^{-\kappa}\left( \phi(y)-\phi(x)\right) K(x-y)\, dy, \\ &B_K(x):= \int_{-\kappa}^{\kappa} \left( \phi(y) -\phi(x)\right) K(x-y) \, dy\\
{\mbox{and }}\qquad
&D_K(x):= \int_{\kappa}^{+\infty}\left( \phi(y)-\phi(x)\right) K(x-y)\, dy.
\end{split}
\end{equation}
In the following computations we will omit the principal value notation,
for the sake of readability.

We claim that 
\begin{equation}\label{x_secs}
\lim_{x \to -\infty} |x |^{1+2s} A_K(x) \leq  \frac{\Lambda\kappa^{1-\sigma}}{\sigma-1}.
\end{equation}
In order to show this, we apply the change of variable~$y:=|x|\theta$ and compute
\begin{equation*}
A_K(x) = \int_{-\infty}^{-\kappa} \left( | y |^{-\sigma} - |x |^{-\sigma} \right) K(x-y)\, dy 
= |x |^{1-\sigma} 
\int_{-\infty}^{-{\kappa}/{|x|}} \left( |\theta |^{-\sigma}-1\right) K\left( x (1+\theta)\right) \, d\theta.
\end{equation*}
We also set
\begin{equation}\label{53BIS}
\begin{split}
& A_{K,I}(x):= |x |^{1-\sigma} \int_{-\infty}^{-{3}/{2}}\left( |\theta |^{-\sigma}-1\right) K\left( x (1+\theta)\right) \, d\theta, \\
& A_{K,II}(x):=  |x |^{1-\sigma} 
\int_{-{3}/{2}}^{-{1}/{2}}\left(|\theta |^{-\sigma}-1\right) K\left( x (1+\theta)\right) \, d\theta \\ \mbox{and}\qquad
&A_{K,III}(x):=  | x |^{1-\sigma} \int_{-{1}/{2}}^{-{\kappa}/{| x|}}\left( | \theta |^{-\sigma}-1\right) K\left( x (1+\theta)\right) \, d\theta.
\end{split}
\end{equation}
Exploiting~\eqref{main_ellipt}, we obtain that
\begin{equation}\label{equno}
|A_{K,I}(x)| \leq \Lambda |x |^{-(\sigma+2s)}\int_{-\infty}^{-{3}/{2}}\frac{1- | \theta |^{-\sigma}}{| 1+\theta|^{1+2s}} \, d\theta \leq C |x |^{-(\sigma+2s)},
\end{equation}
for some~$C>0$, depending on~$s$, $\sigma$ and~$\Lambda$.

Moreover, we check that
\begin{equation}\label{jgsp}
| A_{K,II}(x)| \leq C |x |^{-(\sigma+2s)}.
\end{equation}
To this aim, we use the change of variable~$z:= 1+\theta$ and
the symmetry of~$K$ in~\eqref{krn_symm}
to see that
\begin{equation*}
\begin{split}
&\int_{-{3}/{2}}^{-{1}/{2}} \left( |\theta|^{-\sigma}-1 \right) K \left( x(1+\theta)\right) \, d\theta 
= \int_{-{1}/{2}}^{{1}/{2}} \left( ( 1-z)^{-\sigma}-1\right) K(xz) \, dz \\
&\qquad=\int_{-{1}/{2}}^{{1}/{2}} \left( \sigma z +O(|z |^2) \right) K(xz) \, dz 
=\int_{-{1}/{2}}^{{1}/{2}} O(|z |^2) K(xz) \, dz.
\end{split}
\end{equation*}
As a consequence,
\begin{equation*}
\begin{split}
&\left| \int_{-{3}/{2}}^{-{1}/{2}} \left( |\theta|^{-\sigma}-1 \right) K \left( x(1+\theta)\right) \, d\theta \right|
= \left|\int_{-{1}/{2}}^{{1}/{2}} O(|z |^2)  K(xz) \, dz\right|\\
&\qquad\qquad\leq C\Lambda |x |^{-(1+2s)}  \int_{-{1}/{2}}^{{1}/{2}}| z|^{1-2s}\, dz \leq C |x |^{-(1+2s)} ,
\end{split}
\end{equation*} up to renaming~$C>0$.
This leads to~\eqref{jgsp}, as desired.

Furthermore, observing that~$\theta^{-\sigma}>1$ if~$\theta \in (0,1/2)$, we obtain that
\begin{equation*}
\begin{split}
A_{K,III}(x) & \leq \Lambda |x |^{-(\sigma+2s)} 
\int_{-{1}/{2}}^{-{\kappa}/{|x|}} \frac{ |\theta|^{-\sigma}-1}{ (1+\theta)^{1+2s}} \, d\theta\\
&= \Lambda |x |^{-(\sigma+2s)} 
\left(
\int_{{\kappa}/{| x |}}^{{1}/{2}} 
\theta^{-\sigma}\big(1+O(\theta)\big) \, d\theta -
\int_{{\kappa}/{|x |}}^{{1}/{2}} \frac{d\theta}{(1+\theta)^{1+2s}}\right)\\
&= \Lambda |x |^{-(\sigma+2s)} 
\left( \frac{1}{1-\sigma}\left(2^{\sigma-1} -\kappa^{1-\sigma}|x|^{\sigma-1} \right) +O(|x|^{\sigma-2})
-\int_{{\kappa}/{|x |}}^{{1}/{2}} \frac{d\theta}{(1+\theta)^{1+2s}}\right).
\end{split}
\end{equation*}
Combining this with~\eqref{equno} and~\eqref{jgsp},
and noticing that~$A_K(x) = A_{K,I}(x)+ A_{K,II}(x)+ A_{K,III}(x)$,
we obtain~\eqref{x_secs}. 

Now, we want to show that
\begin{equation}\label{fcpmd}
\lim_{x \to - \infty} |x |^{1+2s} B_K(x) \leq  \Lambda \int_{-\kappa}^{\kappa} \phi(y)\, dy.
\end{equation}
For this, we take~$\gamma$ as in~\eqref{addipotesi76}
and we suppose that~$|x|>\gamma^{-1/\sigma}$.
In this way, we have that~$\phi(x)=|x|^{-\sigma}<\gamma\le\phi(y)$
for all~$y\in(-\kappa,\kappa)$,
thanks to~\eqref{addipotesi76}.
Therefore, by~\eqref{main_ellipt},
\begin{eqnarray*}
|x|^{1+2s}B_K(x)&=& |x|^{1+2s}
\int_{-\kappa}^{\kappa} \left( \phi(y) -\phi(x)\right) K(x-y) \, dy\\
&\le& \Lambda \int_{-\kappa}^{\kappa} \left( \phi(y)-\phi(x) \right) \left| 1 -\frac{y}{x} \right|^{-(1+2s)}\, dy.
\end{eqnarray*}
We can now take the limit as~$x\to-\infty$ and use the
Dominated Convergence Theorem to deduce~\eqref{fcpmd}.

Furthermore, we claim that
\begin{equation}\label{fcpmt}
\lim_{x \to - \infty} |x |^{1+2s} D_K(x) \leq  \frac{\Lambda\kappa^{1-\tau}}{\tau-1},
\end{equation}
Indeed, the change of variable~$y:=|x|\theta$ 
and~\eqref{main_ellipt} give that
\begin{equation*}
\begin{split}
D_K(x)&= \int_{\kappa}^{+\infty} (y^{-\tau} - |x |^{-\sigma})K(x-y) \, dy \\
&= |x |\int_{{\kappa}/{|x|}}^{+\infty}\big(
|x |^{-\tau} \theta^{-\tau} - |x |^{-\sigma}\big) K(x (1+\theta))\, d\theta\\
&\leq \Lambda |x |^{-2s} \int_{{\kappa}/{|x|}}^{+\infty} \frac{|x |^{-\tau} \theta^{-\tau} +|x |^{-\sigma}}{(1+\theta)^{1+2s}}\, d\theta\\
&\leq  \Lambda |x |^{-(\tau+2s)} 
\int_{{\kappa}/{|x|}}^{+\infty}  \frac{\theta^{-\tau}}{(1+\theta)^{1+2s}}\, d\theta
+\Lambda | x |^{-(\sigma+2s)}
\int_{{\kappa}/{| x|}}^{+\infty} \frac{d\theta}{(1+\theta)^{1+2s}}.
\end{split}\end{equation*}
We notice that
\begin{eqnarray*}
\int_{{\kappa}/{|x|}}^{+\infty}  \frac{\theta^{-\tau}}{(1+\theta)^{1+2s}}\, d\theta
&=&\int_{{\kappa}/{|x|}}^{1/2}
\theta^{-\tau}\big(1+O(\theta)\big)\, d\theta
+ \int_{1/2}^{+\infty}  \frac{\theta^{-\tau}}{(1+\theta)^{1+2s}}\, d\theta\\
&=&\frac{1}{1-\tau}\left( 2^{\tau-1}-\kappa^{1-\tau}|x|^{\tau-1}
\right) + O(|x|^{\tau-2})
+ \int_{1/2}^{+\infty}  \frac{\theta^{-\tau}}{(1+\theta)^{1+2s}}\, d\theta.
\end{eqnarray*}
As a result,
\begin{eqnarray*}
|x|^{1+2s}D_K(x)&\le& 
\Lambda|x|^{1-\tau}\left(\frac{1}{1-\tau}\left( 2^{\tau-1}-\kappa^{1-\tau}|x|^{\tau-1}
\right) + O(|x|^{\tau-2})
+ \int_{1/2}^{+\infty}  \frac{\theta^{-\tau}}{(1+\theta)^{1+2s}}\, d\theta\right)\\&&\qquad
+\Lambda |x |^{1-\sigma}
\int_{{\kappa}/{|x|}}^{+\infty} \frac{d\theta}{(1+\theta)^{1+2s}},
\end{eqnarray*}
which gives the desired claim in~\eqref{fcpmt}.

Thus, since~$L_K\phi(x)= A_K(x) + B_K(x) + D_K(x)$,
combining together~\eqref{x_secs}, \eqref{fcpmd} and~\eqref{fcpmt} yields the thesis.
\end{proof}    

\begin{prop}\label{tontobis}
Let~$\kappa\in (0, +\infty)$ and~$\sigma$, $\tau\in(1,+\infty)$. Let~$K$ satisfy~\eqref{krn_symm} and~\eqref{newbound}. 

Let~$\phi \in C^{\infty}(\R)$ be such that
\begin{equation*}
\phi(x)= \begin{cases}
|x|^{-\sigma} &\mbox{if } x<-\kappa, \\
|x|^{-\tau} &\mbox{if } x>\kappa
\end{cases}
\end{equation*}
and
\begin{equation}\label{addipotesi76bis}
{\mbox{$\phi(x)\ge\gamma$ for all~$x\in[-\kappa,\kappa]$, for some~$\gamma\in(0,+\infty)$.}}
\end{equation}

Then,
\begin{equation}\label{yuuy_secsbis}
\lim_{|x | \to  +\infty} |x|^{1+2s} {L}_K\phi(x) \geq \lambda \left( \frac{\kappa^{1-\sigma}}{\sigma-1}+ \int_{-\kappa}^{\kappa} \phi(y) \, dy+\frac{\kappa^{1-\tau}}{\tau-1}\right),
\end{equation}
where~$\lambda$ is the quantity appearing in~\eqref{newbound}.
\end{prop}

The proof of Proposition~\ref{tontobis} is similar to the one of Proposition~\ref{tonto}, but requires some care in the estimates
of some terms, therefore we provide here below the details for the facility of the reader.

\begin{proof}[Proof of Proposition~\ref{tontobis}]
We will only prove~\eqref{yuuy_secsbis} as~$x \to -\infty$, being
the limit as~$x\to +\infty$ analogous.

For any~$x<-2\kappa$, we take~$A_K(x)$, $B_K(x)$ and~$D_K(x)$ as in~\eqref{52BIS}.

We claim that 
\begin{equation}\label{x_secsbis}
\lim_{x \to -\infty} |x |^{1+2s} A_K(x) \geq  \frac{\lambda\kappa^{1-\sigma}}{\sigma-1}.
\end{equation}
In order to show this, in the notation of~\eqref{53BIS},
we write~$A_K(x)= A_{K,I}(x)+A_{K,II}(x)+A_{K,III}(x)$.
Moreover, using the estimates in~\eqref{equno} and~\eqref{jgsp}, we see that
\begin{equation}\label{STAR67839287654}
|A_K(x)|\ge A_{K,III}(x)-|A_{K,I}(x)|-|A_{K,II}(x)|\ge  A_{K,III}(x)-C | x |^{-(\sigma+2s)},
\end{equation}
for some~$C>0$, depending on~$s$, $\sigma$ and~$\Lambda$ (being~$\Lambda$ the quantity appearing in~\eqref{newbound}).

Furthermore, since~$\theta^{-\sigma}>1$ if~$\theta \in (0,1/2)$, we obtain that
\begin{equation*}
\begin{split}
A_{K,III}(x) & \geq \lambda |x |^{-(\sigma+2s)} 
\int_{-{1}/{2}}^{-{\kappa}/{|x|}} \frac{ |\theta|^{-\sigma}-1}{ (1+\theta)^{1+2s}} \, d\theta\\
&= \lambda | x |^{-(\sigma+2s)} 
\left(
\int_{{\kappa}/{|x |}}^{{1}/{2}} 
\theta^{-\sigma}\big(1+O(\theta)\big) \, d\theta -
\int_{{\kappa}/{|x |}}^{{1}/{2}} \frac{d\theta}{(1+\theta)^{1+2s}}\right)\\
&= \lambda |x |^{-(\sigma+2s)} 
\left( \frac{1}{\sigma-1}\left( \kappa^{1-\sigma}|x|^{\sigma-1}-2^{\sigma-1} \right) +O(|x|^{\sigma-2})
-\int_{{\kappa}/{|x |}}^{{1}/{2}} \frac{d\theta}{(1+\theta)^{1+2s}}\right).
\end{split}
\end{equation*}
Therefore, from this and~\eqref{STAR67839287654} we deduce~\eqref{x_secsbis}.
 
Now, we show that
\begin{equation}\label{fcpmdbis}
\lim_{x \to - \infty} |x |^{1+2s} B_K(x) \geq  \lambda \int_{-\kappa}^{\kappa} \phi(y)\, dy.
\end{equation}
For this, we take~$\gamma$ as in~\eqref{addipotesi76bis}
and we suppose that~$|x|>\gamma^{-1/\sigma}$.
In this way, we have that~$\phi(x)=|x|^{-\sigma}<\gamma\le\phi(y)$
for all~$y\in(-\kappa,\kappa)$,
thanks to~\eqref{addipotesi76bis}.
Therefore, by~\eqref{newbound},
\begin{eqnarray*}
|x|^{1+2s}B_K(x)&=& |x|^{1+2s}
\int_{-\kappa}^{\kappa} \left( \phi(y) -\phi(x)\right) K(x-y) \, dy\\
&\geq& \lambda \int_{-\kappa}^{\kappa} \left( \phi(y)-\phi(x) \right) \left| 1 -\frac{y}{x} \right|^{-(1+2s)}\, dy.
\end{eqnarray*}
We can now take the limit as~$x\to-\infty$ and use the
Dominated Convergence Theorem to deduce~\eqref{fcpmdbis}.

Also, we claim that
\begin{equation}\label{fcpmtbis}
\lim_{x \to - \infty} |x |^{1+2s} D_K(x) \geq  \frac{\lambda\kappa^{1-\tau}}{\tau-1},
\end{equation}
Indeed, the change of variable~$y:=|x|\theta$ and the positivity of the kernel give that
\begin{equation}\label{9hgv9bis}
\begin{split}
D_K(x)&= \int_{\kappa}^{+\infty} (y^{-\tau} - |x |^{-\sigma})K(x-y) \, dy \\
&= |x |\int_{{\kappa}/{|x|}}^{+\infty}\big(
|x |^{-\tau} \theta^{-\tau} - |x |^{-\sigma}\big) K(x (1+\theta))\, d\theta\\
&\geq |x |^{1-\tau} \int_{k /|x |}^{+\infty} \theta^{-\tau} K(x(1+\theta)) \, d\theta - |x |^{1-\sigma} \int_{0}^{+\infty} K(x(1+\theta)) \, d\theta.
\end{split}\end{equation}
Now, from the upper bound in~\eqref{newbound} we infer that
\begin{eqnarray*}
&&|x |^{1-\sigma} \int_{0}^{+\infty} K(x(1+\theta)) \, d\theta
\le\Lambda  |x |^{1-\sigma} \int_{0}^{+\infty} \frac{d\theta}{|x|^{1+2s}(1+\theta)^{1+2s}}=C|x|^{-(\sigma+2s)},
\end{eqnarray*} for some~$C>0$ depending on~$s$ and~$\Lambda$.

Moreover, from the lower bound in~\eqref{newbound}, we deduce that
\begin{equation*}
\begin{split}&
|x |^{1-\tau} \int_{k /|x |}^{+\infty} \theta^{-\tau} K(x(1+\theta)) \, d\theta 
\ge \lambda  |x |^{1-\tau} \int_{k /|x |}^{+\infty} \frac{\theta^{-\tau}}{| x|^{1+2s}(1+\theta)^{1+2s}} \, d\theta 
\\&\qquad =\lambda |x |^{-(\tau+2s)} \left( \int_{k /|x |}^{\frac{1}{2}} \theta^{-\tau} \left(1+O(\theta)\right) \, d\theta +\int_{\frac{1}{2}}^{+\infty} \frac{\theta^{-\tau}}{(1+\theta)^{1+2s}} \, d\theta \right) \\
 &\qquad = \lambda |x |^{-(\tau+2s)} \left( \frac{k^{1-\tau}}{\tau-1} | x |^{\tau-1} -\frac{2^{\tau-1}}{\tau-1}+\int_{\frac{1}{2}}^{+\infty} \frac{\theta^{-\tau}}{(1+\theta)^{1+2s}} \, d\theta \right)+  O(|x |^{-2-2s}) .
\end{split}
\end{equation*}
Plugging the last two displays into~\eqref{9hgv9bis} give that
\begin{equation*}
\begin{split}
D_K(x)&\geq \lambda |x |^{-(\tau+2s)} \left( \frac{k^{1-\tau}}{\tau-1} |x |^{\tau-1} -\frac{2^{\tau-1}}{\tau-1}+\int_{\frac{1}{2}}^{+\infty} \frac{\theta^{-\tau}}{(1+\theta)^{1+2s}} \, d\theta \right)\\&\qquad\qquad+  O(|x |^{-2-2s})-C|x|^{-(\sigma+2s)}.
\end{split}\end{equation*}
Multiplying by~$|x|^{1+2s}$ and taking the limit show~\eqref{fcpmtbis}.

Thus, since~$L_K\phi(x)= A_K(x) + B_K(x) + D_K(x)$,
combining together~\eqref{x_secsbis},~\eqref{fcpmdbis} and~\eqref{fcpmtbis} yields the thesis.
\end{proof}    

\section{Minimizing the energy in intervals}\label{min_on_int}
In this section, we deal with the problem of minimizing the energy~$\Ec$ in bounded intervals~$I \subset \R$.

More precisely, in Lemma~\ref{min_open_inter} below, we provide the existence of a minimizer~$v_{I}$ and present an upper bound for its energy as a function of the lenght of the interval~$ I$. Moreover, in Propositions~\ref{eupro} and~\ref{symm} we investigate respectively the monotonicity property and the oddness of~$v_I$ (in case~\eqref{pot_symm} is in force). Also, Proposition~\ref{comepo} studies the asymptotic behavior of the minimizer~$v_{[0,R]}$ in the interval~$[0,R]$ as~$R\to+\infty$.
\medskip

In the rest of this section, the analytic framework introduced
in Section~\ref{secanalytfra00} is assumed.

\begin{lemma}\label{leaj}
Let~$K$ satisfy~\eqref{krn_symm} and~\eqref{main_ellipt}.
Let~$\Omega$ be a domain of~$\R^n$ and ~$u$ and~$v$ be two measurable functions in~$\Hbb(\Omega)$.

Then
\begin{equation}\label{oas}
\Ec({\rm min}\{u,v\},\Omega) + \Ec({\rm max}\{u,v\},\Omega) \leq \Ec(u, \Omega) + \Ec(v,\Omega).
\end{equation}
Moreover, equality holds in~\eqref{oas} if and only if
\begin{equation}\label{jjss}
{\mbox{either }} \quad u(x) \leq v(x) \quad\mbox{or}\quad u(x)\ge v(x) \quad\mbox{for any } x \in \R^n.
\end{equation}
\end{lemma}

In~\cite[Lemma~3]{PSV13}, the authors prove the same result in the particular case~$K(z):=|z|^{-n-2s}$. In~\cite[Lemma~3.2]{CozziValdNONLINEARITY}, this result is then stated for rough kernels,
but we provide here the full proof, since we add the claim in~\eqref{jjss}.

\begin{proof}[Proof of Lemma~\ref{leaj}]
We define
\begin{equation}\label{wisk}
m(x):= {\rm min}\{u(x),v(x)\} \qquad\mbox{and}\qquad M(x):={\rm max}\{u(x),v(x)\}.
\end{equation}
Also, since
\begin{equation}\label{oks}
\Pc({\rm min}\{u,v\},\Omega) +\Pc({\rm max}\{u,v\},\Omega) = \Pc(u, \Omega) + \Pc(v,\Omega),
\end{equation}
we will just focus on the kinetic term~$\Hc$.

Let~$x$, $y\in\R^n$ and consider the two possible scenarios:
\begin{enumerate}[\rm(i)]
        \item either~$u(x) \leq v(x)$ and~$u(y) \leq v(y)$, or~$u(x)\ge v(x)$ and~$u(y) \ge v(y)$,
        \label{ozssz}
        \item either~$u(x) \leq v(x)$ and~$u(y)\ge v(y)$, or~$u(x)\ge v(x)$ and~$u(y) \leq v(y)$.\label{ozsszz}
    \end{enumerate}
If~\eqref{ozssz} holds, then~\eqref{oas} plainly follows with the equal sign. Thus, we might suppose that~\eqref{ozsszz} holds. If this is the case, we compute
\begin{equation}\label{spem}
\begin{split}
&|m(x)-m(y)|^2 + | M(x) - M(y)|^2 \\
&\qquad= |u(x)-v(y)|^2 + |v(x) - u(y)|^2 \\
&\qquad= | u(x)-u(y)|^2 + | v(x) - v(y)|^2 + 2 \left( u(x) - v(x) \right) \left( u(y)-v(y)\right) \\
&\qquad \leq |u(x)-u(y)|^2 + | v(x) - v(y)|^2,
\end{split}
\end{equation}
thus leading to~\eqref{oas}.

Furthermore, the claim in~\eqref{jjss} follows by inspection of~\eqref{oks} and~\eqref{spem}.
\end{proof}

We recall the following statement, that guarantees
the existence of a minimizer for the energy in a given domain.

\begin{lemma}[Lemma 4.7 in~\cite{CP16}]\label{en_minimiz}
Let~$n\geq1$ and~$s \in(0,1)$. Let~$K$ satisfy~\eqref{krn_symm} and~\eqref{main_ellipt}. Assume that\footnote{In~\cite{CP16} the authors refer to a condition on~$W$ reading as: $ W(\pm 1)= W\rq{}(\pm 1)=0$. Nevertheless, the proof only exploits~$W(\pm 1)=0$.} $W(\pm 1) =0$.

Let~$\Omega \subset \R^n$ be a bounded Lipschitz domain. 
Let~$w_0 : \R^n \to [-1,1]$ be a measurable function and
suppose that there exists another measurable function~$w$ which coincides whith~$w_0$ in~$\R^n \setminus \Omega$ and such that
$$ \Ec (w, \Omega)< +\infty.$$

Then, there exists a local minimizer~$v_{\Omega}: \R^n \to [-1,1]$ for~$\Ec$ in~$\Omega$ which coincides with~$w_0$ in~$\R^n \setminus \Omega$.
\end{lemma}

We include  here another preliminary result which takes care of the growth of the energy~$\Ec$ of local minimizers inside large intervals.

\begin{prop}\label{mdss}
Let~$K$ satisfy~\eqref{krn_symm} and~\eqref{main_ellipt} and let~$W \in L^{\infty}(\R)$ be such that~$W(\pm1)=0$. Let~$\alpha$, $\beta \in [-1,1]$ and~$J:=[a,b] \subset \R$ such that~$|J|>6$.

If~$v:\R \to [-1,1]$ is a local minimizer of~$\Ec$ in~$J$ satisfying 
\begin{equation*}
v=\alpha \mbox{ if } x\leq a \qquad\mbox{and}\qquad v=\beta \mbox{ if } x\geq b,
\end{equation*}
then there exists a positive constant~$C$, depending only on~$s$, $\Lambda$ and~$\Vert W \Vert_{L^{\infty}(\R)}$ such that
\begin{equation*}
\Ec(v, J\rq{})\leq  C \, \Psi_s(| J\rq{} |),
\end{equation*}
for any~$J\rq{} \subset J$ and~$| J\rq{}|>6$.

Here above,~$\Psi_s$ is the function introduced in~\eqref{fnc_psi}.
\end{prop}

\begin{proof}
We will denote by~$C$ any positive constant that depends only on~$s$, $\Lambda$ and~$\Vert W\Vert_{L^{\infty}(\R)}$.

Without loss of generality we suppose that~$J=[-a,a]$ for some~$a>3$. Also, if~${\rm dist}(J\rq{},\R\setminus(-a,a))>2$, then Proposition~\ref{mdss} follows\footnote{
We point out that Proposition~$3.1$ in~\cite{CozziValdNONLINEARITY} is proved under an analogue assumption of~\eqref{main_ellipt} here, with~$r_0=1$. Nevertheless, the proof only exploits the right-hand inequality in~\eqref{main_ellipt} and thus it is valid in our framework too.
} from~\cite[Proposition 3.1]{CozziValdNONLINEARITY} and by observing that~$v$ is also a local minimizer in any sub interval of~$J$ (see Remark~\ref{remmarco}). Consequently, we might set~$J\rq{}:=[-c,d]$, such that~$a-c<2$ and~$a-d<2$.

Now, we define the following function
\begin{equation*}\psi(x):=
\begin{cases}
\alpha &\mbox{if } x \leq -a, \\
-x(\alpha+1)/2 +\alpha-a(1+\alpha)/2 &\mbox{if } x \in [-a, -a+2],\\
-1 & \mbox{if }   |x | \leq a-2, \\
x (\beta+1)/2 +\beta-a(\beta+1)/2 &\mbox{if } x \in [a-2, a],\\
\beta  &\mbox{if } x \geq a
\end{cases}
\end{equation*}
and we claim that
\begin{equation}\label{cheinerp}
\Ec(\psi,  J ) \leq C \,\Psi_s(|J\rq{}|),
\end{equation}
for some~$C>0$ depending only on~$s$, $\Lambda$ and~$\Vert W \Vert_{L^{\infty}(\R)}$.

Indeed, let~$x\in[-a,a]$ and set~$d(x):= {\rm max} \left\{ a-2-|x| , 1 \right\}$.
Then, for any~$y \in \R$,
\begin{equation*}|\psi(x)-\psi(y)| \leq
\begin{cases}
d(x)^{-1} | x-y| &\mbox{if } |x-y| \leq d(x), \\
2 &\mbox{if } |x-y| > d(x) .
\end{cases}
\end{equation*}
As a consequence, recalling~\eqref{main_ellipt},
\begin{equation*}
\begin{split}&
\int_{\R} |\psi(x)-\psi(y)|^2 K(x-y) \, dy \\
&\leq 4 \Lambda \left( d(x)^{-2} \int_{\{| x-y|\leq d(x)\}} |x-y |^{-1+2(1-s)} \, dy+  \int_{\{|x-y|> d(x)\}}|x-y |^{-1-2s} \, dy \right) \\
&=  8 \Lambda \left( d(x)^{-2} \int_{0}^{d(x)} |y |^{-1+2(1-s)} \, dy+ \int_{ d(x)}^{+\infty} |  y|^{-1-2s} \, dy \right) \\
&\leq C_{s,\Lambda} \ d(x)^{-2s}.
\end{split}
\end{equation*}
{F}rom this inequality, we deduce that
\begin{equation*}
\begin{split}
&\Hc(\psi,J ) \leq  C_{s,\Lambda} \int_{-a}^{a} d(x)^{-2s} \, dx \le C_{s,\Lambda} \left( \int_0^{a-3} \frac{dx}{(a-2-x)^{2s}}  + 1 \right) \\
&\qquad\le C_{s,\Lambda} +C_{s,\Lambda} \begin{cases}
\log(a-2) &\mbox{if } s =1/2,\\
\frac{1}{1-2s} \left( ( a -2 )^{1-2s}-1 \right) &\mbox{if } s\neq 1/2
\end{cases}\\
&\qquad \leq C_{s,\Lambda} \,\Psi_s(a-2) \leq C_{s,\Lambda}  \,\Psi_s(|J\rq{} |),
\end{split}
\end{equation*}
up to relabeling~$C_{s,\Lambda}$. Since
exploiting the boundedness of~$W$ we also have that
\begin{equation*}
\Pc(\psi,J ) = \int_{-a}^a W(\psi(x))\, dx =
\int_{-a}^{-a+2}W(\psi(x))\, dx +\int_{a-2}^{a}W(\psi(x))\, dx \leq C_{\Vert W \Vert_{L^{\infty}(\R)}},
\end{equation*}
it follows that~\eqref{cheinerp} holds true.

Being~$v$ a local minimizer of~$\Ec$ in~$J$ and~$\psi$ a suitable competitor, \eqref{cheinerp} yields the thesis.
\end{proof}

\begin{rem}
{\rm
We stress that Proposition~\ref{mdss} is an adaption of~\cite[Proposition~3.1]{CozziValdNONLINEARITY} to our framework. In particular, in~\cite{CozziValdNONLINEARITY}, the authors operate in higher dimensions and impose no constraints on~$v$ outside~$J$ (other than~$v \in [-1,1]$). On the other hand, Proposition~\ref{mdss} allows~$J\rq{}$ to be arbitrarily close to~$J$.
}
\end{rem}

We are now in the position to prove the first main result of this section.

\begin{lemma}\label{min_open_inter}
Let~$K$ satisfy~\eqref{krn_symm} and~\eqref{main_ellipt} and let~$W \in L^{\infty}(\R)$ be such that~$W(\pm1)=0$. Let~$I:=[a,b] \subset \R$ be an interval with length~$|I|= b-a>6$.   
	
Then, there exists a local minimizer~$v_{I}: \R \to [-1,1]$ for~$\Ec$ in~$I$. 

In particular,
\begin{equation}\label{vi_prop}
{v_{I}(x)=-1} \mbox{ if }{x\leq a}\qquad {\mbox{and}}\qquad
{v_{I}(x)=1}  \mbox{ if }{x\geq b}.
\end{equation}
Also, there exists~$C>0$, depending only on~$s$, $\Lambda$ and~$W$, such that
\begin{equation}\label{bound_energy}
    		\Ec(v_{I};J) \leqslant C \, \Psi_s(|J |),
\end{equation}
where~$J$ is either~$I$ or any subinterval of~$I$ with~$|J | >6$.

Here above, $\Psi_s$ is the function introduced in~\eqref{fnc_psi}.
\end{lemma}

In~\cite[Lemma~6.1]{CP16} a proof of this result is presented. Nevertheless, the authors in~\cite{CP16} work with an even potential~$W$, which allows them to retrieve an odd minimizer~$v_I$. Since we dropped this hypothesis on~$W$, we provide here a self contained proof.

\begin{proof}[Proof of Lemma~\ref{min_open_inter}]
For any interval~$I=[a,b]$, we would like to prove the existence of the function~$v_{I}$ by means of Lemma~\ref{en_minimiz}. In this way, we only need to construct a suitable competitor in~$[a,b]$.

Without loss of generality, we suppose that~$a<-3$ and~$b>3$. We consider the function
\begin{equation*}
h(x)= \begin{cases}
-1 &\mbox{if } x \leq -1,\\
x &\mbox{if } x \in (-1,1),\\
1 &\mbox{if } x\geq 1.
\end{cases}
\end{equation*}
By~\cite[Lemma~2]{PSV13} and exploiting that~$|I | >6$, we have that
\begin{equation*}
\begin{split}
\mathscr{F}(h, I) &\leq
\begin{cases}
C_s (1+ |I |^{1-2s}) &\mbox{if } s \in (0,1/2), \\
C_s (1 + \log |I | ) &\mbox{if } s = 1/2,\\
C_s &\mbox{if } s \in (1/2,1)
\end{cases}\\
&\leq 2\, C_s \,\Psi_s(|I |).
\end{split}
\end{equation*}
where~$C_s>1$ is a constant depending only on~$s$ and~$\mathscr{F}$ has been defined in~\eqref{psv_en}.

Thus, taking advantage of~\eqref{en_relation}, we see that
\begin{equation*}
\Ec(h, I) \leq \left( \Lambda +2 \right) C_s \,\Psi_s(|I |).
\end{equation*}
Then, Lemma~\ref{en_minimiz} yields~\eqref{vi_prop} and~\eqref{bound_energy} with~$J=I$
and~$C:=  \left( \Lambda +2 \right) C_s$.

Proposition~\ref{mdss} completes the proof of~\eqref{bound_energy} for any subinterval~$J$ of~$I$ with~$|J |>6$.
\end{proof}

\begin{prop}\label{eupro}
The minimizer~$v_I$ given by
Lemma~\ref{min_open_inter} is non-decreasing.
\end{prop}

\begin{proof}
We remark that for any functions~$w$ and~$z$ such that~$w=z$ outside a set~$J\rq{} \subseteq J$, it holds that
\begin{equation}\label{nevnolab}
\Ec(w,J\rq{})-\Ec(z,J\rq{})= \Ec(w,J)-\Ec(z,J).
\end{equation}

Now, for any~$\tau>0$ we set
\begin{equation*}
u(x):= v_I(x) \qquad\mbox{and}\qquad v(x) := v_I(x + \tau)
\end{equation*}
and recall the setting in~\eqref{wisk}. In view of Lemma~\ref{leaj}, we set~$J\rq{}:= [a-\tau, b-\tau]$ and~$J:=[a-\tau,b]$ and we see that
\begin{equation}\label{ses}
\Ec(m, J) + \Ec(M, J) \leq \Ec(u, J)  + \Ec(v, J).
\end{equation}

Also, since
\begin{equation*}
M(x)= -1 \mbox{ if } x \leq a-\tau \qquad\mbox{and}\qquad M(x)= 1 \mbox{ if } x \geq b-\tau,
\end{equation*}
we exploit~\eqref{nevnolab} and the minimality of~$v$ in~$J\rq{}$ to obtain that
\begin{equation}\label{elapri}
\Ec(M, J) - \Ec(v, J) \geq 0.
\end{equation}

Analogously, one can infer that
\begin{equation*}
\Ec(m, J) - \Ec(u, J) \geq 0.
\end{equation*}
Putting this and~\eqref{elapri} together, we conclude that
\begin{equation*}
\Ec(m, J) + \Ec(M, J) \geq \Ec(u, J)  +\Ec(v, J).
\end{equation*}
This and~\eqref{ses} imply that
\begin{equation*}
\Ec(m, J) + \Ec(M, J) = \Ec(u, J)  + \Ec(v, J).
\end{equation*}
We now exploit formula~\eqref{jjss} of Lemma~\ref{leaj} to deduce that~$u-v$ does not change sign. Namely~$v_I-v_I(\cdot+\tau)$ does not change sign, which means that~$v_I$ is monotone.

The fact that~$v_I$ is non-decreasing follows
from the external conditions~\eqref{vi_prop}.
\end{proof}

\begin{prop}\label{symm}
Let~$v_{[-M,M]}$ be the minimizer given by
Lemma~\ref{min_open_inter}. Assume, in addition, that~\eqref{pot_symm} holds true.

Then, $v_{[-M,M]}$ is an odd function.
\end{prop}

\begin{proof}
We set~$w(x):=-v_{[-M,M]}(-x)$ and we claim that
\begin{equation}\label{ojhvrcc}
\Ec(w,[-M,M]) = \Ec(v_{[-M,M]},[-M,M]).
\end{equation}
Indeed, one can easily check that~$\Hc(w,[-M,M])=\Hc(v_{[-M,M]},[-M,M])$.
Moreover, by~\eqref{pot_symm},
\begin{equation*}
\Pc(w,[-M,M])= \int_{-M}^{M} W(-v(-x))\, dx =\int_{-M}^{M} W(v(-x))\, dx =\Pc(v_{[-M,M]},[-M,M]).
\end{equation*}
Thus, \eqref{ojhvrcc} holds and, in particular, both~$w$ and~$v_{[-M,M]}$ are local minimizers of~$\Ec$ in~$[-M,M]$.

As a consequence,
\begin{eqnarray*}&&
\Ec(\max\{w,v_{[-M,M]}\},[-M,M])+\Ec(\min\{w,v_{[-M,M]}\},[-M,M])\\&&\qquad \geq \Ec(w,[-M,M]) +\Ec(v_{[-M,M]},[-M,M]).
\end{eqnarray*}
On the other hand, by Lemma~\ref{leaj},
\begin{eqnarray*}&&
\Ec(\max\{w,v_{[-M,M]}\},[-M,M])+\Ec(\min\{w,v_{[-M,M]}\},[-M,M]) \\&&\qquad= \Ec(w,[-M,M]) +\Ec(v_{[-M,M]},[-M,M]).
\end{eqnarray*}
Therefore, equality holds and so, by~\eqref{jjss},
either~$w \geq v_{[-M,M]}$ or~$v_{[-M,M]}\ge w$. Now, the definition of~$w$ yields the thesis.
\end{proof}

\begin{corol}\label{alm_min}
Let~$v_{[0,R]}$ be the minimizer given by
Lemma~\ref{min_open_inter}.
Assume, in addition, that~$K$ satisfies~\eqref{nuovissima}.

Then, for any~$\ell>0$, there exists a function~$\alpha_{\ell}: (0,+\infty) \to (0,+\infty)$ satisfying
\begin{equation}\label{gfhdjskyrueityrueigfhdjs0}
\lim_{R \to +\infty} \alpha_{\ell}(R)=0\end{equation} such that,
for any~$\phi \in C_c^{\infty}(-\ell,\ell)$,
\begin{equation}\label{gfhdjskyrueityrueigfhdjs}
\Ec(v_{[0,R]}, [-\ell, \ell]) \leq \Ec(v_{[0,R]}+\phi, [-\ell,\ell])+ \alpha_{\ell}(R).
\end{equation}
\end{corol}

\begin{proof}
Let~$\ell>0$. In this proof we will make use of the following notations:
for all~$R>0$,
\begin{equation*}
\bar{K}_R:= \sup_{x \in \R^n \setminus\{0\}} \frac{K(\frac{xR}{R+2\ell})}{K(x)}
\qquad{\mbox{and}}\qquad
C_R := \bar{K}_R\left(\frac{R}{R+2\ell} \right)^2 .
\end{equation*}
We recall that~$\Psi_s$ is the function introduced in~\eqref{fnc_psi} and we define the functions
\[\widetilde{\beta}_{\ell}(R):= 2 \left( \max\left\{C_R, \frac{R}{R+2\ell}\right\}-1 \right)\Psi_s(R)\qquad{\mbox{and}}\qquad \beta_{\ell}(R): = \max \big\{ \widetilde{\beta}_{\ell}(R),0 \big\}.\]
Also, let~$s\in (0,1)$ be as in~\eqref{main_ellipt} and set
\begin{equation*}
\gamma_{\ell}(R) :=
\begin{cases}
\displaystyle 2\Lambda \log\left( 1 +\frac{{2}\ell}{R} \right)
&\mbox{if } s = 1/2 ,\\
\displaystyle 
\frac{\Lambda}{s(1-2s)} \big( (R+2\ell)^{1-2s}-R^{1-2s} \big)&\mbox{if } s \neq 1/2 .
\end{cases}
\end{equation*}
We notice that
 \begin{equation}\label{mnbvcx12wsxdr5tgbhuuj}
\lim_{R\to +\infty}\gamma_{\ell}(R)=0.
\end{equation}

Now, we claim that, for any~$\epsilon \in (0,1)$,
\begin{equation}\label{bdbdd}
\limsup_{R\to +\infty} (C_R -1) R^{1-\epsilon} \leq 0.
\end{equation}
Indeed, we observe that~$\bar{K}_R \in [0,+\infty]$. Therefore, if~$R$ is sufficiently large, possibly in dependence of~$\ell$,
\begin{eqnarray}\label{gyureihrje6574854fhsdjfgwygry3i}
(C_R -1) R^{1-\epsilon} & = & \left(\bar{K}_R\left(1-\frac{2\ell}{R+2\ell} \right)^2 -1\right) R^{1-\epsilon} \notag\\
& = & \left(  \bar{K}_R \left( 1-\frac{4\ell}{R+2\ell}+\frac{4\ell^2}{(R+2\ell)^{2}}\right) -1 \right) R^{1-\epsilon} \notag\\
& = & \left( \bar{K}_R -1 \right) R^{1-\epsilon} - \frac{4\ell \bar{K}_R R^{1-\epsilon}}{R+2\ell}
\left( 1 -\frac{\ell}{R+2\ell}\right)\notag\\
& \leq &  \left( \bar{K}_R -1 \right) R^{1-\epsilon}\notag\\
& \le & \left( \bar{K}_R -1 \right) (R+2\ell)^{1-\epsilon} .
\end{eqnarray}
We point out that
\begin{eqnarray*}
\left( \bar{K}_R -1 \right)(R+2\ell)^{1-\epsilon}=(2\ell)^{1-\epsilon}
\left( \sup_{x \in \R^n \setminus\{0\}} \frac{K(\frac{xR}{R+2\ell})}{K(x)}-1\right)
\left(\frac{R+2\ell}{2\ell}\right)^{1-\epsilon},
\end{eqnarray*}
and thus, using~\eqref{nuovissima} with~$\sigma_j$ replaced by~$R/(R+2\ell)$, we have that
\begin{equation*}
\begin{split}
\limsup_{R\to +\infty} \left( \bar{K}_R -1 \right)(R+2\ell)^{1-\epsilon}
&=(2\ell)^{1-\epsilon}\limsup_{R\to +\infty}
\left( \sup_{x \in  \R^n \setminus\{0\}} \frac{K(\frac{xR}{R+2\ell})}{K(x)}-1\right)
\left(\frac{R+2\ell}{2\ell}\right)^{1-\epsilon}\le 0
.
\end{split}
\end{equation*}
This and~\eqref{gyureihrje6574854fhsdjfgwygry3i} give~\eqref{bdbdd}.

{F}rom~\eqref{bdbdd}, we also find that
\begin{eqnarray*}&&
\limsup_{R\to +\infty}\widetilde{\beta}_{\ell}(R)
=\limsup_{R\to +\infty} 2\max\left\{C_R-1, -\frac{2\ell}{R+2\ell}\right\}\Psi_s(R)\\&&\qquad= \limsup_{R\to +\infty}2
\max\left\{(C_R-1)R^{1-\epsilon}, -\frac{2\ell R^{1-\epsilon}}{R+2\ell}\right\} \frac{\varPsi_s(R)}{R^{1-\epsilon}}\le 0,
\end{eqnarray*} and, as a result,
\begin{equation}\label{ovvr}
\lim_{R \to +\infty}\beta_{\ell}(R) = 0.
\end{equation}

Now, we let~$v_{[-\ell, R+\ell]}$ be  a local minimizer
for~$\Ec$ in~$[-\ell, R+\ell]$, as given by
Lemma~\ref{min_open_inter}, and we
define the function
$$z_{R}(x):= v_{[-\ell, R+\ell]}\left( \frac{(R+2\ell)x}{R}-\ell \right).$$
We observe that~$z_{R} (x)=-1$ if~$x \leq 0$ and~$z_{R}(x)=1$ if~$x \geq R$, thanks to the properties of~$v_{[-\ell, R+\ell]}$
in~\eqref{vi_prop}.
Also, by minimality,
\begin{equation*}
\Ec(v_{[0,R]}, [0,R]) \leq \Ec(z_{R}, [0,R]).
\end{equation*}

Moreover, we claim that
\begin{equation}\label{now_comp}
\Ec(v_{[0,R]}, [0,R]) \leq \Ec(v_{[-\ell,R+\ell]}, [-\ell,R+\ell])+ \beta_{\ell}(R).
\end{equation}
In the aim of proving~\eqref{now_comp}, we notice that
\begin{eqnarray*}
&&\Hc(z_R, [0,R]) \le  C_R \Hc(v_{[-\ell, R+\ell]}, [-\ell, R+\ell])
\\{\mbox{and}}\quad &&
\Pc(z_R, [0,R]) =\frac{R}{R+2\ell}\Pc(v_{[-\ell, R+\ell]}, [-\ell, R+\ell]).
\end{eqnarray*}
As a consequence of this and formula~\eqref{bound_energy}
in Lemma~\ref{min_open_inter}, we obtain that
\begin{equation*}
\begin{split}
&\Ec(z_R, [0,R]) 
\le  \max\left\{C_R, \frac{R}{R+2\ell}\right\}\Ec(v_{[-\ell, R+\ell]}, [-\ell, R+\ell])
\\
&= \Ec(v_{[-\ell, R+\ell]}, [-\ell, R+\ell]) + \left( \max\left\{C_R, \frac{R}{R+2\ell}\right\}-1 \right)\Ec(v_{[-\ell, R+\ell]}, [-\ell, R+\ell]) \\
&\leq \Ec(v_{[-\ell, R+\ell]}, [-\ell, R+\ell]) + 2 \left( \max\left\{C_R, \frac{R}{R+2\ell}\right\}-1 \right)\Psi_s(R) \\
&=\Ec(v_{[-\ell, R+\ell]}, [-\ell, R+\ell]) + \beta_{\ell}(R),
\end{split}
\end{equation*}
which is~\eqref{now_comp}.

Furthermore, we claim that
\begin{equation}\label{sub_gam}
0 \leq \Ec (v_{[0,R]}, [-\ell, R+\ell])- \Ec (v_{[0,R]}, [0,R])\leq  \gamma_{\ell}(R).
\end{equation}
The left-hand inequality in~\eqref{sub_gam}
follows by the minimality of~$v_{[0,R]}$ and
Remark~\ref{remmarco}.

We now focus on the proof of the right-hand inequality in~\eqref{sub_gam}.
Exploiting the properties of~$v_{[0,R]}$ and
the assumptions on the kernel and the potential in~\eqref{krn_symm}, \eqref{main_ellipt} and~\eqref{pot_deg}, we find that
\begin{equation*}
\Pc(v_{[0,R]},[-\ell,R+\ell])= \Pc(v_{[0,R]},[0,R])
\end{equation*}
and that
\begin{equation*}
\begin{split}
&\Hc (v_{[0,R]}, [-\ell, R+\ell])- \Hc (v_{[0,R]}, [0,R]) \\
&\le \frac{1}{2} \bigg( \int_{-\ell}^0 \int_{R}^{+\infty}| v_{[0,R]}(x)-v_{[0,R]}(y)|^2 K(x-y) \, dy \, dx + \int_R^{R+\ell} \int_{-\infty}^{-\ell} | v_{[0,R]}(x)-v_{[0,R]}(y)|^2 K(x-y) \, dy \, dx \bigg) \\
&\leq  \frac{\Lambda}{2}\bigg( \int_{-\ell}^0 \int_{R}^{+\infty}\frac{| v_{[0,R]}(x)-v_{[0,R]}(y)|^2}{|x-y|^{1+2s}} \, dy \, dx + \int_R^{R+\ell} \int_{-\infty}^{-\ell} \frac{| v_{[0,R]}(x)-v_{[0,R]}(y)|^2}{|x-y|^{1+2s}} \, dy \, dx \bigg) .
\end{split}
\end{equation*}
In particular, when~$s \in(0,1) \setminus \{1/2\}$,
\begin{equation*}
\begin{split}
&\int_{-\ell}^0 \int_R^{+\infty}\frac{| v_{[0,R]}(x)-v_{[0,R]}(y)|^2}{|x-y|^{1+2s}} \, dy \, dx \leq 4 \int_{-\ell}^0 \int_R^{+\infty}(y-x)^{-(1+2s)} \, dy \, dx \\
&\qquad=\frac{2}{s} \int_{-\ell}^0 (R-x)^{-2s} \, dx= \frac{2}{s(1-2s)} \big( (R+\ell)^{1-2s}-R^{1-2s}\big)\,,
\end{split}
\end{equation*}
and
\begin{equation*}
\begin{split}
&\int_{R}^{R+\ell}\int_{-\infty}^{-\ell}\frac{| v_{[0,R]}(x)-v_{[0,R]}(y)|^2}{|x-y|^{1+2s}} \, dy \, dx \leq 4\int_{R}^{R+\ell}\int_{-\infty}^{-\ell}(x-y)^{-(1+2s)} \, dy \, dx \\
&\qquad=\frac{2}{s} \int_R^{R+\ell}(x+\ell)^{-2s}\, dx = \frac{2}{s(1-2s)} \big( (R+2\ell)^{1-2s}-(R+\ell)^{1-2s}\big)\,,
\end{split}
\end{equation*}
so that
\begin{eqnarray*}
&& \int_{-\ell}^0 \int_R^{+\infty}\frac{| v_{[0,R]}(x)-v_{[0,R]}(y)|^2}{|x-y|^{1+2s}} \, dy \, dx + \int_{R}^{R+\ell}\int_{-\infty}^{-\ell}\frac{| v_{[0,R]}(x)-v_{[0,R]}(y)|^2}{|x-y|^{1+2s}} \, dy \, dx \\
&&\quad  \leq \frac{2}{s(1-2s)} \big( (R+2\ell)^{1-2s}-(R+\ell)^{1-2s}\big) +\frac{2}{s(1-2s)} \big( (R+\ell)^{1-2s}-R^{1-2s}\big)\\
&& \quad = \frac{2\gamma_{\ell}(R)}{\Lambda}
\end{eqnarray*}
In the same spirit, when~$s=1/2$, we see that
\begin{equation*}
\int_{-\ell}^0 \int_R^{+\infty}\frac{|v_{[0,R]}(x)-v_{[0,R]}(y)|^2}{|x-y|^{2}} \, dy \, dx \leq 4 \int_{-\ell}^0 (R-x)^{-1} \, dx = 4 \log(R+\ell) -4\log(R)
\end{equation*}
and
\begin{equation*}
\int_{R}^{R+\ell} \int_{-\infty}^{-\ell}\frac{|v_{[0,R]}(x)-v_{[0,R]}(y)|^2}{|x-y|^{2}} \, dy \, dx \leq 4 \int_{R}^{R+\ell} (y+\ell)^{-1} \, dy = 4\log(R+2\ell)-4\log(R+\ell)\,,
\end{equation*}
so that
\begin{eqnarray*}
	&& \int_{-\ell}^0 \int_R^{+\infty}\frac{| v_{[0,R]}(x)-v_{[0,R]}(y)|^2}{|x-y|^{1+2s}} \, dy \, dx + \int_{R}^{R+\ell}\int_{-\infty}^{-\ell}\frac{| v_{[0,R]}(x)-v_{[0,R]}(y)|^2}{|x-y|^{1+2s}} \, dy \, dx \\*[1ex]
	&&\quad  \leq 4\log(R+2\ell)-4\log(R+\ell) + 4 \log(R+\ell) -4\log(R) \\
	&& \quad = \frac{2\gamma_{\ell}(R)}{\Lambda}
\end{eqnarray*}

These considerations establish~\eqref{sub_gam}.

As a consequence of~\eqref{now_comp}, \eqref{sub_gam} and the minimality of~$v_{[-\ell,R+\ell]}$, we have that
\begin{equation}\label{mli}\begin{split}
&\Ec (v_{[0,R]}, [-\ell,R+\ell]) - \beta_{\ell}(R)-\gamma_{\ell}(R) 
\le  \Ec (v_{[0,R]}, [0,R])- \beta_{\ell}(R)
\\&\qquad\le \Ec(v_{[-\ell,R+\ell]}, [-\ell,R+\ell])
\leq \Ec(v_{[0,R]}+\phi,[-\ell,R+\ell] ),
\end{split}
\end{equation}
for any~~$\phi \in C^{\infty}_c(-\ell,\ell)$. 

Also, exploiting the fact that~$\phi$ is supported in~$[-\ell,\ell]$, we see that
\begin{align*}
&\Ec(v_{[0,R]}+\phi, [-\ell,R+\ell])- \Ec(v_{[0,R]}+\phi, [-\ell,\ell]) \\
&= \frac{1}{4} \int_{\ell}^{R+\ell}\int_{\ell}^{R+\ell}|v_{[0,R]}(x)-v_{[0,R]}(y)|^2 K(x-y)\, dx\,dy \\
&\,+ \frac{1}{2} \int_{\ell}^{R+\ell}\int_{(
(-\infty,-\ell)\cup(R+\ell,+\infty))} |v_{[0,R]}(x)-v_{[0,R]}(y)|^2 K(x-y)\, dx\,dy + \int_{\ell}^{R}W\big(v_{[0,R]}(x)\big)\, dx \\
&\le \Ec(v_{[0,R]}, [-\ell,R+\ell])- \Ec(v_{[0,R]}, [-\ell,\ell]).
\end{align*}
By plugging this inequality into~\eqref{mli}, we obtain the desired result
in~\eqref{gfhdjskyrueityrueigfhdjs}
with~$\alpha_{\ell}(R):= \beta_{\ell}(R)+\gamma_{\ell}(R)$.

{F}rom~\eqref{mnbvcx12wsxdr5tgbhuuj} and~\eqref{ovvr}, we also obtain the limit property in~\eqref{gfhdjskyrueityrueigfhdjs0}.
\end{proof}

\begin{prop}\label{comepo}
Let~$v_{[0,R]}$ be the minimizer given by
Lemma~\ref{min_open_inter}.
Assume, in addition, that~$K$ satisfies~\eqref{nuovissima}.

Then, the function~$v_{[0, R]}$ converges to~$-1$ locally unformly as~$R\to +\infty$.
\end{prop}

\begin{proof}
We observe that, by minimality, the function~$v_{[0,R]}$ is a weak solution of
\begin{equation*}
{L}_K v_{[0,R]}= W'(v_{[0,R]}) \quad\mbox{in } (0,R).
\end{equation*}
In particular, in view of Proposition~\ref{reg_dirichlet_pbm}, we have that~$v_{[0,R]} \in C^{\bar{\theta}}(\R) \cap C^{2s+\bar{\theta}}(0,R)$ for  some~${\bar{\theta}} \in (0,s)$ and its
H\"{o}lder norm is bounded independently of~$R$ (see\footnote{
A careful analysis of the proof of~\cite[Proposition~1.1]{RS14} shows that the H\"{o}lder norm of the solution to the Dirichlet problem~\eqref{dirichlet_pbm} is bounded by a constant that is independent from~$\Omega$ as a whole, but only depends on the~$C^{1,1}$ norm of its boundary. Then, since in our case~$n=1$, this constant is given regardless from the value of~$R$. 
} \cite[Proposition~1.1]{RS14}). 
Therefore, by means of the Ascoli-Arzel\'a Theorem, there exist~$\theta\in(0,\bar{\theta})$ and a function~$v\in C^{2s+\theta}_{\rm loc}(0,+\infty) \cap L^{\infty}(\R)$ such that, up to subsequences, $v_{[0,R]}\to v$ as~$R \to +\infty$ locally uniformly in~$\R$. 

Also, we claim that~$v$ is a pointwise solution of
\begin{equation}\label{vveif}
{L}_Kv = W\rq{}(v) \quad\mbox{in} \ (0,+\infty).
\end{equation}
In the aim of showing~\eqref{vveif}, we set an arbitrary open set~$(a,b) \Subset (0,+\infty)$. Then, we obtain by~\eqref{resis}, Remark~\ref{remmarco} and Proposition~\ref{mdss} that
\begin{equation*}
\begin{split}&
[v_{[0,R]}]_{H^K(a,b)}^2 = 2 \Hc\left(v_{[0,R]},(a,b)\right) \leq 2 \Hc\left(v_{[0,R]},(a,b+6)\right) \\
&\qquad\leq 2 \Ec\left(v_{[0,R]},(a,b+6)\right) \leq C\Psi_s(b-a+6).
\end{split}
\end{equation*}
Consequently, we can apply Lemma~\ref{eq_stab} and obtain that~$v \in \Hc(a,b) \cap L^{\infty}(\R)$ is a weak solution of~\eqref{vveif}. Then, Proposition~\ref{reg_dirichlet_pbm} gives that~$v$ is a pointwise solution, and therefore
the claim~\eqref{vveif} is established.

Now we show that, for any~$\ell>0$,
\begin{equation} \label{nunfie}
\lim_{R\to +\infty} \Ec (v_{[0,R]},[-\ell,\ell]) = \Ec (v,[-\ell,\ell]).
\end{equation}
For this, we observe that~\eqref{main_ellipt} and the uniform boundedness of~$v_{[0,R]}$ imply that
\begin{equation*}
\left( | v_{[0,R]}(x)- v_{[0,R]}(y) |^2 K(x-y) \right) \in L^1([-\ell,\ell] \times(\R\setminus[-2\ell, 2\ell])).
\end{equation*}
Consequently, by the Dominated Convergence Theorem,
\begin{equation*}\begin{split}&
\lim_{R\to +\infty} \iint_{[-\ell,\ell]\cup (\R\setminus[-2\ell, 2\ell])}
|v_{[0,R]}(x)-v_{[0,R]}(y)|^2 K(x-y) \, dx\, dy\\&\qquad = 
\iint_{[-\ell,\ell]\cup (\R\setminus[-2\ell, 2\ell])}|v(x)-v(y)|^2 K(x-y) \, dx\, dy.
\end{split}
\end{equation*}
{F}rom this and the uniform convergence of~$v_{[0,R]}$
on compact sets of~$\R^n$,
we obtain that
\begin{equation}\label{685943asxcfrtgbhy78ijmko}
\lim_{R\to +\infty} \Hc (v_{[0,R]},[-\ell,\ell]) = \Hc (v,[-\ell,\ell]).
\end{equation}

Moreover, the regularity of the potential~$W$ implies
\begin{equation*}
\lim_{R\to +\infty} \Pc(v_{[0,R]},[-\ell,\ell]) = \Pc (v,[-\ell,\ell]).
\end{equation*}
This and~\eqref{685943asxcfrtgbhy78ijmko}
give~\eqref{nunfie}.

Similarly, one obtains that, for any~$\phi \in C_c^{\infty}(-\ell,\ell)$,
\begin{equation}\label{dnf}
\lim_{R\to +\infty} \Ec (v_{[0,R]}+\phi,[-\ell,\ell]) = \Ec (v+\phi,[-\ell,\ell]).
\end{equation}

Accordingly, from~\eqref{nunfie}, \eqref{dnf}
and Corollary~\ref{alm_min}, we deduce that~$v$
is a class~${A}$ minimizer for~$\Ec$.
In particular, this implies that~$v$ solves
\begin{equation*}
{L}_K v= W'(v) \quad\mbox{in } \R.
\end{equation*}
Thus, since~$v(0)=-1$ and~$W\rq{}(-1)=0$, we can exploit~\eqref{main_ellipt} 
and gather that
\begin{equation*}
0= {L}_K v (0) = \int_{\R}( v(y)+1) K(y)\, dy \geq \lambda \, \int_{-r_0}^{r_0} \frac{( v(y)+1)}{|y|^{1+2s}}\, dy \geq 0.
\end{equation*}
As a consequence, $v =-1$ in~$(-\infty,r_0)$. 

Hence, we can write
\begin{equation*}
0=  {L}_K v (r_0) =  \int_{\R}( v(y)+1) K(r_0-y)\, dy \geq \lambda \,  \int_{0}^{2r_0} \frac{( v(y)+1)}{|r_0-y|^{1+2s}}\, dy \geq 0,
\end{equation*}
so that~$v =-1$ in~$(-\infty,2r_0)$.

Repeating this argument iteratively yields that~$v \equiv -1$ in~$\R$,
wich completes the proof of Proposition~\ref{comepo}.
\end{proof}

\begin{rem}\label{remm}
{\rm
We point out that a similar statement to the one in Proposition~\ref{comepo} holds true for~$v_{[-R,0]}$. In this case,
one obtains that the function~$v_{[-R,0]}$ converges locally uniformly to~$1$ as~$R\to +\infty$. 
}
\end{rem}

     \section{A useful lemma}\label{usflemma}
     
 In this section we present a result that will be useful in the proofs of Theorems~\ref{main_thm} and~\ref{main_thm_symm}.   
   
   We recall the definition of the
   renormalized energy~$\mathscr{G}$ in~\eqref{G_star}
   and of
   the space~${\mathcal{X}}$ in~\eqref{defmathcalixs00}, and we define
     \begin{equation}\label{emme}
     \mathscr{M}:= \big\{ u \in \mathcal{X}\;{\mbox{ s. ~\!t. }}\; \mathscr{G}(u)<+\infty \text{ and~$u$ is a class~A minimizer of} \ \Ec\big\}.
\end{equation}
In addition, for any~$x_0\in \R$, we define the following subset of~$\mathscr{M}$
\begin{equation}\label{emmezero}
 \mathscr{M}^{(x_0)} := \big\{u \in \mathscr{M}\;{\mbox{ s. ~\!t. }}\;x_0= \sup\{x\in \R: u(x)<0\}\big\}.
\end{equation}

\begin{lemma}\label{psc0-500}
Let~\eqref{krn_symm}, \eqref{main_ellipt}, \eqref{pot_reg}, \eqref{pot_zero} and~\eqref{pot_deg} be satisfied. Let~$u^{(0)}$ be a non-decreasing function such that~$u^{(0)} \in \mathscr{M}^{(0)}$. 

Then, for any~$x_0 \in \R$, the function~$u^{(x_0)}:= u^{(0)}(x-x_0)$ is such that
\begin{enumerate}[(i)]
\item $u^{(x_0)} \in \mathscr{M}^{(x_0)}$, \label{iunoi}
\item $u^{(x_0)}$ is strictly increasing, \label{bigi}
\item $\mathscr{M}^{(x_0)}$ is a singleton, in particular~$\mathscr{M}^{(x_0)}= \{ u^{(x_0)}\}$, \label{iiunoii}
\item $u^{(x_0)}$ satisfies the decay estimates in~\eqref{asymp_decay}. \label{9bbf0}
\end{enumerate}

Moreover, if~\eqref{newbound} holds, then
\begin{enumerate}[(i)]
\setcounter{enumi}{4}
\item $u^{(x_0)}$  satisfies the decay estimates in~\eqref{asymp_decay_lowbound} \ \mbox{and~\eqref{398r7gree}}. \label{9873grv}
\end{enumerate}

In addition, if both~\eqref{newbound} and~\eqref{ricdifar} hold, then
\begin{enumerate}[(i)]
\setcounter{enumi}{5}
\item $u^{(x_0)}$  satisfies the decay estimates in~\eqref{eq:asymp-derivata}.\label{iiiiunoiiii}
\end{enumerate}
\end{lemma}

Being quite long, the proof of Lemma~\ref{psc0-500} is divided in five separate subproofs.

\begin{proof}[Proof of statements~\eqref{iunoi} and~\eqref{bigi} of Lemma~\eqref{psc0-500}]
Let us recall that, by minimality, $u^{(0)}$ is a weak solution of 
\begin{equation}\label{miser45}
{L}_K u^{(0)} = W\rq{}(u^{(0)}) \quad\mbox{in } \R.
\end{equation}
Thus, $u^{(0)}: \R \to [-1,1]$ is such that
\begin{equation}\label{propert45}\begin{split}&
{\mbox{$u^{(0)}\in C^{1+2s+\theta}(\R)\cap L^{\infty}(\R)$
for some~$\theta\in(0,1)$,}}\\&{\mbox{$u^{(0)}$ is non-decreasing and~$u^{(0)}(0)=0$,}}\end{split}
\end{equation}
where the regularity follows from Proposition~\ref{reg_entire_sol}. 

Now, \eqref{iunoi} is a consequence of the translation invariance.

Moreover, since~\eqref{krn_symm} and~\eqref{main_ellipt} hold true, we can make use of~\cite[Lemma~4.6]{CP16} and infer that~$u^{(0)}$ is strictly increasing, which is~\eqref{bigi}.
\end{proof}

We mention here that~$u^{(x_0)} \in \mathcal{X}$ and~\eqref{bigi}  implies that
\begin{equation}\label{big}
|u^{(0)}(x) | < 1 \quad\mbox{for any} \ x \in \R.
\end{equation}

\begin{proof}[Proof of statement~\eqref{iiunoii} of Lemma~\eqref{psc0-500}]
Let~$u \in \mathscr{M}^{(x_0)}$.
In order to establish~\eqref{iiunoii}, we prove that
\begin{equation}\label{aAGTTET54y547u65uJK}
u \equiv u^{(x_0)}.\end{equation}

Since~$u$ is a weak solution of~\eqref{miser45}, by Proposition~\ref{reg_entire_sol} we have that~$u \in C^{1+2s+\theta}(\R)$ for some~$\theta>0$.

Also, we know that~$|u | \leq 1$ in~$\R$. Thus, for any~$\epsilon\in(0,1)$, we can find~$k(\epsilon)\in\R$ such that, for any~$k \in [k(\epsilon),+\infty)$,
\begin{equation*}
u(x+k)+\epsilon > u^{(x_0)}(x) \quad\mbox{for any } x \in \R.
\end{equation*}
Now, we take~$k$ as small as possible with this property; that is, we take~$k_{\epsilon}$ such that
\begin{equation}\label{nonlab}
u(x+k_{\epsilon})+\epsilon > u^{(x_0)}(x) 
\quad\mbox{for any } x \in \R
\end{equation}
and there exist a sequence~$\eta_{j,\epsilon}\in[ 0,1)$ and  points~$x_{j,\epsilon} \in \R$ satisfying
\begin{equation}\label{nonlabt}
\lim_{j \to +\infty}\eta_{j,\epsilon}=0
\end{equation}
and
\begin{equation}\label{nonlabtt}
u(x_{j,\epsilon}+(k_{\epsilon}-\eta_{j,\epsilon})) +\epsilon \leq  u^{(x_0)}(x_{j,\epsilon})\quad\mbox{for any }j \in \N.
\end{equation}

We point out that~$x_{j,\epsilon}$ must be bounded in~$j$. Otherwise, if
\begin{equation*}
\lim_{j \to  +\infty} x_{j, \epsilon}= \pm \infty,
\end{equation*}
we would have by~\eqref{nonlabt} and~\eqref{nonlabtt} that
\begin{equation*}
\pm 1+\epsilon = \lim_{j \to + \infty} u(x_{j,\epsilon}+(k_{\epsilon} -\eta_{j,\epsilon})) +\epsilon \leq  \lim_{j \to + \infty} u^{(x_0)}(x_{j,\epsilon}) = \pm 1,
\end{equation*}
which gives a contradiction. 

As a consequence, there exists~$x_{\epsilon} \in \R$ such that, up to a subsequence, $$\lim_{j \to +\infty} x_{j,\epsilon}=x_{\epsilon}.$$ Accordingly, setting~$u_{\epsilon}(x):= u(x+k_{\epsilon})+\epsilon$,
formulas~\eqref{nonlab} and~\eqref{nonlabtt} give that
\begin{equation}\label{eqfo}
u_{\epsilon}(x) \geq u^{(x_0)}(x) \;\mbox{ for any } x \in \R \qquad
{\mbox{and}}\qquad
u_{\epsilon}(x_{\epsilon})= u^{(x_0)}(x_{\epsilon}) .
\end{equation}
Also, by translation invariance, we see that
\begin{equation*}
{L}_K u_{\epsilon}(x) = W\rq{}(u_{\epsilon}(x)-\epsilon),
\end{equation*}
and therefore, by~\eqref{eqfo},
\begin{equation}\label{questdt}
\begin{split}&
W\rq{}\left( u^{(x_0)}(x_{\epsilon})-\epsilon \right) - W\rq{}\left( u^{(x_0)}(x_{\epsilon}) \right) = W\rq{}\left( u_{\epsilon}(x_{\epsilon})-\epsilon \right) -  W\rq{}\left( u^{(x_0)}(x_{\epsilon}) \right)\\ 
&\qquad=  {L}_K u_{\epsilon}(x_{\epsilon}) - {L}_K u^{(x_0)}(x_{\epsilon})
= \int_{\R} \left(u_{\epsilon}(y)-u^{(x_0)}(y)\right) K(x_{\epsilon}-y) \, dy \geq 0.
\end{split}
\end{equation}

Now, we claim that
\begin{equation}\label{clami}
{\mbox{$x_{\epsilon}$ is bounded uniformly in~$\epsilon$.}}
\end{equation}
For this, we argue by contradiction and suppose that, up to subsequences, $x_{\epsilon}\to +\infty$ as~$\epsilon \searrow 0$ (the contradiction in case~$x_{\epsilon}\to -\infty$ would be obtained through a symmetric argument).
Then, there exists~$\bar{\epsilon}\in(0,1)$ such that, for any~$\epsilon \in (0,\bar{\epsilon})$,
\begin{equation*}
u^{(x_0)}(x_{\epsilon}) \geq 1-\frac{\xi}{2},
\end{equation*}
where~$\xi$ is given in~\eqref{pot_deg}. Therefore, for any~$\epsilon \in (0, \min\{\bar{\epsilon}, \xi/2 \})$,
\begin{equation*}
1-\xi \leq u^{(x_0)}(x_{\epsilon})-\epsilon \leq u^{(x_0)}(x_{\epsilon}) \leq 1.
\end{equation*}
In this way, we can exploit Lemma~\ref{thc_deg} with~$t:=u^{(x_0)}(x_{\epsilon})$ and~$r:= u^{(x_0)}(x_{\epsilon})-\epsilon$ and obtain that
\begin{equation*}
\begin{split}
W\rq{}(u^{(x_0)}(x_{\epsilon})) &\geq W\rq{}(u^{(x_0)}(x_{\epsilon})-\epsilon) + \frac{C_3}{\gamma-1} \left[ \left(1-u^{(x_0)}(x_{\epsilon})+\epsilon\right)^{\gamma-1} - \left(1-u^{(x_0)}(x_{\epsilon}) \right)^{\gamma-1} \right]\\
&> W\rq{}(u^{(x_0)}(x_{\epsilon})-\epsilon).
\end{split}
\end{equation*}
This is in contradiction with~\eqref{questdt}, 
and therefore the claim in~\eqref{clami} is established. 

Now, by~\eqref{clami}, we have that, up to subsequences,
\begin{equation}\label{clami_t}
\lim_{\epsilon \to 0} x_{\epsilon} = \bar{x}\in\R.
\end{equation}

Furthermore, we show that there exists~$\bar{k}\in \R$ such that, up to subsequences,
\begin{equation}\label{k_conv}
\lim_{\epsilon \to 0} k_{\epsilon} = \bar{k}.
\end{equation}
Indeed, if~$k_{\epsilon}\to\pm \infty$, then it would follow from~\eqref{eqfo} and~\eqref{clami_t} that
\begin{equation*}
\pm 1=\lim_{\epsilon\to 0}  u(x_{\epsilon}+k_{\epsilon})+\epsilon= \lim_{\epsilon\to 0} u^{(x_0)}(x_{\epsilon}) =  u^{(x_0)}(\bar{x}),
\end{equation*}
which is not the case since~$u^{(x_0)}(\bar{x}) \in (-1,1)$ by~\eqref{big}. This proves~\eqref{k_conv}.

We can then take the limit as~$\epsilon\searrow0$ in~\eqref{eqfo}
and use~\eqref{clami_t} and~\eqref{k_conv} to obtain that
\begin{equation}\label{kexpv}
\begin{cases}
u(x+\bar{k})\geq u^{(x_0)}(x)  &\text{for any} \ x \in \R,\\
u(\bar{x}+\bar{k})= u^{(x_0)}(\bar{x}).
\end{cases}
\end{equation}

Now, we claim that
\begin{equation}\label{dxgide}
u(x+\bar{k})= u^{(x_0)}(x) \quad\mbox{for all } x\in \R.
\end{equation}
In order to show this, we define the function~$v:=u(\cdot +\bar{k})- u^{(x_0)}$ and observe that, by~\eqref{kexpv}, it holds that~$v\geq 0$ in~$\R$ and~$v(\bar{x})=0$. Thus, since both~$u(\cdot +\bar{k})$ and~$u^{(x_0)}$ belong to~$\mathscr{M}$, we have that
$$
{L}_K v (\bar{x})= W\rq{}(u(\bar{x}+\bar{k})) -W\rq{}(u^{(x_0)}(\bar{x}))=0.
$$
Therefore, exploiting~\eqref{main_ellipt},
\begin{equation*}
0= {L}_K v(\bar{x}) = \int_{\R} v(y) K(\bar{x}-y) \, dy \geq \lambda \, \int_{\bar{x}-r_0}^{\bar{x}+r_0} \frac{v(y)}{| \bar{x}- y |^{1+2s}}\, dy \geq 0 .
\end{equation*}
As a consequence, $u(x+\bar{k})=u^{(x_0)}(x)$ for all~$x\in[\bar{x}-r_0, \bar{x}+r_0]$.

Then, we see that
\begin{equation*}
0= {L}_K v(\bar{x}+r_0) = \int_{\R} v(y) K(\bar{x}+r_0-y) \, dy \geq \lambda \, \int_{\bar{x}}^{\bar{x}+2r_0} \frac{v(y)}{| \bar{x}+r_0- y |^{1+2s}}\, dy \geq 0 
\end{equation*}
and
\begin{equation*}
0= {L}_K v(\bar{x}-r_0) = \int_{\R} v(y) K(\bar{x}-r_0-y) \, dy \geq \lambda \, \int_{\bar{x}-2r_0}^{\bar{x}} \frac{v(y)}{|\bar{x}-r_0- y |^{1+2s}}\, dy \geq 0.
\end{equation*}
Therefore, $u(x+\bar{k})=u^{(x_0)}(x)$ in~$[\bar{x}-2r_0, \bar{x}+2r_0]$.

Repeating this argument iteratively yields~\eqref{dxgide}.

Now, since~$u \in \mathscr{M}^{(x_0)}$ and is continuous, it must be that~$u(x_0)=0$. 
Moreover, by~\eqref{dxgide},
$$  u(x_0+\bar{k})= u^{(x_0)}(x_0)=u^{(0)}(0)=0. $$ 
Therefore,
$$ u^{(x_0)}(x_0)=u(x_0+\bar{k})=u(x_0)=u^{(x_0)}(x_0-\bar{k}).
$$
Since~$u^{(x_0)}$ is strictly increasing (thanks to~\eqref{bigi}),
this implies that~$\bar{k}=0$.
Plugging this information into~\eqref{dxgide},
we obtain~\eqref{aAGTTET54y547u65uJK}, as desired.
\end{proof}

\begin{proof}[Proof of statement~\eqref{9bbf0} of Lemma~\eqref{psc0-500}]
For simplicity, we show that~\eqref{9bbf0} holds for~$u^{(0)}$
(as this implies the estimate for~$u^{(x_0)}$ for any~$x_0\in\R$). 

We recall the notation in~\eqref{pot_deg} and we set
\begin{equation}\label{vttv4869302}
 c:= \min\left\{ \frac{C_1}{\alpha-1}, \xi \right\}.
\end{equation}
We exploit Proposition~\ref{prop_bar} with~$\zeta:=c$ and~$ m:= \alpha$, and we consider the barrier~$w$ constructed there.
 
{F}rom the properties of~$w$ in~\eqref{barrier} and~\eqref{bar_outside_interval} and the fact that~$u^{(0)} \in \mathcal{X}$, we deduce that there exists~$\kappa \in (0,+\infty)$ such that~$w(x+k) > u^{(0)}(x)$ for all~$x \in \R$ and all~$k \in [\kappa,+\infty)$. In particular, we can choose~$\bar{k}\leq \kappa$, an infinitesimal sequence~$\eta_j \in [0,1)$ and points~$x_j \in \R$ such that
    	\begin{align}
&w(x+k) >u^{(0)}(x) \qquad \mbox{for all } k> \bar{k} \mbox{ and any } x \in \R, \label{maj} \\
{\mbox{and }}\quad
&w(x_j+\bar{k}-\eta_j)\leq u^{(0)}(x_j) \qquad \mbox{for any } j. \label{min}
\end{align}
In particular, this and~\eqref{big} give that 
\[|w(x_j+\bar{k}-\eta_j)| <1. \]

As a consequence, for any~$R$ sufficiently large, as given by
Proposition~\ref{prop_bar},
\begin{equation}\label{fus}
 \big|x_j+\bar{k}-\eta_j\big| < R.
\end{equation} 
This gives that
\[|x_j| \leq R+ |\bar{k}|+1,\]
namely, $x_j$ is bounded uniformly in~$j$,
and therefore there exists~$\bar{x}\in\R$ such that, up to a subsequence,
\begin{equation}\label{lim}
\lim_{j \to +\infty} x_j = \bar{x}.
\end{equation} 
  
Also, using~\eqref{fus} we gather that
\begin{equation}\label{pl}
|\bar{x}+\bar{k}| \leq R.
\end{equation}
Moreover, from~\eqref{maj}, \eqref{min} and~\eqref{lim} we find that
\begin{equation}\label{eu}
w(\bar{x}+\bar{k})= u^{(0)}(\bar{x}).
\end{equation}
In light of this and~\eqref{big}, we can refine~\eqref{pl} into
\begin{equation}\label{refinem}
|\bar{x}+\bar{k}| < R.
\end{equation}

Now we claim that
\begin{equation}\label{claim_maj}
u^{(0)}(\bar{x}) \geq-1+c.
\end{equation}
We prove this fact by contradiction, by supposing instead that
\begin{equation}\label{758493bdyfuew12345fcgh}
u^{(0)}(\bar{x}) \in (-1, -1+c).\end{equation}
We define
\[ \Omega:= \big\{ x \in (-\bar{k}-R, -\bar{k}+R)\;{\mbox{ s. ~\!t. }}\; u^{(0)}(x)\in(-1,-1+c )\big\}.\]
By~\eqref{refinem} and~\eqref{758493bdyfuew12345fcgh} we obtain that~$\bar{x} \in \Omega$. Then, the monotonicity of~$u^{(0)}$ gives that, for every~$x\in(-\bar{k}-R, \bar{x}]$,
$$ u^{(0)}(x)\le u^{(0)}(\bar{x})<-1+c,
$$
and therefore~$(-\bar{k}-R, \bar{x}] \subset \Omega$.
In addition, $\Omega$ is open since~$u^{(0)}$ is continuous. 

Now, Lemma~\ref{thc_deg}
(used here with~$t:=u^{(0)}(x)$ and~$r:=-1$) gives that, for all~$x\in\Omega$,
\begin{equation}\label{fnhekfhwuti732trufgekjfguwyt34i7t5qwqazwsx}
L_K u^{(0)}(x)=W\rq{}(u^{(0)}(x))\ge \frac{C_1}{\alpha-1} (1+ u^{(0)}(x))^{\alpha-1}\ge c(1+ u^{(0)}(x))^{\alpha-1}. 
\end{equation}
Furthermore, we set~$w_{\bar{k}}(x):= w(x +\bar{k})$ and we obtain by~\eqref{bar_sol} that, for any~$x \in (-\bar k-R, -\bar k+R)$,
\begin{equation*}
{L}_K w_{\bar{k}}(x) \leq c (1+ w_{\bar{k}}(x))^{\alpha-1}.
\end{equation*}
This, \eqref{maj} and~\eqref{fnhekfhwuti732trufgekjfguwyt34i7t5qwqazwsx}
lead to the following situation:
\begin{equation*}
\begin{cases}
{L}_K w_{\bar{k}}\leq c (1+ w_{\bar{k}})^{\alpha-1} &\mbox{in } \Omega, \\
{L}_K u^{(0)} \geq c (1+ u^{(0)})^{\alpha-1} &\mbox{in } \Omega, \\
w_{\bar{k}} \geq u^{(0)} &\mbox{in } \R. \\
\end{cases}
\end{equation*}
That is, the assumptions of Proposition~\ref{comp} are satisfied
with~$f_1(x,w(x))=f_2(x,w(x))=c (1+ w(x))^{\alpha-1}$.
Hence, since~$\bar x\in\Omega$, from~\eqref{eu}
we conclude that~$\omega_{\bar{k}}= u^{(0)}$ in~$\Omega$.

Thus, exploiting the continuity of~$u^{(0)}$ and~\eqref{big}, 
and recalling also~\eqref{bar_outside_interval},
we obtain that
\begin{equation*}
1 =\lim_{x \to -\bar{k}-R} w({\bar{k}}+x)
= \lim_{x \to -\bar{k}-R} w_{\bar{k}}(x) 
=  \lim_{x \to -\bar{k}-R} u^{(0)}(x)<1,
\end{equation*}
which is a contradiction. The claim in~\eqref{claim_maj}
is thereby established.

We now show that
\begin{equation}\label{mved}
\bar{x}+\bar{k} \in [0, R).
\end{equation}
In light of~\eqref{refinem}, to obtain~\eqref{mved} it is enough to check that~$\bar{x}+\bar{k}\geq 0 $. 
This is proved by contradiction supposing that~$\bar{x}+\bar{k}<0$
and setting~$\widetilde{k}:=-(\bar{x}+\bar{k})>0$. We exploit the fact that~$w$ is even
(recall Proposition~\ref{prop_bar})
and~\eqref{maj} (used here for~$x:=\widetilde{k}-\bar k$)
and we find that
\begin{equation}\label{vxzbncbWQHBGUHJGjh}
w_{\bar{k}}(\bar{x})= w_{-\bar{k}}(-\bar{x})= w(\widetilde{k}) 
=w((\widetilde{k}-\bar k) +\bar k)> u^{(0)}(\widetilde{k}- \bar k)
= u^{(0)}(-\bar{x}-2\bar k).
\end{equation}
Notice that, since~$-(\bar{x}+\bar{k})>0$, we have that~$-\bar{x}-2\bar k>\bar x$, and therefore from the strict monotonicity of~$u^{(0)}$ we deduce that~$u^{(0)}(-\bar{x}-2\bar k)>
u^{(0)}(\bar{x})$. This and~\eqref{vxzbncbWQHBGUHJGjh}
give that~$w_{\bar{k}}(\bar{x})>u^{(0)}(\bar{x})$,
which is in contradiction with~\eqref{eu}. Thus~\eqref{mved} holds true.

Now, let
\begin{equation}\label{y_range}
y \in \left[\frac{R}{3}, \frac{R}{2}\right].
\end{equation}
In this way, by~\eqref{mved}, we get that
\begin{equation*}
\bar{x}+\bar{k}-y \in \left[  - \frac{R}{2} ,\frac{2R}{3}  \right] \subset \left[-\frac{2R}{3}  ,\frac{2R}{3}\right].
\end{equation*}
As a consequence, in light of~\eqref{bar_decay},
\begin{equation}\label{est_guz}
1+ w(\bar{x}+\bar{k}-y ) \leq C \big(R+1-|\bar{x}+\bar{k}-y | \big)^{-\frac{2s}{\alpha-1}}\leq C (R/3)^{-\frac{2s}{\alpha-1}} \leq \widetilde{C} \, y^{-\frac{2s}{\alpha-1}},
\end{equation}
for some positive constant~$\widetilde{C}$.

Now, let~$\widetilde x\in\R$ be such that~$u^{(0)}(\widetilde x)= -1+c/2$. By the strict monotonicity of~$u^{(0)}$ and~\eqref{claim_maj}, it follows that~$\widetilde x<\bar{x}$. Consequently, putting together~\eqref{maj} and~\eqref{est_guz}, we obtain that
\begin{equation}\label{iq}
u^{(0)}(\widetilde x-y) < u^{(0)}(\bar{x}-y) \leq w(\bar{x}+\bar{k}-y) \leq -1+\widetilde{C} \, y^{-\frac{2s}{\alpha-1}},
\end{equation}
for any~$y$ as in~\eqref{y_range}.

Since~$\widetilde x$ and~$\widetilde{C}$ are independent of~$R$, and~$R$ can be taken as large as desired, \eqref{iq} says that,
if~$x\in(-\infty,0)$ and~$|x|$ is sufficiently large,
\begin{equation}\label{289dy793736ggd980}
u^{(0)}(x) \leq -1+ \widetilde{C} |x|^{-\frac{2s}{\alpha-1}},
\end{equation}
which establishes the first estimate in~\eqref{asymp_decay}.

We now show the second estimate in~\eqref{asymp_decay}.
To this aim, we define the function~$v^{(0)}(x) := u^{(0)}(-x)$
and we notice that~$v^{(0)}$ inherits the regularity
properties from~$u^{(0)}$ and it is strictly decreasing.

{F}rom~\eqref{big} and the fact that~$u^{(0)}\in{\mathcal{X}}$, we also see that
$$
| v^{(0)}| < 1 \ \mbox{in} \ \R
\qquad\mbox{and}\qquad
\lim_{x\to\pm\infty}v^{(0)}(x)=\mp1 .
$$
Moreover, for all~$x\in\R$,
$$ L_K v^{(0)}(x) = L_K u^{(0)} (-x)=W\rq{}(u^{(0)}(-x)).
$$

Hence, we are in the position of exploiting the first part of this proof that took care of the first estimate in~\eqref{asymp_decay}, applied now to~$v^{(0)}$,
with the only caveat that Proposition~\ref{prop_bar} must be used here with~$c$ in~\eqref{vttv4869302} replaced by
$$ \min\left\{ \frac{C_3}{\gamma-1}, \xi \right\}$$
and~$m:=\gamma$.

In this way, we obtain that
        $$
          1-v^{(0)}(x) \leq \widetilde{C} |x|^{-\frac{2s}{\gamma-1}} \quad\mbox{if } x\leq -R.$$
Namely,
        $$
          1-u^{(0)}(-x) \leq \widetilde{C} |x|^{-\frac{2s}{\gamma-1}} \quad\mbox{if } x\leq -R,$$
which gives the second estimate in~\eqref{asymp_decay}, as desired.
\end{proof}

\begin{proof}[Proof of statement~\eqref{9873grv} of Lemma~\ref{psc0-500}]
For simplicity, we show that~\eqref{9873grv} holds for~$u^{(0)}$
(as this implies the estimates for~$u^{(x_0)}$ for any~$x_0\in\R$).

Thanks to the regularity of~$u^{(0)}$, as provided by~\eqref{propert45},
we can differentiate~\eqref{ru3qio2ruebn68594} and find that, for all~$x\in\R$,
$$ {L}_K (u^{(0)})\rq{}(x) = W\rq{}\rq{}(u^{(0)}(x))  (u^{(0)})\rq{} (x).$$
Therefore, using~\eqref{pot_deg} and~\eqref{asymp_decay}, and recalling also that~$u^{(0)}$ is strictly increasing,
we obtain that there exist~$\widetilde{C}>0$ and~$x_0>0$ such that
\begin{equation}\label{90377djoi}
\begin{split}
L_K ( u^{(0)})\rq{}(x) &\leq
\begin{cases}C_2 \big(1+u^{(0)}(x)\big)^{\beta-2}(u^{(0)})\rq{}(x) &\mbox{if } x\leq -x_0, \\
C_4\big(1-u^{(0)}(x)\big)^{\delta-2}(u^{(0)})\rq{}(x) &\mbox{if } x\geq x_0.
\end{cases}\\
&\leq \begin{cases}
\widetilde{C} |x |^{-\frac{2s(\beta-2)}{\alpha-1}}(u^{(0)})\rq{}(x) &\mbox{if } x\leq -x_0,\\
\widetilde{C}|x |^{-\frac{2s(\delta-2)}{\gamma-1}}(u^{(0)})\rq{}(x) &\mbox{if } x\geq x_0.
\end{cases}
\end{split}
\end{equation}

Now, we consider~$\phi \in C^{\infty}(\R)$ such that
\begin{equation}\label{3wertyjxsdcvb0ffdj88}
\phi(x) := \begin{cases}
|x |^{-\left(1+\frac{2s(\alpha-\beta+1)}{\alpha-1}\right)}  &\mbox{if } x \leq -x_0, \\
|x |^{-\left(1+\frac{2s(\gamma-\delta+1)}{\gamma-1}\right)}  &\mbox{if } x \geq x_0.
\end{cases}
\end{equation}
Also, we ask that~$\phi >0$ in~$\R$ and that 
\begin{equation*}
\int_{-x_0}^{x_0} \phi(x) \, dx \geq \frac{2\widetilde{C}}{\lambda},
\end{equation*}
where~$\lambda$ is the quantity appearing in~\eqref{newbound}.

As a consequence, by means of Proposition~\ref{tontobis}
(used here with~$\kappa:=x_0$, $\sigma:=1+\frac{2s(\alpha-\beta+1)}{\alpha-1}$ and~$\tau:=1+\frac{2s(\gamma-\delta+1)}{\gamma-1}$)
we have that,  
\begin{eqnarray*} &&\lim_{|x | \to  +\infty} |x|^{1+2s} {L}_K\phi(x) \geq \lambda \left( \frac{x_0^{-
\frac{2s(\alpha-\beta+1)}{\alpha-1}}( \alpha-1)}{2s(\alpha-\beta+1)}+ \int_{-x_0}^{x_0} \phi(y) \, dy+
\frac{x_0^{-\frac{2s(\gamma -\delta+1) }{\gamma-1}}(\gamma-1)}{2s(\gamma -\delta+1)}
\right)\\&&\qquad \qquad \ge
\lambda  \int_{-x_0}^{x_0} \phi(y) \, dy
\ge 2\widetilde{C}.
\end{eqnarray*}
Hence, for some~$x_1>x_0$, we have that, for all~$x\in(-\infty, x_1)\cup(x_1,+\infty)$,
\begin{equation}\label{1afdbf75gjviy87jjmnmkjo} L_K\phi(x)\ge \frac{\widetilde{C}}{|x|^{1+2s}}.\end{equation}

We now point out that, if~$x\le -x_0$,
\begin{equation*}
\frac{\phi(x)}{|x|^{\frac{2s(\beta-2)}{\alpha-1}}}
=\frac{|x |^{-\left(1+\frac{2s(\alpha-\beta+1)}{\alpha-1}\right)}}{|x|^{\frac{2s(\beta-2)}{\alpha-1}}}=\frac1{|x|^{1+2s}}
\end{equation*}
and similarly, if~$x\ge x_0$,
\begin{equation*}
\frac{\phi(x)}{|x|^{\frac{2s(\delta-2)}{\gamma-1}}}
=\frac{|x |^{-\left(1+\frac{2s(\gamma-\delta+1)}{\gamma-1}\right)}}{|x|^{\frac{2s(\delta-2)}{\gamma-1}}}=\frac1{|x|^{1+2s}}.
\end{equation*}
Hence, from these observations and~\eqref{1afdbf75gjviy87jjmnmkjo}, we deduce that
\begin{equation}\label{p48d4d}
	L_K \phi (x) \geq \begin{cases}
\widetilde{C} |x|^{-\frac{2s(\beta-2)}{\alpha-1}} \phi(x) &\mbox{if } x \leq -x_1,\\
\widetilde{C} |x|^{-\frac{2s(\delta-2)}{\gamma-1}} \phi(x) &\mbox{if } x \geq x_1.
\end{cases}
\end{equation}

Now, we set
\begin{equation*}
\widehat{C}:=\max_{x \in [-x_1,x_1]}\phi(x) \left(  \min_{x \in [-x_1,x_1]} (u^{(0)})\rq{} (x)\right)^{-1}
\end{equation*}
and we notice that, for any~$|x | < x_1$,
\begin{equation}\label{546HJHDGILUIDyre65t}
\widehat{C} (u^{(0)})\rq{} (x) - \phi(x)  \geq 0. 
\end{equation}

We claim that
\begin{equation}\label{039hgvce0}
\widehat{C} (u^{(0)})\rq{} (x) \geq \phi(x) \quad\mbox{for any } x \in \R.
\end{equation}
In order to prove the claim, we define, for any~$b \in [0,+\infty)$, the function~${v_b:= \widehat{C} (u^{(0)})\rq{}+b-\phi}$.
Since~$(u^{(0)})\rq{}>0$ and~$\phi$ is bounded, for any~$b\ge \|\phi\|_{L^\infty(\R)}$, we have that~$v_b>0$ in~$\R$.

Now, if~$v_b > 0$ in~$\R$ for any~$b \in [0,+\infty)$, the claim in~\eqref{039hgvce0} plainly follows taking~$b=0$. Hence, from now on, we suppose that there exists~$b_0\in(0,+\infty)$
such that~$v_b>0$ for all~$b\in(b_0,+\infty)$ and~$v_{b_0}(z)=0$ at some point~$z\in\R$.

By the definition of~$b_0$, there exist points~$x_k$ such that~$v_{b_0}(x_k)< 2^{-k}$. Without loss of generality, we may suppose that~$x_k\leq 0$ (otherwise,
in what follows, we use the information coming from the decay at~$+\infty$).

Furthermore, the sequence~$x_k$ is bounded from below since, if not, we would have
\begin{equation*}
b_0 = \limsup_{k \to -\infty} \left( \widehat{C} (u^{(0)})\rq{}(x_k)+b_0 - \phi(x_k) \right) = \lim_{k \to -\infty} v_{b_0}(x_k) = 0,
\end{equation*}
which is a contradiction. 

Moreover, exploiting~\eqref{546HJHDGILUIDyre65t}, we obtain for any~$ |x | \leq x_1$ and~$k > -\log_{2}b_0$ 
\begin{equation*}
 2^{-k} < b_0 \leq \widehat{C} (u^{(0)})\rq{} (x) - \phi(x)+b_0 = v_{b_0}(x).
\end{equation*}
As a consequence, since~$v_{b_0}(x_k)\leq 2^{-k}$, we have that~$x_k\in (-\infty, -x_1]$ for any~$k > -\log_{2} b_0 $.

Gathering these pieces of information,
we conclude that there exists~$x_{\infty} \in (-\infty, -x_1]$ such that~$x_k \to x_{\infty}$, up to a subsequence, as~$k\to+\infty$. 
The continuity of~$v_{b_0}$ also gives that~$v_{b_0}({x_{\infty}})=0$.

As a result, since~$x_{\infty}\leq -x_1$, we can exploit~\eqref{90377djoi} and~\eqref{p48d4d} to compute
\begin{equation*}
\begin{split}
L_K v_{b_0} ({x_{\infty}}) &= \widehat{C} L_K (u^{(0)})\rq{}({x_{\infty}}) -  L_K \phi ({x_{\infty}})\\
&\leq \widetilde{C} | {x_{\infty}} |^{-\frac{2s(\beta-2)}{\alpha-1}} (\widehat{C} (u^{(0)})\rq{}({x_{\infty}}) - \phi ({x_{\infty}}))\\
&= - b_0 \widetilde{C}| {x_{\infty}}|^{-\frac{2s(\beta-2)}{\alpha-1}}\\&<0.
\end{split}
\end{equation*}
On the other hand, we have that
\begin{equation*}
L_K v_{b_0} ({x_{\infty}}) = \int_{\R} (v_{b_0}(y)-v_{b_0}({x_{\infty}})) K({x_{\infty}}-y) \, dy = \int_{\R} v_{b_0}(y) K({x_{\infty}}-y) \, dy \geq 0.
\end{equation*}
We thereby obtain the desired contradiction, which completes the proof of~\eqref{039hgvce0}.

Formulas~\eqref{3wertyjxsdcvb0ffdj88} and~\eqref{039hgvce0} yield the estimates in~\eqref{398r7gree}.

Hence we now focus on the proof of the estimates in~\eqref{asymp_decay_lowbound}. To this aim, exploiting the regularity of~$u^{(0)}$, we apply the Fundamental Theorem of Calculus 
and use the first estimate in~\eqref{398r7gree} to obtain that, for any~$x \leq -x_0$,
\begin{equation*}
\begin{split}&
u^{(0)}(x)+1 = \int_{-\infty}^x (u^{(0)})\rq{}(y) \, dy \geq  \widehat{C}\int_{x}^{+\infty}  y^{-\left( 1+\frac{2s(\alpha-\beta+1)}{\alpha-1}\right)}\, dy\\ &\qquad\qquad= \frac{(\alpha-1)}{2s\widehat{C}(\alpha-\beta+1)} | x |^{-\frac{2s(\alpha-\beta+1)}{\alpha-1}}.
\end{split}
\end{equation*}
This provides the first estimate in~\eqref{asymp_decay_lowbound}.

Similarly, making use of the second estimate in~\eqref{398r7gree},
for any~$x \geq x_0$,
\begin{equation*}
\begin{split}&
1- u^{(0)}(x) =  \int_{x}^{+\infty} (u^{(0)})\rq{}(y) \, dy \geq {\widehat{C}}\int_{x}^{+\infty}  y^{-\left( 1+\frac{2s(\gamma-\delta+1)}{\gamma-1}\right)}\, dy\\ &\qquad\qquad= \frac{(\gamma-1)}{2s\widehat{C}(\gamma-\delta+1)} |x |^{-\frac{2s(\gamma-\delta+1)}{\gamma-1}},
\end{split}
\end{equation*} which completes the proof of~\eqref{asymp_decay_lowbound}.
\end{proof}

\begin{proof}[Proof of statement~\eqref{iiiiunoiiii} of Lemma~\ref{psc0-500}]
For simplicity, we show that~\eqref{iiiiunoiiii} holds for~$u^{(0)}$
(as this implies the estimates for~$u^{(x_0)}$ for any~$x_0\in\R$). 

By~\eqref{pot_deg} and~\eqref{asymp_decay_lowbound}, there exists~$\widetilde{C}>0$ and~$x_0>0$ such that
\begin{equation}\label{dbctovaj}\begin{split}
W\rq{}\rq{}(u^{(0)}(x)) &\geq 
\begin{cases} C_1
( 1+u^{(0)}(x))^{\alpha-2} &\mbox{if }x\leq -x_0  , \\
C_3( 1-u^{(0)}(x))^{\gamma-2} &\mbox{if } x\geq x_0 
\end{cases}
\\&\geq
\begin{cases}
 \widetilde{C} |x |^{-\frac{2s(\alpha -2)(\alpha-\beta+1)}{\alpha-1}} &\mbox{if }x\leq -x_0 , \\
\widetilde{C}|x |^{-\frac{2s(\gamma-2)(\gamma-\delta-1)}{\gamma-1}} &\mbox{if } x\geq x_0 .
\end{cases}\end{split}
\end{equation}

Possibly taking~$x_0$ larger, thanks to~\eqref{ricdifar} we can assume that
\begin{equation}\label{xyxve}
\frac{\widetilde{C}}{2\Lambda} - \left( \frac{x_0^{-\frac{2s\left(1-(\alpha-2)(\alpha-\beta)\right)}{\alpha-1}}(\alpha-1)}{2s\left(1-(\alpha-2)(\alpha-\beta)\right)} +\frac{x_0^{-\frac{2s\left(1-(\gamma-2)(\gamma-\beta)\right)}{\gamma-1}}(\gamma-1)}{2s\left(1-(\gamma-2)(\gamma-\delta)\right)}\right)>0.
\end{equation}

Now, thanks to the regularity of~$u^{(0)}$ (recall~\eqref{propert45}), we can differentiate the equation
in~\eqref{miser45} and we see that, for all~$x\in\R$,
$$ {L}_K (u^{(0)})\rq{}(x) = W\rq{}\rq{}(u^{(0)}(x))  (u^{(0)})\rq{} (x).$$
Thus, exploiting~\eqref{dbctovaj} and recalling that~$u^{(0)}$
is strictly increasing, 
\begin{equation}\label{jgox}
{L}_K (u^{(0)})\rq{}(x) \geq \begin{cases}
 \widetilde{C}|x |^{-\frac{2s(\alpha -2)(\alpha-\beta+1)}{\alpha-1}} (u^{(0)})\rq{}(x) &\mbox{if } x \le -x_0 ,\\
 \widetilde{C} |x |^{-\frac{2s(\gamma-2)(\gamma-\delta-1)}{\gamma-1}} (u^{(0)})\rq{}(x) &\mbox{if } x \ge x_0.
\end{cases}
\end{equation}

We now take~$\phi \in C^{\infty}(\R)$ such that
\begin{equation}\label{vfurejmbheGFTYSTRDFKFGJ0987654}
\phi(x) =
\begin{cases}
|x |^{-\big(1+\frac{2s\left(1-(\alpha-2)(\alpha-\beta)\right)}{\alpha-1}\big)} &\mbox{if } x<-x_0,\\
|x|^{-\big(1+\frac{2s\left(1-(\gamma-2)(\gamma-\delta)\right)}{\gamma-1}\big)}&\mbox{if } x>x_0.
\end{cases}
\end{equation}
Also, we ask that~$\phi >0$ in~$[-x_0,x_0]$ and
\begin{equation}\label{pdbj}
\int_{-x_0}^{x_0}\phi(y) \, dy \leq \frac{\widetilde{C}}{2\Lambda} - \left( \frac{x_0^{-\frac{2s\left(1-(\alpha-2)(\alpha-\beta)\right)}{\alpha-1}}(\alpha-1)}{2s\left(1-(\alpha-2)(\alpha-\beta)\right)} +\frac{x_0^{-\frac{2s\left(1-(\gamma-2)(\gamma-\beta)\right)}{\gamma-1}}(\gamma-1)}{2s\left(1-(\gamma-2)(\gamma-\beta)\right)}\right)
\end{equation}
and we remark that this is possible thanks to~\eqref{xyxve}.

Now, thanks to~\eqref{ricdifar} we are allowed to exploit Proposition~\ref{tonto} to find that
\begin{equation*}
\lim_{|x | \to  +\infty} |x|^{1+2s} {L}_K\phi(x) \leq \Lambda \left( \frac{x_0^{-\frac{2s\left(1-(\alpha-2)(\alpha-\beta)\right)}{\alpha-1}}(\alpha-1)}{2s\left(1-(\alpha-2)(\alpha-\beta)\right)} + \int_{-x_0}^{x_0} \phi(y) \, dy
+\frac{x_0^{-\frac{2s\left(1-(\gamma-2)(\gamma-\beta)\right)}{\gamma-1}}(\gamma-1)}{2s\left(1-(\gamma-2)(\gamma-\beta)\right)}\right).
\end{equation*}
This and~\eqref{pdbj} imply that
\begin{equation*}
\lim_{|x | \to +\infty}| x |^{1+2s}{L}_K \phi (x) \leq \frac{\widetilde{C}}{2}.
\end{equation*}

We now point out that, if~$x\le -x_0$,
\begin{equation*}
\frac{\phi(x)}{|x |^{\frac{2s(\alpha -2)(\alpha-\beta+1)}{\alpha-1}}}=\frac{|x |^{-\big(1+\frac{2s\left(1-(\alpha-2)(\alpha-\beta)\right)}{\alpha-1}\big)}}{|x |^{\frac{2s(\alpha -2)(\alpha-\beta+1)}{\alpha-1}}}=\frac1{|x|^{1+2s}}
\end{equation*}
and similarly, if~$x\ge x_0$,
\begin{equation*}
\frac{\phi(x)}{|x |^{\frac{2s(\gamma -2)(\gamma-\delta+1)}{\gamma-1}}}=\frac{|x |^{-\big(1+\frac{2s\left(1-(\gamma-2)(\gamma-\delta)\right)}{\gamma-1}\big)}}{|x |^{\frac{2s(\gamma -2)(\gamma-\delta+1)}{\gamma-1}}}=\frac1{|x|^{1+2s}}.
\end{equation*}

As a consequence, we have that
there exists~$x_1\ge x_0$ such that
\begin{equation}\label{jddo}
{L}_K \phi(x) \leq \begin{cases}
\widetilde{C} |x|^{-\frac{2s(\alpha -2)(\alpha-\beta+1)}{\alpha-1}} \phi(x) &\mbox{if } x < -x_1,\\
\widetilde{C} | x|^{-\frac{2s(\gamma -2)(\gamma-\delta+1)}{\gamma-1}} \phi(x) &\mbox{if } x >x_1.
\end{cases}
\end{equation}

Now, we set 
\begin{equation*}
x_2:= \max\{ x_1,1\} \qquad\mbox{and}\qquad 
\widehat{C} :=4\Vert (u^{(0)})\rq{}\Vert_{L^{\infty}(\R)}  \left(\min_{x \in [-x_2,x_2]}\phi(x)\right)^{-1}
\end{equation*}
and we claim that
\begin{equation}\label{mvcjp}
(u^{(0)})\rq{}(x) <\widehat{C} \phi(x) \quad\mbox{for any } x \in \R.
\end{equation}
In order to prove the above inequality, we take~$b \in [0,+\infty)$ and we define~$v_b:= \widehat{C} \phi +b -(u^{(0)})\rq{}$.
We recall that, by~\eqref{propert45} and~\eqref{big}, $(u^{(0)})\rq{}$ is bounded and vanishes at infinity. Therefore, for any~$b>\Vert (u^{(0)})\rq{}\Vert_{L^{\infty}(\R)}$, we have that~$v_b >0$ in~$\R$.

Now, if~$v_b > 0$ in~$\R$ for any~$b \in [0,+\infty)$, then the claim in~\eqref{mvcjp} plainly follows taking~$b=0$. Hence, from now on, we suppose that there exists~$b_0\in(0,+\infty)$
such that~$v_b>0$ for all~$b\in(b_0,+\infty)$ and~$v_{b_0}(z)=0$ at some point~$z\in\R$.
 
By the definition of~$b_0$, there exist points~$x_k$ such that~$v_{b_0}(x_k)< 2^{-k}$. Without loss of generality, we may suppose that~$x_k\geq 0$ (otherwise,
in what follows, we use the information coming from the decay at~$-\infty$). Moreover, $x_k$ are bounded from above since, if not, we would have
\begin{equation*}
b_0 = \limsup_{k \to +\infty} \left( \widehat{C}\phi(x_k) +b_0 - (u^{(0)})\rq{}(x_k) \right) = \lim_{k \to +\infty} v_{b_0}(x_k) = 0,
\end{equation*}
which is a contradiction. 

Also, for~$k$ sufficiently large, we have that
\begin{equation*}
\Vert (u^{(0)})\rq{}\Vert_{L^{\infty}(\R)} \geq 2^{-k} >v_{b_0}(x_k)\geq \widehat{C} \phi(x_k) -(u^{(0)})\rq{}(x_k) \geq \widehat{C} \phi(x_k) - \Vert (u^{(0)})\rq{}\Vert_{L^{\infty}(\R)}.
\end{equation*}
Accordingly, recalling the definition of~$\widehat{C}$,
\[\phi(x_k) \leq\frac{\min_{x \in [-x_2,x_2]}\phi(x)}{2}.\] 
This implies that~$x_k>x_2$.

Gathering these pieces of information, we conclude that
there exists~$x_{\infty} \in [ x_2, +\infty)$ such that~$x_k \to x_{\infty}$, up to a subsequence, as~$k\to+\infty$. 

In addition, thanks to the continuity of~$v_{b_0}$,
\begin{equation}\label{zeroyes}
v_{b_0}(x_{\infty})=0.
\end{equation}

Now, by means of~\eqref{jgox} and~\eqref{jddo},
\begin{equation*}
\begin{split}
{L}_K v_{b_0}(x_k) &= \widehat{C}{L}_K \phi(x_k) -  {L}_K (u^{(0)})\rq{} (x_k) \leq  \widetilde{C} x_k^{-\frac{2s(\gamma -2)(\gamma-\delta+1)}{\gamma-1}} \big(\widehat{C}\phi(x_k) - (u^{(0)})\rq{}(x_k) \big) \\
&= \widetilde{C} x_k^{-\frac{2s(\gamma -2)(\gamma-\delta+1)}{\gamma-1}}v_{b_0}(x_k) -  \widetilde{C}  x_k^{-\frac{2s(\gamma -2)(\gamma-\delta+1)}{\gamma-1}} b_0 \\
&\leq C( 2^{-k}- b_0),
\end{split}
\end{equation*}
for some~$C>0$. 

The regularity of~$v_{b_0}$ then implies that
\begin{equation*}
{L}_K v_{b_0}(x_{\infty}) = \lim_{k \to +\infty} {L}_K v_{b_0}(x_k) \leq -C b_0 <0.
\end{equation*}
On the other hand, by~\eqref{zeroyes},
\begin{equation*}
{L}_K v_{b_0}(x_{\infty}) = \int_{\R} (v_{b_0}(y)-v_{b_0}(x_{\infty})) K(x_{\infty}-y) \, dy 
= \int_{\R} v_{b_0}(y) K(x_{\infty}-y) \, dy \geq 0.
\end{equation*}
We thereby obtain the desired contradiction,
and therefore~\eqref{mvcjp} is established.

The desired estimates now follow from~\eqref{vfurejmbheGFTYSTRDFKFGJ0987654}
and~\eqref{mvcjp}.
\end{proof}

\section{Proofs of Theorems~\ref{main_thm} and~\ref{main_thm_symm}}\label{main_thm_proof}

We are now ready to deal with the~$1$-D minimizers of~$\Ec$ and provide the proofs of the main results of this paper.

We recall the definitions of~$\mathscr{M}$ and of~$\mathscr{M}^{(x_0)}$, respectively in~\eqref{emme} and~\eqref{emmezero}.
    
     \subsection*{Proof of Theorem~\ref{main_thm}} 
For any~$R>3$, we use Lemma~\ref{min_open_inter} with~$a:=-R$ and~$b:=R$ to obtain a minimizer~$v_{[-R,R]}: \R \to [-1,1]$ such that~$v_{[-R,R]}(x) = -1$ if~$x\leq-R$ and~$v_{[-R,R]}(x)=1$ if~$x \geq R$. Also, the energy bound~\eqref{bound_energy} holds for~$v_{[-R,R]}$ and, in view of Proposition~\ref{eupro}, $v_{[-R,R]}$ is non-decreasing.

The minimization property of~$v_{[-R,R]}$ yields that
\begin{equation*}
{L}_K v_{[-R,R]} = W\rq{}(v_{[-R,R]}) \quad\mbox{in }  (-R,R).
\end{equation*}
Therefore, by Proposition~\ref{reg_dirichlet_pbm}, we have that~$v_{[-R,R]} \in C^{\theta}(\R) \cap C^{2s+\theta}(-R,R)$, for some~$\theta \in (0,s)$, and its  H\"{o}lder norm is bounded independently of~$R$.

Now, by continuity, there must be a point~$\bar{x}_R\in\R$ such that~$v_{[-R,R]}(\bar{x}_R)=0$. We claim that
\begin{equation}\label{evedser}
\lim_{R \to +\infty}( R- |\bar{x}_R|)= +\infty.
\end{equation}
To check this, we argue by contradiction and we suppose that
\begin{equation}\label{eitor}
\text{either}\quad \lim_{R \to +\infty} ( R+ \bar{x}_R) \leq C \quad\mbox{or}\quad \lim_{R \to +\infty} ( R - \bar{x}_R) \leq C,
\end{equation}
for some~$C>0$. 

We suppose that
the first case in~\eqref{eitor} occurs (the other being analogous).
In this case, we consider the minimizer~$v_{[-2R,0]}$ and notice that, in view of the translation invariance, $v_{[-2R,0]}(x)=v_{[-R,R]}(x-R)$. Then, taking~$C$ as in~\eqref{eitor} and exploiting the monotonicity of~$v_{[-2R,0]}$, we see that, for~$R$ large enough,
\begin{equation*}
v_{[-2R,0]}(C)= v_{[-R,R]}(C - R) \geq v_{[-R,R]}(\bar{x}_R) =0 . 
\end{equation*}
This inequality gives that
\begin{equation*}
\lim_{R \to +\infty} v_{[-2R,0]}(C) \geq 0,
\end{equation*}
which is in contradiction with Proposition~\ref{comepo} and Remark~\ref{remm}. The proof of~\eqref{evedser} is thereby complete.

Now, we set
\begin{equation*}
v^{(0)}_R(x) :=  v_{[-R,R]}(x+\bar{x}_R).
\end{equation*}
In this way, $v^{(0)}_R(0)=0$. 

Moreover, $v^{(0)}_R$ is a local minimizer of~$\Ec$ in~$[-R-\bar{x}_R,R-\bar{x}_R] $ and, thanks to~\eqref{evedser},
\begin{equation*}
\lim_{R \to +\infty} (-R-\bar{x}_R) = - \infty \qquad\mbox{and}\qquad \lim_{R \to +\infty} (R-\bar{x}_R) = +\infty.
\end{equation*} 
Consequently, we may suppose that~$v^{(0)}_R$ converges locally uniformly in~$\R$ to some~$u^{(0)}:\R \to [-1,1]$ such that~$u^{(0)}\in C_{{\rm loc}}^{2s+\theta}(\R) \cap L^{\infty}(\R)$. Also,
we obtain the pointwise equality
\begin{equation}\label{miser}
{L}_K u^{(0)}= W\rq{}(u^{(0)}) \quad\mbox{in }\R.
\end{equation}
Furthermore, we have that
\begin{equation}\label{propert}
{\mbox{$u^{(0)}\in C^{1+2s+\theta}(\R)\cap L^{\infty}(\R)$ for some~$\theta\in(0,1)$, $u^{(0)}$ is non-decreasing
and~$u^{(0)}(0)=0$,}}
\end{equation}
where we deduce the first property from Proposition~\ref{reg_entire_sol}. 

Now, we prove that
\begin{equation}\label{ppdd}
\mathscr{G}(u^{(0)}) <+\infty.
\end{equation}
In order to show this, we fix~$\rho>0$ and take~$R$ large enough
such that~$R-\bar{x}_R>\rho$. Then, by the Fatou\rq{}s Lemma and Proposition~\ref{mdss}, we obtain, for some~$C>0$,
\begin{equation*}
\Ec(u^{(0)}, [-\rho,\rho]) \leq \liminf_{R \to +\infty} \Ec(v^{(0)}_R, [-\rho,\rho]) \leq C\Psi_s(\rho),
\end{equation*}
thus showing~\eqref{ppdd}. 

Furthermore,
\begin{equation}\label{lim_doub}
\lim_{x \to \pm \infty} u^{(0)}(x)= \pm1.
\end{equation}
Indeed, we first notice that, by~\eqref{propert}, the limits in~\eqref{lim_doub} exist. Now, we take~$a_{-}$,
$a_{+} \in [-1,1]$ such that 
\begin{equation*}
\lim_{x \to -\infty}u^{(0)}(x) = a_{-} \qquad\mbox{and}\qquad \lim_{x \to +\infty}u^{(0)}(x) = a_{+}.
\end{equation*}
By~\eqref{propert}, we know that~$a_{+} \in [0,1]$. Moreover, if~$a_{+}\neq 1$, then~\eqref{pot_zero} gives that
\begin{equation*}
C_W :=\inf_{x \in [0,a_{+}]}W(x)>0.
\end{equation*}
As a consequence, for any~$\rho>0$,
\begin{equation*}
\Ec (u^{(0)}, [-\rho,\rho]) \geq \int_{0}^{\rho}W(u^{(0)}(x)) \, dx \geq C_W \rho,
\end{equation*}
this contradicting~\eqref{ppdd} if~$\rho$ is sufficiently large.
Therefore~$a_{+}=1$, and similarly one can check that~$a_{-}=-1$. This completes the proof of~\eqref{lim_doub}. 

Now, thanks to~\eqref{propert}, \eqref{lim_doub} and~\cite[Remark 2.7]{CP16}, we exploit~\cite[Theorem 3]{CP16} to infer that~$u^{(0)}$ is a class~A minimizer of~$\Ec$.

Consequently, $u^{(0)} \in \mathscr{M}^{(0)}$ and we can apply Lemma~\ref{psc0-500} to obtain that~$u^{(0)}$ satisfies the decay estimates~\eqref{asymp_decay},~\eqref{asymp_decay_lowbound},~\eqref{398r7gree} and~\eqref{eq:asymp-derivata}. Moreover, Lemma~\ref{psc0-500} also gives that~$u^{(0)}$ is, up to translations, the unique class~A minimizer of~$\Ec$.

Finally, it holds that, up to translations, $u^{(0)}$ is the only non-decreasing solution to~\eqref{miser} in the family of admissible functions~$\mathcal{X}$. Indeed, let us consider~$v \in \mathcal{X}$ to be a non-decreasing solution to~\eqref{miser}.  Then, by Lemma~\ref{lemma:degiorgi}, $v$ is a class~A minimizer of~$\Ec$ and thus~$v \in \mathscr{M}^{(x_0)}$ for some~$x_0$. Hence, \eqref{iiunoii} gives that~$v(x)= u^{(0)}(x-x_0)$.

\subsection*{Proof of Theorem~\ref{main_thm_symm}}
For any~$R>3$, we use Lemma~\ref{min_open_inter} with~$a:=-R$ and~$b:=R$ to obtain a minimizer~$v_{[-R,R]}: \R \to [-1,1]$ such that~$v_{[-R,R]}(x) = -1$ if~$x\leq-R$ and~$v_{[-R,R]}(x)=1$ if~$x \geq R$. Also, the energy bound~\eqref{bound_energy} holds for~$v_{[-R,R]}$ and, in view of Proposition~\ref{eupro}, $v_{[-R,R]}$ is non-decreasing. Furthermore, by Proposition~\ref{symm}, ~$v_{[-R,R]}$ is odd and so~$v_{[-R,R]}(0)=0$ for any~$R$.

The minimization property of~$v_{[-R,R]}$ implies that
\begin{equation*}
{L}_K v_{[-R,R]} = W\rq{}(v_{[-R,R]}) \quad\mbox{in }  (-R,R).
\end{equation*}
Then, by Proposition~\ref{reg_dirichlet_pbm}, $v_{[-R,R]} \in C^{\theta}(\R) \cap C^{2s+\theta}(-R,R)$ for some~$\theta \in (0,s)$, and its  H\"{o}lder norm is bounded independently of~$R$. 

Consequently, we may suppose that~$v_{[-R,R]}$ converges locally uniformly in~$\R$ to some~$u^{(0)}:\R \to [-1,1]$ such that~$u^{(0)}\in C_{{\rm loc}}^{2s+\theta}(\R) \cap L^{\infty}(\R)$. Also,
we obtain the pointwise equality
\begin{equation}\label{miser_symm}
{L}_K u^{(0)}= W\rq{}(u^{(0)}) \quad\mbox{in }\R.
\end{equation}
Furthermore, we have that
\begin{equation}\label{propert_symm}
{\mbox{$u^{(0)}\in C^{1+2s+\theta}(\R)\cap L^{\infty}(\R)$
for some~$\theta\in(0,1)$, $u^{(0)}$ is non-decreasing and odd,}}
\end{equation}
where we deduce the first property from Proposition~\ref{reg_entire_sol} and we stress that~$u^{(0)}(0)=0$ by symmetry. 

Now, we prove that
\begin{equation}\label{ppdd_symm}
\mathscr{G}(u^{(0)}) <+\infty.
\end{equation}
For this, we fix~$\rho>0$ and take~$R>\rho$. Then, by the Fatou\rq{}s Lemma and Proposition~\ref{mdss}, we obtain that
\begin{equation*}
\Ec(u^{(0)}, [-\rho,\rho]) \leq \liminf_{R \to +\infty} \Ec(v_{[-R,R]}, [-\rho,\rho]) \leq C\Psi_s(\rho),
\end{equation*}
which gives~\ref{ppdd_symm}. 

Furthermore,
\begin{equation}\label{lim_doub_symm}
\lim_{x \to \pm \infty} u^{(0)}(x)= \pm1.
\end{equation}
Indeed, we first notice that by~\eqref{propert_symm} the limits in~\eqref{lim_doub_symm} exist. Now, we take~$a_{-}$, $a_{+} \in [-1,1]$ such that 
\begin{equation*}
\lim_{x \to -\infty}u^{(0)}(x) = a_{-} \qquad\mbox{and}\qquad \lim_{x \to +\infty}u^{(0)}(x) = a_{+}.
\end{equation*}
By~\eqref{propert_symm}, we know that~$a_{+} \in [0,1]$. Moreover, if~$a_{+}\neq 1$, then~\eqref{pot_zero} gives that
\begin{equation*}
C_W :=\inf_{x \in [0,a_{+}]}W(x)>0.
\end{equation*}
As a consequence, for any~$\rho>0$ we compute
\begin{equation*}
\begin{split}
\Ec (u^{(0)}, [-\rho,\rho]) &\geq \int_{0}^{\rho}W(u^{(0)}(x)) \, dx \geq C_W \rho,
\end{split}
\end{equation*}
which contradicts~\eqref{ppdd_symm} when~$\rho$ is large enough.
Accordingly, we have that~$a_{+}=1$, and similarly one can show that~$a_{-}=-1$ and complete the proof of~\eqref{lim_doub_symm}.

Finally, thanks to~\eqref{propert_symm}, \eqref{lim_doub_symm} and~\cite[Remark 2.7]{CP16}, we can use~\cite[Theorem 3]{CP16} to infer that~$u^{(0)}$ is a class~A minimizer of~$\Ec$. Consequently, $u^{(0)} \in \mathscr{M}^{(0)}$ and we can apply Lemma~\ref{psc0-500} to obtain that~$u^{(0)}$ satisfies the decay estimates~\eqref{asymp_decay_symm},~\eqref{asymp_decay_symm_lowbound},~\eqref{97e382d} and~\eqref{98fgvrjfbgvvfoi88}. By Lemma~\ref{psc0-500} we also have that~$u^{(0)}$ is, up to translations, the unique class~A minimizer of~$\Ec$.

Finally, it holds that, up to translations, $u^{(0)}$ is the only non-decreasing solution to~\eqref{miser_symm} in the family of admissible functions~$\mathcal{X}$. Indeed, let us consider~$v \in \mathcal{X}$ to be a non-decreasing solution to~\eqref{miser_symm}.  Then, by Lemma~\ref{lemma:degiorgi}, $v$ is a class~A minimizer of~$\Ec$ and thus~$v \in \mathscr{M}^{(x_0)}$ for some~$x_0$. Hence, \eqref{iiunoii} gives that~$v(x)= u^{(0)}(x-x_0)$.

     \appendix

 \section{On the kernel~$K$}\label{kern_ex}
 In this section, we provide three explicit examples of \lq\lq{}admissible\rq\rq{} kernels~$K$, where we define \lq\lq{}admissible\rq\rq{} any kernel~$K:\R^n \to [0,+\infty]$ satisfying~\eqref{krn_symm}, \eqref{main_ellipt}  and~\eqref{nuovissima}. 
 
\begin{example}{\rm{
The kernel of the fractional Laplacian
$$
K(x) := | x |^{-( n+2s)}
$$
is admissible. 
Indeed, it obviously satisfies~\eqref{krn_symm} and~\eqref{main_ellipt}.

We now check that it also satisfies~\eqref{nuovissima}. 
For this, let~$\sigma_j\nearrow1$ as~$j \to +\infty$ and~$\epsilon \in (0,1)$. Then,
\begin{eqnarray*}
\left(\sup_{x \in \R^n \setminus\{0\}} \frac{K(\sigma_j x)}{K(x)} -1\right) \frac{1}{(1-\sigma_j)^{1-\epsilon}} &
= & \frac{\sigma_j^{-(n+2s)}-1}{(1-\sigma_j)^{1-\epsilon}} \\ 
&=& \frac{(n+2s)(1-\sigma_j)+O(| 1 -\sigma_j|^2)}{(1-\sigma_j)^{1-\epsilon}}.
\end{eqnarray*}
Taking the limit as~$j\to+\infty$, we thereby see that~\eqref{nuovissima} holds true.
}}
\end{example}

We mention here that condition~\eqref{main_ellipt} is very general and allows us to consider a great variety of translation invariant kernels, only locally comparable to that of the fractional Laplacian. For instance, a kernel satisfying~\eqref{main_ellipt} is
\begin{equation}\label{hppppf}
K(x) := \mathds{1}_{B_{r_0}}(x) \frac{a(x)}{| x |^{n+2s}},
\end{equation}
with~$r_0>0$ and~$a$ bounded and bounded away from zero.
Kernels of this form
have been widely considered in the literature (see for instance~\cite{KKL16} and the references therein).

Nevertheless, a kernel as in~\eqref{hppppf} does not satisfy~\eqref{nuovissima}. To show this, we consider~$\sigma_j \nearrow 1$ such that~$\sigma_j r_0 < r_0 < r_0/\sigma_j$. Also, we pick~$\bar{x} \in \partial B_{r_0/\sqrt{\sigma_j}}$ and consider a sequence of points~$x_k \in  B_{r_0/\sqrt{\sigma_j}}$ converging to~$\bar{x}$. In particular, ${x_k \in B_{r_0/\sqrt{\sigma_j}} \setminus B_{r_0}}$ for any~$k$ large enough, while~$\sigma_j x_k \in B_{\sqrt{\sigma_j}r_0} \Subset B_{r_0}$. Then, we see that
\[
\sup_{x \in \R^n} \frac{K(\sigma_j x)}{K(x)} \geq \frac{\displaystyle\inf_{x\in \R^n} a(x)} {\displaystyle\sup_{x\in\R^n}a(x)}\sigma_j^{-(n+2s)} \lim_{k \to +\infty} \frac{ \mathds{1}_{B_{r_0}}(\sigma_j x_k)}{\mathds{1}_{B_{r_0}}(x_k)} = +\infty,
\]
which entails that~\eqref{nuovissima} fails to hold.

\begin{example}{\rm
An admissible kernel that exploits the generality of~\eqref{main_ellipt} is the following:
\begin{equation}\label{khvfvvcd}
K(x): = \begin{cases}
| x |^{-(n+2s)} & \mbox{if } | x | < \rho, \\
2 \rho^{(\theta-1)(n+2s)}| x |^{-\theta(n+2s)} & \mbox{if } | x| \geq \rho
\end{cases}
\end{equation}
for some~$\theta>1$ and~$\rho>0$. This kernel satisfies~\eqref{krn_symm} and~\eqref{main_ellipt}. 

In order to show~\eqref{nuovissima}, we can reason as follows. Let~$\sigma_j \nearrow 1$ and~$\epsilon \in (0,1)$. 
We claim that, for any~$x\in\R^n$,
\begin{equation}\label{guefiowdj5434ur7hufndcssfbry54t854yi}
\frac{K(\sigma_j x)}{K(x)}\le \sigma_j^{-\theta(n+2s)}.
\end{equation}
Indeed, if~$ x  \in{B_{\rho}}$, it holds that
\begin{equation*}
\frac{K(\sigma_j x)}{K(x)}= \sigma_j^{-(n+2s)} \leq \sigma_j^{-\theta(n+2s)}.
\end{equation*}
Moreover, if~$ x \in B_{\rho /\sigma_j} \setminus {B_{\rho}}$, we have that
\begin{equation*}
\begin{split}&
\frac{K(\sigma_j x)}{K(x)}  = \frac{\sigma_j^{-(n+2s)}| x |^{-(n+2s)}}{2 \rho^{(\theta-1)(n+2s)}| x|^{-\theta(n+2s)}} = \frac{\sigma_j^{-(n+2s)}}{2} \left(\frac{| x |}{\rho}\right)^{(\theta-1)(n+2s)} \\
&\qquad\qquad \leq \frac{\sigma_j^{-\theta(n+2s)}}{2} \leq\sigma_j^{-\theta(n+2s)}.
\end{split}
\end{equation*}
Furthermore, it is immediate to check that, for any~$ x \in \R^n \setminus B_{\rho /\sigma_j}$ it holds
\[\frac{K(\sigma_j x)}{K(x)}  =\sigma_j^{-\theta(n+2s)}.\]
Gathering these pieces of information, we obtain~\eqref{guefiowdj5434ur7hufndcssfbry54t854yi}.

{F}rom~\eqref{guefiowdj5434ur7hufndcssfbry54t854yi} we thus conclude that
\[
\begin{split}
\left(\sup_{x \in \R^n \setminus\{0\}} \frac{K(\sigma_j x)}{K(x)} -1\right) \frac{1}{(1-\sigma_j)^{1-\epsilon}}&
\leq \frac{\sigma_j^{-\theta(n+2s)}-1}{(1-\sigma_j)^{1-\epsilon}} \\&= \frac{\theta(n+2s)(1-\sigma_j)+O(| 1 -\sigma_j|^2)}{(1-\sigma_j)^{1-\epsilon}}.
\end{split}
\]
Passing to the limit as~$j \to +\infty$ shows~\eqref{nuovissima}.
}
\end{example}

\begin{example}{\rm
Given~$\tau \in \R$ and~$\zeta \geq1$, we define the function
\[ g(x):=e^{-| x|^2}\cos(\tau | x| )+\zeta. \]
Also, taking~$K$ as in~\eqref{khvfvvcd}, we define the kernel
$$
\widetilde{K}(x):= C K(x) g(x),
$$
for some~$C>0$.

The kernel~$\widetilde{K}$ is admissible.
Indeed, it clearly satisfies~\eqref{krn_symm} and~\eqref{main_ellipt}, thus we only focus on showing~\eqref{nuovissima}. To this aim, let~$\sigma_j \nearrow 1$ and~$\epsilon \in (0,1)$. In light of~\eqref{guefiowdj5434ur7hufndcssfbry54t854yi}, we have see that,
for any~$x \in \R^n \setminus \{0\}$,
\begin{equation}\label{0j9h0h121}
\begin{split}
\frac{\widetilde{K}(\sigma_jx)}{\widetilde{K}(x)}-1 &\leq \frac{\sigma_j^{-\theta(n+2s)} g(\sigma_jx)-g(x)}{g(x)}\\&=  \frac{(\sigma_j^{-\theta(n+2s)} -1) g(\sigma_jx)}{g(x)}+\frac{g(\sigma_jx)-g(x)}{g(x)} \\
&\leq \frac{\displaystyle\sup_{x \in\R^n} g(x)}{\displaystyle\inf_{x\in\R^n} g(x)} \big(\sigma_j^{-\theta(n+2s)} -1\big) + \frac{|g(\sigma_jx)-g(x)|}{\displaystyle \inf_{x\in\R^n} g(x)}.
\end{split}
\end{equation}

Now, we claim that, for any~$x \in \R^n$,
\begin{equation}\label{845t9y79egriev}
|g(\sigma_jx)-g(x) |\leq C (1-\sigma_j),
\end{equation}
for some~$C>0$.

In order to show~\eqref{845t9y79egriev}, we observe that,
if~$j$ is sufficiently large,
\begin{equation}\label{bvfjf8y7776}
\begin{split}
|g(\sigma_jx)-g(x)|&=\left| e^{-\sigma_j^2 | x |^2}\cos(\tau \sigma_j | x |)- e^{-| x|^2}\cos(\tau | x| )\right| \\
&\le\left|  e^{-\sigma_j^2 | x |^2} \left( \cos(\tau \sigma_j | x |) -\cos(\tau  | x |)  \right) \right|+ \left|\cos(\tau | x |) \left(  e^{-\sigma_j^2 | x |^2} -e^{-| x|^2}\right)\right| \\
&\leq  \tau e^{- | x |^2 /2} | x |  (1-\sigma_j) + \left|  e^{-\sigma_j^2 | x |^2} -e^{-|x|^2}\right| \\
&\leq C (1-\sigma_j)+e^{-\sigma_j^2 | x|^2} -e^{-| x|^2},
\end{split}
\end{equation}
for some~$C>0$.

Now, we define the function~$f(x):= e^{-| x |^2} | x |^2$ and we point out that, if~$t \in \R \setminus \{0\}$,
$$
e^{-t^2 | x |^2}t= \frac{f(tx)}{t | x|^2} \leq \frac{\| f \|_{L^{\infty}(\R^n)}}{t | x |^2}.
$$
Thus, exploiting the Fundamental Theorem of Calculus and a Taylor expansion of the logarithm near~$1$, we get that
\begin{equation*}
\begin{split}&
 e^{-\sigma_j^2 | x |^2} -e^{-| x|^2} = \int_1^{\sigma_j} \frac{d}{dt}e^{-t^2 |x|^2} \, dt  = 2 |x|^2 \int_{\sigma_j}^1 e^{-t^2 |x|^2}t \, dt \leq 2 \|f\|_{L^{\infty}(\R^n)} \int_{\sigma_j}^1 \frac{dt}{t} \\
&\qquad\qquad= 2 \|f\|_{L^{\infty}(\R^n)} \ln(\sigma_j^{-1}) \leq C (1-\sigma_j),
\end{split}
\end{equation*}
for some~$C>0$. Combining this with~\eqref{bvfjf8y7776} yields~\eqref{845t9y79egriev}.

As a consequence of~\eqref{0j9h0h121} and~\eqref{845t9y79egriev}, we obtain that
\begin{equation*}
\begin{split}
\sup_{x \in \R^n \setminus \{0\}} \frac{\widetilde{K}(\sigma_jx)}{\widetilde{K}(x)}-1 &\leq \frac{\displaystyle\sup_{x \in\R^n} g(x)}{\displaystyle\inf_{x\in\R^n} g(x)} \big(\sigma_j^{-\theta(n+2s)} -1\big) + \frac{C (1-\sigma_j)}{\displaystyle\inf_{x\in\R^n} g(x)} \leq C (1-\sigma_j),
\end{split}
\end{equation*}
up to relabeling~$C$. 

This entails that~\eqref{nuovissima} holds true, as desired.}
\end{example}

\end{document}